\documentclass[review]{elsarticle}
\usepackage{comment}
\usepackage{color}
\usepackage{booktabs}
\usepackage{multirow} 
\usepackage{amsmath}
\usepackage{amssymb}
\usepackage{upgreek}
\usepackage{lineno}
\usepackage{hyperref}
\usepackage{amsthm}
\modulolinenumbers[5]
\usepackage{pifont}
\usepackage{pgfgantt}
\usepackage{rotating}
\usepackage{algorithmic}%
\newtheorem{theorem}{Theorem} 

\newtheorem{proposition}[theorem]{Proposition}

\usepackage[a4paper, left=2cm, right=2cm, top=2cm, bottom=2cm]{geometry} 
\usepackage{mathtools}  
\usepackage{mathrsfs}   
\usepackage[linesnumbered,ruled]{algorithm2e} 
\SetKwComment{Comment}{$\triangleright$\ }{} 
\usepackage{rotating}
\usepackage{multicol}
\usepackage{multirow}
\usepackage{graphicx}
\usepackage{subfig}
\SetKw{Break}{break}
\usepackage{pdflscape}
\usepackage{afterpage}
\usepackage{capt-of}
\usepackage{soul}

\usepackage{numcompress}\biboptions{authoryear}
\usepackage{enumitem}

\usepackage{adjustbox}

\begin{document}

\begin{frontmatter}

\title{Scheduling on parallel machines with a common server in charge of loading and unloading operations}

\author[1]{Abdelhak Elidrissi\corref{mycorrespondingauthor}}
\cortext[mycorrespondingauthor]{Corresponding author}
\ead{abdelhak.elidrissi@uir.ac.ma}

\author[2]{Rachid Benmansour}

\author[3]{Keramat Hasani}

\author[3]{Frank Werner}

\address[1]{Rabat Business School, International University of Rabat, Rabat, Morocco}
\address[2]{Institut National de Statistique et d'Economie Appliquée (INSEA), Rabat, Morocco}
\address[3]{National University of Singapore, Singapore}
\address[3]{Fakult\"{a}t f\"{u}r Mathematik, Otto-von-Guericke-Universit\"{a}t Magdeburg,
PSF 4120, 39016 Magdeburg, Germany}

\begin{abstract}
This paper addresses the scheduling problem on two identical parallel machines with a single server in charge of loading and unloading operations of jobs. Each job has to be loaded by the server before being processed on one of the two machines and unloaded by the same server after its processing. No delay is allowed between loading and processing, and between processing and unloading. The objective function involves the minimization of the makespan.   This problem referred to
as $P2,S1|s_j, t_j|C_{max}$ generalizes the classical parallel machine scheduling problem with a single server which performs only the loading (i.e., setup) operation of each job. For this $\mathcal{NP}$-hard problem, no solution algorithm was proposed in the literature. Therefore, we present two mixed-integer linear programming (MILP) formulations, one with completion-time variables along with two valid inequalities and one with time-indexed variables. In addition, we propose some polynomial-time solvable cases and a tight theoretical lower bound. In addition, we show that the minimization of the makespan is equivalent to the minimization of the total idle times on the machines. To solve 
large-sized instances of the problem, an efficient General Variable Neighborhood Search (GVNS) metaheuristic with two mechanisms for finding an initial solution is designed. The GVNS is evaluated by comparing its performance with the results provided by the MILPs and another metaheuristic. The results show that the average percentage deviation from the theoretical lower-bound of GVNS is within 0.642\%. Some managerial insights are presented and our results are compared with the related literature. 

\end{abstract}



\end{frontmatter}



\sloppy
\section{Introduction and literature review} \label{sec1}

Parallel machine scheduling problem with a single server (PMSSS problem) has received much attention over the last two decades. In the PMSSS problem, the server is in charge of the setup operation of the jobs. This setup operation can be defined as the time required to prepare the necessary resource (e.g., machines, people) to perform a task (e.g., job, operation) (\cite{allahverdi2008significance, bektur2019mathematical, hamzadayi2017modeling, kim2012mip}). Indeed, in the classical parallel machine scheduling problem, it is assumed that the jobs are to be executed without prior setup. However, this assumption is not always satisfied in practice, where industrial systems are more flexible (e.g., flexible manufacturing system). Under certain conditions, this assumption can lead also to a shortfall and/or waste of time. In addition, in the PMSSS problem, it is assumed that after the loading and processing operations, the job is automatically removed from the machine and no unloading operation is considered.

The PMSSS problem has many industrial applications.  In network computing, the network server sets up the workstations by loading the required software. In production applications, the setting up of machines involves the simultaneous use of a common resource which might be a robot or a human operator attending each setup (\cite{bektur2019mathematical}). In automated material handling systems, robotic cells or in the semiconductor industry (\cite{kim2012mip}), it is necessary to share a common server, for example a robot, by a number of machines to carry out the machine setups. Then the job processing is executed automatically and independently by the individual machines. 

The literature regarding the problem PMSSS can be classified into four main categories. $i)$ the first category, where only one single server is used for the setup operations \citep{kravchenko1997parallel,kravchenko1998scheduling,
hasani2014block,kim2012mip,hamzadayi2017modeling,bektur2019mathematical,elidrissi2021mathematical}; $ii)$  the second category, with multiple servers for unloading jobs (without considering the loading operations) \citep{ou2010parallel}; $iii)$ the third category, where two servers are considered, the first server is used for the loading operations and the second one for the unloading operations \citep{jiang2017scheduling,benmansour2021scheduling,
elidrissi2022general}; $iv)$ the last category, where only one single server is used for both loading and unloading operations \citep{xie2012scheduling,hu2013parallel,jiang2014optimal,
jiang2015single,jiang2015online}. Table~\ref{tab-lit} summarizes the papers included into the categories : $ii)$, $iii)$ and $iv)$. \cite{elidrissi2021mathematical} presented a short review of the papers considering the category $i)$.

In this paper, we address the scheduling problem with two identical parallel machines and a single server in charge of loading and unloading operations of jobs. The objective  involves the minimization of the makespan. The static version is considered, where the information about the problem is available before the scheduling starts. Following the standard three-field notation  \citep{graham1979optimization}, the considered problem can be denoted as $P2,S1|s_j,t_j|C_{max}$, where $P2$ represents the two identical parallel machines, $S1$ represents  the single server, $s_j$ is the loading time of job $J_j$, $t_j$ is the unloading time of job $J_j$ and $C_{max}$ is the objective to be minimized (i.e., the makespan). 

For the problem involving several identical servers,
\cite{kravchenko1998scheduling} studied the problem with~$k \geq 2$~servers in order to minimize the makespan.
The authors state that the multiple servers are in charge of only the loading operation of the jobs. In addition, they showed that the problem is unary $\mathcal{NP}$-hard for each $k < m $. Later, \cite{werner2010scheduling} showed that the problem with $k$ servers with an objective function involving the minimization of the makespan is binary $\mathcal{NP}$-hard. In the context of the milk run operations of a logistics company that faces limited unloading docks at the warehouse, \cite{ou2010parallel} studied the problem of scheduling an arbitrary number of identical parallel machines with multiple unloading servers, with an objective function involving the minimization of the total completion time. The authors showed that the shortest processing time (SPT) algorithm has a worst-case bound of 2 and proposed other heuristic algorithms as well as a branch-and-bound algorithm to solve the problem. Later, \cite{jiang2017scheduling} studied the problem  $P2,S2|s_j = t_j = 1|C_{max}$ with unit loading times and unloading times. They showed that the classical list scheduling (LS) and the largest processing time (LPT) heuristics have worst-case ratios of 8/5 and 6/5, respectively. Later, \cite{benmansour2021scheduling} suggested a mathematical programming formulation and a general variable neighborhood search (GVNS) metaheuristic for the general case of the problem $P2,S2|s_j,t_j|C_{max}$ with only two identical parallel machines. Recently, \cite{elidrissi2022general} addressed the problem $P,S2|s_j, t_j|C_{max}$ with an arbitrary number of machines. The authors considered the regular case of the problem, where $ \forall i,j \quad p_i < s_j+p_j+t_j$. They proposed two mathematical programming formulations and three versions of the GVNS metaheuristic with different mechanisms for finding an initial solution.

\begin{sidewaystable}
\caption{ An overview of the approaches in the literature on the identical parallel machine scheduling problem with
a single server or multiple servers involving loading or/and unloading operations.}
\centering
\renewcommand{\tabcolsep}{1.5pt}
{\scriptsize
\renewcommand{\arraystretch}{1.2}
\begin{tabular}{lcccccccccccccccl}
\hline
Publications &   & \multicolumn{3}{l}{Number of machines} &   & \multicolumn{5}{l}{Server (s) constraint} &   & \multicolumn{3}{l}{Methods} &   & Approaches \\
\cline{3-5}\cline{7-11}\cline{13-15}      &   & \multicolumn{1}{c}{$m = 2$} &   & $m \geq2$ &   & \multicolumn{1}{c}{$Sk$ unloading } &   & \multicolumn{1}{l}{$S2$ : loading server} &   & \multicolumn{1}{l}{$S1$ for loading } &   & \multicolumn{1}{c}{Exact} & \multicolumn{1}{c}{Approximate} & \multicolumn{1}{c}{Worst-case} &   &  \\
      &   &   &   &   &   & \multicolumn{1}{c}{ servers} &   & \multicolumn{1}{c}{and unloading server} &   & \multicolumn{1}{c}{ and unloading } &   &   &   & \multicolumn{1}{c}{analysis} &   &  \\
\hline
\cite{ou2010parallel} &   &   &   & \checkmark &   & \checkmark &   &   &   &   &   &   &   &  \checkmark &   & Worst-case analysis  \\
 &   &   &   &  &   &  &   &   &   &   &   &   &   &   &   &   \\
\cite{xie2012scheduling}  &   & \checkmark &   &   &   &   &   &   &   & \checkmark &   &   &   & \checkmark  &   & Worst-case analysis  \\
 &   &   &   &  &   &  &   &   &   &   &   &   &   &   &   &   \\
\cite{hu2013parallel}  &   &   &   &   &   &   &   &   &   & \checkmark &   &   &   & \checkmark  &   & Worst-case analysis  \\
 &   &   &   &  &   &  &   &   &   &   &   &   &   &   &   &   \\
\cite{jiang2014optimal}  &   & \checkmark &   &   &   &   &   &   &   & \checkmark &   &   &  \checkmark &   &   & O(nlogn) algorithm for the problem \\
 &   &   &   &  &   &  &   &   &   &   &   &   &   &   &   &   \\
\cite{jiang2015single}   &   & \checkmark &   &   &   &   &   &   &   & \checkmark &   &   &   & \checkmark  &   & Worst-case analysis  \\
 &   &   &   &  &   &  &   &   &   &   &   &   &   &   &   &   \\
\cite{jiang2015online}   &   & \checkmark &   &   &   &   &   &   &   & \checkmark &   &   &   & \checkmark  &   & Worst-case analysis for the online version of the problem \\
 &   &   &   &  &   &  &   &   &   &   &   &   &   &   &   &   \\
\cite{jiang2017scheduling} &   & \checkmark &   &   &   &   &   & \checkmark &   &   &   &   &   &  \checkmark &   & Worst-case analysis  \\
 &   &   &   &  &   &  &   &   &   &   &   &   &   &   &   &   \\
\cite{benmansour2021scheduling} &   & \checkmark &   &   &   &   &   & \checkmark &   &   &   &  \checkmark & \checkmark  &   &   & MIP formulation and GVNS algorithm \\
 &   &   &   &  &   &  &   &   &   &   &   &   &   &   &   &   \\
\cite{elidrissi2021mathematical} &   &   &   & \checkmark &   &   &   & \checkmark &   &   &   &  \checkmark & \checkmark  &   &   & MIP formulations and GVNS algorithm for the regular case  \\
 &   &   &   &  &   &  &   &   &   &   &   &   &   &   &   &   \\
\textbf{This paper.}  &   & \checkmark &   &   &   &   &   &   &   & \checkmark &   & \checkmark   &  \checkmark  &   &   & MILP formulations, polynomial solvable cases, lower bounds \\
   &   &  &   &   &   &   &   &   &   & &   &    &   &   &   & and metaheuristics  \\
\hline
\end{tabular}
}
\label{tab-lit}
\end{sidewaystable}

In the scheduling literature, the problem $P2,S1|s_j, t_j|C_{max}$ involving both loading and unloading operations has attracted the attention of the researchers. \cite{xie2012scheduling} addressed the problem $P2,S1|s_j, t_j|C_{max}$. They derived some optimal properties, and they showed that the LPT heuristic generates a tight worst-case bound of $3/2 - 1/2m$. \cite{hu2013parallel} considered the classical algorithms LS and LPT for the problem $P2,S1|s_j, t_j|C_{max}$ where $s_j=t_j=1$. They showed that LS and LPT  generate  tight worst case ratios of  12/7 and 4/3, respectively. \cite{jiang2014optimal} addressed the problem $P2,S1|pmpt,s_j= t_j=1|C_{max}$ with preemption, and unit loading and unloading times. They presented  an $O(n \log n)$ solution algorithm for the problem. Later, \cite{jiang2015online} considered the online version of the problem $P2,S1|s_j, t_j|C_{max}$. The authors suggested an algorithm with a competitive ratio of $5/3$. In another paper, \cite{jiang2015single} studied the problem $P2,S1|s_j=t_j=1|C_{max}$ with unit loading and unloading times. They showed that the LS and LPT algorithms have tight worst-case ratios of 12/7 and 4/3, respectively. As far as we know, no solution methods are proposed in the literature for the problem $P2,S1|s_j, t_j|C_{max}$. A goal of our paper aims at bridging this gap. We also compare our results with the literature regarding the problem $P2,S2|s_j, t_j|C_{max}$ involving a dedicated loading server and a dedicated unloading server.

The main contributions of this paper are as follows:

\begin{itemize}
\item To the best of our knowledge, no study proposes solution methods for the parallel machine scheduling problem with a single server in charge of loading and unloading operations of the jobs. Our study generalizes the classical parallel machine scheduling problem with a single server by considering the unloading operations. 

\item We present for the first time in the literature two  mixed-integer linear programming formulations for the problem $P2,S1|s_j,t_j|C_{max}$. The first one is based on completion-time variables and the second one is based on time-indexed variables. Two valid inequalities are suggested to enhance the completion-time variables formulation. 

\item We show that for the problem $P2,S1|s_j,t_j|C_{max}$, the minimization of the makespan is equivalent to the minimization of the idle times of the machines. In addition, three polynomial-time solvable cases and  a tight theoretical lower bound are proposed. 

\item We design an efficient GVNS algorithm with two mechanisms for finding an initial solution to solve large-sized instances of the problem. We provide a new data set and examine the solution quality of different problem instances. The performance of GVNS is compared with a greedy randomized adaptive search procedures metaheuristic. 

\item Some managerial insights are presented, and our results are compared with the literature regarding the problem $P2,S2|s_j,t_j|C_{max}$ involving two dedicated servers (one for the loading operations and one for the unloading operations). 
\end{itemize}

The rest of this paper is organized as follows. Section~\ref{sec2} presents a formal description of the problem. In Section~\ref{sec3}, we present two MILP formulations along with two valid inequalities for the addressed problem. A machines idle-time property, polynomial-time solvable cases and a lower bound are presented in Section~\ref{sec4}. In Section~\ref{sec5},  an iterative improvement procedure and two metaheuristics are presented. Numerical experiments are discussed in Section~\ref{sec6}. Section~\ref{sec7} presents some managerial insights and a comparison with the literature.  Finally, concluding remarks are given in Section~\ref{sec8}.


\section{Definition of the problem and notation} \label{sec2}

The aim of this section is to give a detailed description of the problem $P2,S1|s_j, t_j|C_{max}$. We are given a set  $M = \{1,2\}$ of two identical parallel machines that are available to process a set $\mathcal{N}_{1} = \{J_1,\ldots,J_n\}$ of $n$ independent jobs. Each job $J_j \in N$ is available at the beginning of the scheduling period and has a known integer processing time $p_j > 0$. Before its
processing, each job $J_j$ has to be loaded by the loading server, and the loading time is $s_j > 0$. After its processing, a job has to be unloaded from the machine by the unloading server, and the unloading time is $t_j > 0$. The processing operation starts immediately after the end of the loading operation, and the unloading operation starts immediately after the end of the processing operation. During the loading (resp. unloading) operation, both the machine and the loading server (respectively unloading server) are occupied and after loading (resp. unloading) a job, the loading server (resp. unloading server) becomes available for loading (resp. unloading) the
next job.  Furthermore, there is no precedence constraints among jobs, and preemption is not allowed. The objective is to find a feasible schedule that minimizes the makespan.

The following notation is used to define this problem:

\emph{Sets}

\begin{itemize}
	\item $n$: number of jobs
	\item $M= \{1,2\}$: set of two machines
	\item $\mathcal{N}_{1} = \{J_1,\ldots,J_n\} $:  set of jobs to be processed on the machines
	\item $\mathcal{N}_{2} = \{J_{n+1},\ldots,J_{2n}\}$:  set of loading dummy jobs  to be processed on the server
	\item $\mathcal{N}_{3} = \{J_{2n+1},\ldots,J_{3n}\}$:  set of unloading dummy jobs  to be processed on the server
	\item $\mathcal{N} = \mathcal{N}_{1} \cup \mathcal{N}_{2} \cup \mathcal{N}_{3} $

\end{itemize}

\emph{Parameters}

\begin{itemize}
\item $s_j$: loading time of job $J_j$
\item $p_j$: processing time of job $J_j$
\item $t_j$: unloading time of job $J_j$
\item $A_j$: length of job $J_j$ ($A_j = s_j+p_j+t_j $)
\item $B$: a large positive integer
\end{itemize}

\emph{Continuous decision variables}

\begin{itemize}
	\item $C_j$: completion time of job $J_j$
	\item $C_{j+n}$: completion time of the loading dummy job $J_{j+n}$
	\item $C_{j+2n}$: completion time of the unloading dummy job $J_{j+2n}$

\end{itemize}

For the purpose of modeling, we adopt the following notations, where  the parameter $\rho$ 
represents the duration of the jobs and dummy jobs, either on the machine or on the server.

$$ \rho_{j} = \left \{
\begin{array}{l c l}
A_j  &  & \forall j \in \mathcal{N}_{1} \\
s_{j-n} & & \forall j \in \mathcal{N}_{2} \\
t_{j-2n}  & & \forall j \in \mathcal{N}_{3} \\
\end{array}
\right .$$


\section{Mixed-integer linear programming (MILP) formulations}  \label{sec3}

MILP formulations are well studied in the literature for
different scheduling problems, such as a single machine, parallel machines, a flow shop, a job
shop and an open shop, etc. (see \cite{michael2018scheduling}). The main MILP formulations for scheduling problems can be classified according to the nature of the decision variables  \citep{unlu2010evaluation,kramer2021mathematical,
elidrissi2021mathematical}. In this section, we derive two MILP formulations based on completion-time variables and time-indexed variables for the problem $P2,S1|s_j,t_j|C_{max}$. Before presenting the two MILP formulations, we define our suggested dummy-job representation. 

\subsection{Dummy-job representation}

A dummy-job representation (see \cite{elidrissi2021mathematical}) is used in this paper to simplify the problem and make it possible to model the problem as a relatively neat MILP model. Indeed, in our modeling, we consider the single server as the $(m + 1)^{th}$ (i.e., third) machine. Each time the server is used to load (resp. unload)  job $J_j \in \mathcal{N}_1 $ on machine $k \in M$, then a dummy job $J_{j + n} \in \mathcal{N}_{2}$ (resp. $J_{j + 2n} \in \mathcal{N}_{3}$) is processed on the dummy machine $(m + 1)$ at the same time. This dummy job $J_{i + n}$ (resp. $J_{i + 2n}$) has a processing time equal to the loading time (resp. unloading time) of the job $J_j$ (i.e., $s_{j}=p_{j+n}$ and $t_{j}=p_{j+2n}$ $ \forall j \in \{1,\ldots,n\}$) (see~Figure~\ref{fig:my_label})~. To define the MILP formulations, we adopt the dummy-job representation. 

\begin{figure}[h!]
	\centering
	\includegraphics[width=14cm,height=4.5cm]{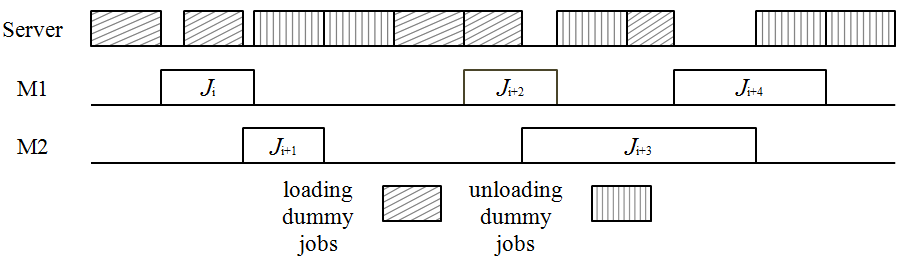}
	\caption{\label{fig:Capture2}Dummy-job representation.}
	\label{fig:my_label}
\end{figure}

\subsection{Formulation 1: Completion-time variables}

In this section, we propose a completion-time variables formulation $(CF)$ for the problem $P2,S1|s_j,t_j|C_{max}$.
A completion-times variables or disjunctive formulation (see \cite{balas1985facial})  has been widely used to model different scheduling problems  \citep{baker2010solving,keha2009mixed,elidrissi2022general}. In our formulation, we use the following decision variables :

Binary decision variables :\\

$x_{i,k} = \left \{
\begin{array}{r l l}
1  &$ if job $J_i$ is processed on machine $k$ $&\\
0  &$otherwise$ & \\
\end{array}
\right .$

$z_{i,j} = \left \{
\begin{array}{r l l}
1  &$ if job $J_i$ is processed before job $J_j$ $&\\
0  &$otherwise$ & \\
\end{array}
\right .$

$y_{i+n,j+n} = \left \{
\begin{array}{r l l}
1  &$ if the loading dummy job $J_{i+n}$ is processed before the loading dummy  $ \\
  &$  job $J_{j+n}$  on the server $ \\
0  &$otherwise$ & \\
\end{array}
\right .$

$y_{i+2n,j+2n} = \left \{
\begin{array}{r l l}
1  &$ if the  unloading dummy job $J_{i+2n}$ is processed before the unloading dummy  $&\\
  &$ job $J_{j+2n}$ on the server $&\\
0  &$otherwise$ & \\
\end{array}
\right .$

The objective function (\ref{ctv_1}) indicates that the makespan (i.e., the completion time of the last job that finishes its processing on the machines) is to be minimized. Constraint set (\ref{ctv_2}) represents the restriction that the makespan of an optimal schedule is greater than or equal to the completion time of the last executed job. Constraint set (\ref{ctv_3}) states that each job must be
processed on exactly one machine. Constraint set (\ref{ctv_4}) ensures that the completion time of each
job is at least greater than or equal to the sum of the loading, the unloading and the processing times of this job.  In addition,  the completion time of each
loading dummy job (resp. unloading dummy job) is at least greater than or equal to its loading time (resp. unloading time). Constraint sets (\ref{ctv_5}) and (\ref{ctv_6}) indicate that no two jobs $J_i$
and $J_j$, scheduled on the same machine (i.e., $x_{i,k} = x_{j,k} = 1$), can overlap in time. Constraint sets (\ref{ctv_7}) and (\ref{ctv_8}) state that no two dummy jobs $J_{j+n}$ (resp. $J_{j+2n}$)  and $J_{i+n}$ (resp. $J_{i+2n}$), scheduled on the single server can overlap in time. 

\begin{eqnarray}
(CF)&\min& C_{max} \label{ctv_1} \\
&s.t.&  C_{max} \geq C_{i} \quad \forall i \in \mathcal{N}_{1} \label{ctv_2}\\
&& \sum_{k \in M}^{} x_{i,k} = 1  \quad \forall i \in \mathcal{N}_{1} \label{ctv_3}\\
&& C_{i} \geq \rho_{i} \quad \forall i \in \mathcal{N}_1 \cup \mathcal{N}_2 \cup\mathcal{N}_3 \label{ctv_4}\\
&& C_{i} + \rho_{j} \leq C_{j} + B(3- x_{i,k}- x_{j,k}- z_{i,j}) \quad  \forall i,j \in N_{1}, i \ne j \label{ctv_5}\\
&& z_{i,j} + z_{j,i} = 1 \quad \forall i,j \in \mathcal{N}_{1}, i \ne j \label{ctv_6}\\
&& C_{i} + \rho_{j} \leq C_{j}  + B(1- y_{i,j}) \quad  \forall i,j \in \mathcal{N}_{2} \cup \mathcal{N}_{3}, i \ne j \label{ctv_7}\\
&& y_{i,j} + y_{j,i} = 1 \quad \forall i,j \in \mathcal{N}_{2} \cup N_{3}, i \ne j \label{ctv_8}\\
&& C_{i}   = C_{i+n} + \rho_{i} - \rho_{i+n} \quad  \forall i \in \mathcal{N}_{1}  \label{ctv_9}\\
&& C_{i}  = C_{i+2n}  \quad  \forall i \in \mathcal{N}_{1}  \label{ctv_10}\\
&& z_{i,j} \in \{0,1\} \quad  \forall i \in \mathcal{N}_{1}  \label{ctv_11} \\
&& x_{i,k} \in \{0,1\} \quad  \forall i \in \mathcal{N}_{1}, \forall k \in M \label{ctv_12} \\
&& y_{i,j} \in \{0,1\} \quad  \forall i,j \in \mathcal{N}_{2} \cup \mathcal{N}_{3} \label{ctv_13}
\end{eqnarray}

Constraints (\ref{ctv_9}) calculate the completion time of each job $J_i$. $C_i$ is equal
to the completion time of the loading operation, $C_{i+n}$, plus the processing time and the unloading time of the same job (i.e., $\rho_i - \rho_{i+n}$). Finally, the completion time of the job $J_i$ is equal to the completion time
of the unloading operation of the same job (\ref{ctv_10}). Constraint sets (\ref{ctv_11}) - (\ref{ctv_13}) define the variables $z_{i,j}$, $x_{i,k}$ and $y_{i,j}$ as binary ones. 

\subsection{Strengthening the completion-time variables formulation}
\label{valin}

We present here two valid inequalities to reduce the
time required to solve problem $P2,S1|s_j,t_j|C_{max}$ by the $CF$ formulation.

\begin{proposition}
	\label{va1}
	The following constraints are valid for CF formulation.
	\begin{align}
	&C_{max} \geq \sum_{j \in \mathcal{N}_{1} } A_jx_{j,k}& \forall k \in M \label{val_1}
	\end{align} 
	
\end{proposition}

\begin{proof}

$\sum_{j=1}^{n}~A_jx_{j,k}$ represents the total work load time of the machine $k$ (idle times are not counted). 

It is obvious to see that $ C_{max} \geq \sum_{j \in \mathcal{N}_{1}}^{}~A_jx_{j,k}$. Hence, inequalities (\ref{val_1}) hold.

\end{proof}

Since the two machines are identical, Constraints~(\ref{val_2}) break the symmetry among the machines.

\begin{proposition}
\label{va2}
The following constraints are valid for the $CF$ formulation.
\begin{align}
& \sum_{j \in \mathcal{N}_{1}, j < i}^{} x_{j,k-1} \geq x_{i,k} & \forall i \in \mathcal{N}_{1}, \forall k \in M \setminus \{1\} \label{val_2}
\end{align} 
	
\end{proposition}

Note that we refer to Eq. (\ref{ctv_1})-(\ref{ctv_13}) as $CF$ and by considering the set of constraints Eq. (\ref{ctv_1})-(\ref{val_2}) as $CF^{+}$. A computational comparison between $CF$ and $CF_1^{+}$ is conducted in Section~\ref{sec6}.

\subsection{Formulation 2 : time-indexed variables}
In this section, we propose a time-indexed variables formulation $(TIF)$ for the problem $P2,S1|s_j,t_j|C_{max}$.
A time-indexed variables formulation was introduced by~\cite{sousa1992time} for the non-preemptive single machine scheduling problem. It has been used to model different scheduling problems (see \citep{keha2009mixed,baker2010solving,unlu2010evaluation}. This formulation is based on a time discretization.  The time is divided into periods $1,2,3,\ldots,T$, where period $t$ starts at time $t-1$ and ends at time $t$. The  horizon $T$ is an important part of the formulation and its size depends on.  Any upper bound ($UB$) can be chosen as $T$.  However, a tighter upper bound is preferable to reduce the problem size as the number of time points is pseudo-polynomial in the size of the input. In our formulation, we choose $T=\sum_{j \in \mathcal{N}_{1}} (A_j)$.

The decision variables are defined as follows:

$x_{\{i,t'\}} = \left \{
\begin{array}{r l l}
1  &$ if job $i$ starts processing at time $t'$ $ &\\
0  & $otherwise$ &\\
\end{array}
\right .$\\

\begin{eqnarray}
(TIF)&\min& C_{max} \label{TIF_0} \\
&s.t.& \sum_{t'=0}^{T- \rho_i } (t'+ \rho_i )x_{\{i,t'\}} \leq C_{max} \quad  \forall i \in \mathcal{N}_{1}  \label{TIF_1}\\ 
&& \sum_{i \in \mathcal{N}_{1} } \sum_{s=max(0,t'- \rho_{i} +1)}^{t'} x_{\{i,s\}} \leq 2 \quad  \forall t' \in [0, T]  \label{TIF_2}\\
&& \sum_{i \in \mathcal{N}_{2} \cup N_{3} } 
\sum_{s=max(0,t'-\rho_{i}+1)}^{t'} x_{\{i,s\}} \leq 1 \quad  \forall t' \in [0, T]  \label{TIF_3}\\
&& \sum_{t'=0}^{T-\rho_i} x_{\{i,t'\}} = 1 \quad  \forall i \in \mathcal{N}_{1} \label{TIF_5}\\
&& \sum_{t'=0}^{T- \rho_i} x_{\{i,t'\}} = 1 \quad  \forall i \in \mathcal{N}_{2} \label{TIF_6}\\
&& \sum_{t'=0}^{T- \rho_{i}} x_{\{i,t'\}} = 1 \quad  \forall i \in \mathcal{N}_{3} \label{TIF_7}\\
&& x_{\{i,t'\}} = x_{\{i+n,t'\}} \quad \forall i \in \mathcal{N}_{1}, \forall t' \in [0, T] \label{TIF_8}\\
&& x_{\{i,t'\}} = x_{\{i+2n,t'+s_{i}+p_{i}\}} \quad \forall i \in \mathcal{N}_{1}, \forall t' \in [0, T-\rho_i] \label{TIF_9}\\
&& x_{\{i,t'\}} \in \{0,1\} \quad \forall i \in \mathcal{N}_1, \forall t' \in [0, T-\rho_i] \label{TIF_10}
\end{eqnarray}

In this formulation, the objective function (\ref{TIF_0}) indicates that the makespan, is to be minimized. Constraint set (\ref{TIF_1}) represents the fact that the makespan of an optimal schedule is greater than or equal to the completion time of all executed jobs, where job $J_i$ that starts its loading operation at time point $t$ (i.e., the job for which~$x_{i,t'}$~=~1) and will finish at time $C_i = t + \rho_i$. The completion time of job $J_i$ is calculated as $C_i = \sum_{t=0}^{T-\rho_{i}} (t' + \rho_{i}) x_{i,t'} $. The set of constraints~(\ref{TIF_2}) specifies that at any given time, at most two jobs can be processed on all machines. Constraints~(\ref{TIF_3}) ensure that at any given time, at most one dummy job (loading dummy job or unloading dummy job) can be processed by the server (i.e., the dummy machine). Constraints (\ref{TIF_5}) express that each job $J_i$ must start at some time point $t'$ in the scheduling horizon, where $t' \leq T-s_i - p_i - t_i$. Constraints (\ref{TIF_6}) state that each loading dummy job must start at some time point $t'$ on the dummy machine in the scheduling horizon, where $t' \leq T-s_i$. Constraints (\ref{TIF_7}) guarantee that each unloading dummy job must start at some time point $t'$ on the dummy machine in the scheduling horizon, where $t' \leq T-u_i$. Constraints (\ref{TIF_8}) express  that the start time of the loading dummy job $J_{i+n}$ on the dummy machine and the start time of the job $J_{i}$ is the same (i.e. $x_{i,t'} = x_{i+n,t'}$). Constraints (\ref{TIF_9}) ensure that the unloading operation of the job $J_{i}$ starts immediately after the end of the processing operation of the same job (i.e. $x_{i,t'} = x_{i+2n,t'+s_i+p_i}$). Finally, constraints~(\ref{TIF_10}) define the feasibility domain of the decision variables. 

\subsection{Enhanced time-indexed formulation}

We show here how to reduce the number of
variables and constraints required by the $TIF$ formulation (\ref{TIF_0})–(\ref{TIF_10}) and therefore, improving its computational behavior. The size of the $TIF$ formulation depends on the time horizon $T$. Thus, a reduction of the length of the  time horizon is necessary. To do so, we fix the value of $T$ to the approximate makespan solution given by the GVNS metaheuristic presented in Section~\ref{GV}. It is clear that the new value of $T$ is less than the upper bound $UB=\sum_{j \in \mathcal{N}_{1}} (A_j) $. A comparative study between these two values of $T$ is conducted in the section on computational results (Section~\ref{GV}). We refer to $TIF$ with the reduced
value of the time horizon $T$ as formulation $TIF^{+}$.

\section{Machines Idle-time property, polynomial-time solvable cases and lower bounds}  \label{sec4}
\subsection{Machines~Idle-time property}

In this section, we show that for the problem $P2,S1|s_j,t_j|C_{max}$, the minimization of the makespan is equivalent to the minimization of the total ~\textit{idle time} of the machines. First, we denote by $\widehat{IT}$ the total ~\textit{idle time} of the machines. The machine \textit{idle time} is the time a machine which has just finished the unloading operation of a job is idle before it starts the loading operation of the next job (we recall that in loading and unloading operations, both the machine and the server are occupied). Indeed, this \textit{idle time} is due to the unavailability of the server. Note that we include in this definition the \textit{Idle-time} on a machine after all of its processing is completed, but before the other machine completes its processing (see \cite{koulamas1996scheduling}). In addition, we denote by $IT_{k}$ the total machine \textit{idle time} in a machine $k$. Therefore, Proposition~\ref{prop0} and Proposition~\ref{prop1} can be derived. 

\begin{proposition}
\label{prop0}
The total idle time of machine $k$ is computed as follows:

\begin{equation}
 IT_{k}= C_{max} - \sum_{j \in \mathcal{N}_1}^{} x_{j,k}A_j \quad \forall k \in M
\end{equation}

\end{proposition}

\begin{proposition}
\label{prop1}

The total ~\textit{idle time} of the machines is equal to:
\begin{equation} 
\widehat{IT} = mC_{max} - \sum_{j \in \mathcal{N}_1}^{} A_j
\end{equation}

\end{proposition}

\begin{proof}

since we have 

$$\sum_{k \in M}^{} x_{j,k} = 1 \quad \forall j \in \mathcal{N}_1  $$ 

$$IT_{k}= C_{max} - \sum_{j \in \mathcal{N}_1}^{} x_{j,k}(s_j+p_j+t_j) \quad \forall k \in M$$

we obtain

\begin{align}
\widehat{IT} & = \sum_{k \in M}^{} IT_k  \nonumber \\
& = \sum_{k \in M}^{}\Big(C_{max} - \sum_{j \in \mathcal{N}_1}^{} x_{j,k}(s_j+p_j+t_j)\Big) \nonumber\\
& = \sum_{k \in M}^{}C_{max} - \sum_{k \in M}^{}\sum_{j \in \mathcal{N}_1}^{} x_{j,k}(s_j+p_j+t_j) \nonumber\\
& = \sum_{k \in M}^{}C_{max} - \sum_{j \in \mathcal{N}_1}^{}\sum_{k \in M}^{} x_{j,k}(s_j+p_j+t_j) \nonumber\\
& = mC_{max} - \sum_{j \in \mathcal{N}_1}^{}(s_j+p_j+t_j) \nonumber\\
& = mC_{max} - \sum_{j \in \mathcal{N}_1}^{}A_j \nonumber
\end{align}

\end{proof}

Therefore, for the problem $P2,S1|s_j,t_j|C_{max}$, the minimization of the makespan is equivalent to the minimization of the total \textit{idle time} of the machines.
\subsection{Polynomial-time solvable cases}  \label{lb}

We now present some polynomial-time solvable cases for the problem $P2,S1|s_j,t_j|C_{max}$. 

\begin{proposition}
\label{propo1}
We consider a set of jobs, where $s_i = p_j \quad \forall i,j \in \mathcal{N}_1$ and 
$ p_i = t_j \quad \forall i,j \in \mathcal{N}_1$. Then all permutations define an optimal schedule. In this case, the optimal makespan is equal to the sum of all loading and unloading times of jobs.
\end{proposition}

\begin{proof}

We assume that the processing time of the job scheduled at position 1 is equal to the loading time of the job scheduled at position 2 and the processing time of the job scheduled at position 2 is equal to the unloading time of the job scheduled at position 1. Then, the job at position 2 will start immediately after the end of the loading operation of the job at position 1 and the unloading operation of the job at position 2 will start immediately after the end of the unloading operation of the job at position 1. Therefore, the completion time of the job at position 2 is equal to $C_{[2]}=C_{[1]}+t_{[2]}$, and the waiting time of the server is equal to 0. Now, if we consider $n$ jobs to be scheduled with  $\forall i,j \in \mathcal{N}_1 \quad s_i = p_j$ and $\forall i,j \in \mathcal{N}_1 \quad p_i = t_j$, then the jobs will alternate on the two machines, and the total waiting time of the server is equal to zero (see Figure~\ref{fig1}). Therefore, in this case all permutations represent an optimal schedule, and the optimal makespan ($C_{max}^{*}$) is equal to the sum of all loading and unloading times (i.e., $ C_{max}^{*} = \sum_{j \in \mathcal{N}_1}^{} (s_j + t_j)$).

\begin{figure}[h!]
	\centering
	\includegraphics[width=14cm,height=4.5cm]{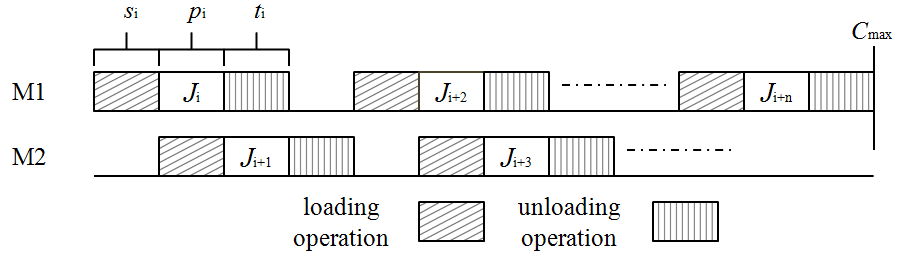}
	\caption{\label{fig:Capture2}Polynomial-time solvable case~1.}
	\label{fig1}
\end{figure}

\end{proof}

\begin{proposition}
\label{prop2}
Consider a set of jobs, where $p_j < s_ i \quad \forall i,j \in \mathcal{N}_1   $. Then all permutations define an optimal solution. In this case, the optimal makespan is equal to the sum of the lengths of all jobs ($ C_{max}^{*} = \sum_{j \in \mathcal{N}_1}^{} A_j $).
\end{proposition}

\begin{proof}

We assume that the processing time of the job scheduled at position 1 is strictly less than the loading time of the job scheduled at position 2. Thus, the job at position 2 
cannot be scheduled immediately after the end of the loading operation of the job scheduled at position 1 (see \ref{fig21}). This is because only one single server is available in the system. Thus, the job at position 2 can be scheduled only after the end of the unloading operation of the job at position 1. In this case, the completion time of the job scheduled  at position 2 is  equal to $C_{[2]}=C_{[1]}+A_{[2]}$. Now, if we consider $n$ jobs with  $\forall i,j \in \mathcal{N}_1 \quad p_j < s_j$, then each job at position $[i]$ can start its leading operation immediately after the end of the unloading operation of the job scheduled at the position $[i-1]$. Therefore, in this case all permutations represent an optimal schedule, and the optimal makespan is equal to the sum of all the lengths of the jobs (i.e., $ C_{max}^{*} = \sum_{j \in \mathcal{N}_1}^{} (A_j)$).

\begin{figure}[h!]
  \centering
  \subfloat[]{\includegraphics[width=0.55\textwidth]{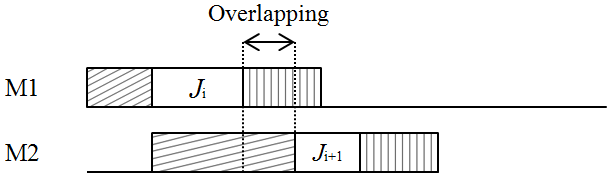}\label{fig21}}
  \hfill
  \subfloat[]{\includegraphics[width=0.55\textwidth]{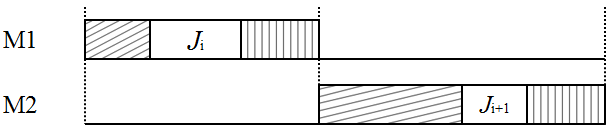}\label{fig2}}
  \caption{Polynomial-time solvable case~2.}
\end{figure}
\end{proof}

\begin{proposition}
\label{prop3}
Consider a set of jobs, where $s_i \leq  p_j \quad \forall i,j \in \mathcal{N}_1$ and 
$ p_i < t_j \quad \forall i,j \in \mathcal{N}_1$. Then all permutations define an optimal solution. In this case, the optimal makespan is equal to the sum of the lengths of all jobs ($ C_{max}^{*} = \sum_{j \in \mathcal{N}_1}^{} A_j $).
\end{proposition}

\begin{proof}

First, we assume that the loading time of the job to be scheduled at position 2 is less than or equal to the processing time of the job scheduled at position 1. Hence, the job to be scheduled at position 2 can start its loading operation in the interval between the end of the loading operation and the start of the unloading operation of the job at position 1. Now, suppose that the processing time of the job at position 2 is strictly less than the unloading time of the job at position 1. Then the job to be scheduled at position 2 can only start its loading operation after the end of the unloading operation of the job at position 1 (see \ref{fig3}). Therefore, all permutations define an optimal schedule, and the optimal makespan is equal to the sum of the lengths of all jobs (i.e., $ C_{max}^{*} = \sum_{j \in \mathcal{N}_1}^{} (A_j)$).

\begin{figure}[h!]
  \centering
  \subfloat[]{\includegraphics[width=0.55\textwidth]{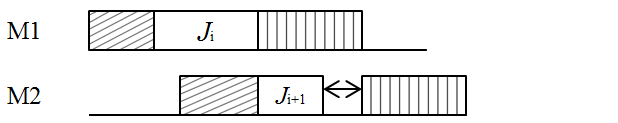}\label{fig31}}
  \hfill
  \subfloat[]{\includegraphics[width=0.55\textwidth]{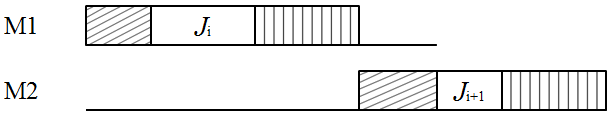}\label{fig3}}
  \caption{Polynomial-time solvable case~3.}
\end{figure}

\end{proof}

\subsection{Lower bound}  \label{lb}

We now introduce a theoretical lower
bound ($LB_T$) on the optimal objective
function value of the problem $P2,S1|s_j,t_j|C_{max}$, namely
$LB_T = \max(LB_1,LB_2)$, where $LB_1$ and $LB_2$  are given in Propositions~\ref{lb1} and~\ref{lb2}, respectively.


\begin{proposition}
\label{lb1}
$$  LB_1 = \frac{\min_{j \in \mathcal{N}_1}{s_j} + \sum_{j \in \mathcal{N}_1}^{} A_j + \min_{j \in \mathcal{N}_1}{t_j}}{2} $$ is a valid lower bound for the problem $P2,S1|s_j,t_j|C_{max}$.
\end{proposition}

\begin{proof}

Let $C_{max}^{*}$ denote the objective function value of an optimal schedule of the problem $P2,S1|s_j, t_j|C_{max}$. If there is no idle time between two consecutive jobs scheduled on the same machine (i.e., the gap between the end of the processing time and the start time of the loading operation of two jobs scheduled on the same machine is equal to zero) in an optimal schedule of the problem $P2,S1|s_j,t_j|C_{max}$,~then $C_{max}^{*}$ will be equal to the sum of all loading times, processing times and unloading times plus $\min_{j \in \mathcal{N}_1}{t_j}$ and $\min_{j \in \mathcal{N}_1}{s_j}$ divided by the number of machines $m=2$. The fact of adding $\min_{j \in \mathcal{N}_1}{t_j}$ and $\min_{j \in \mathcal{N}_1}{s_j}$ with all the loading times, processing times and unloading times will constitute the total load to be executed by the two machines. It is then sufficient to divide this charge by $2$ (i.e., two machines) to obtain the aforementioned lower bound. 
		
\end{proof}

\begin{proposition}
\label{lb2}
$$ LB_2 = \sum_{j \in \mathcal{N}_1}^{} (A_j - p_j) $$ is a valid lower bound for the problem $P2,S1|s_j,t_j|C_{max}$.
\end{proposition}

\begin{proof}
This lower bound can be easily derived from Proposition~\ref{prop1}.
\end{proof}

Therefore:

$$ LB_T = \max\left(\frac{\min_{j \in \mathcal{N}_1}{s_j} + \sum_{j \in \mathcal{N}_1}^{} A_j + \min_{j \in \mathcal{N}_1}{t_j}}{2},  \sum_{j \in \mathcal{N}_1}^{} (A_j - p_j) \right)$$

\section{Solution approaches}  \label{sec5}

This section presents the solution methods to solve large-sized instances of the problem $P2,S1|s_j,t_j|C_{max}$. First, the solution representation and an initial solution based on an iterative improvement procedure in the insertion neighborhood are presented (Section~\ref{Ini}). Then, two  metaheuristics, namely General variable neighborhood search (GVNS) (Section~\ref{GV}) and Greedy randomized adaptive search procedures (GRASP) are proposed (Section~\ref{GR}). The solution approaches are evaluated by extensive computational experiments described in Section~\ref{sec6}.

\subsection{Solution representation and initial solution }\label{Ini}

A solution of the problem $P2,S1|s_j,t_j|C_{max}$ can be represented as a permutation $\Pi= \{\pi_1,\ldots,\pi_k,\ldots,\pi_n\}$ of the job set $\mathcal{N}_{1}$, where $\pi_k$ represents the
job scheduled at the $k^{th}$ position. Any permutation of jobs is feasible if a particular machine and the single server are available simultaneously. A job at the $k^{th}$ position is scheduled as soon as possible on an available machine taking into account the loading and unloading constraints of the single server. Note that in our problem the loading, processing and unloading operations are not separable.

We now present an iterative improvement procedure based on the insertion neighborhood that is used as initial solution for our suggested GVNS (Section~\ref{GV}). This procedure has been successfully used in different scheduling problems (see \cite{ruiz2007simple,ruiz2008iterated}). In each step, a job $\pi_k$ is removed at random from $\Pi$ and then inserted at all possible $n$ positions. The procedure stops if no improvement is found. It is depicted at Algorithm~\ref{ITP}. In Section~\ref{sec6}, we show the benefit of using the iterative improvement procedure as a solution finding mechanism for the GVNS metaheuristic.

\begin{footnotesize}
	\begin{algorithm}[!h]
	\scriptsize
		\SetAlgoLined \DontPrintSemicolon
		\KwData{An instance $\Pi$ of the problem $P2,S1|s_j,t_j|C_{max}$; number of jobs $n$}
		\KwResult{$\Pi$}
		\SetKwFunction{proc}{}
		\SetKwProg{myproc}{procedure}{}{}
		\myproc{\proc{}}{
			\While{ there is no improvement}{
			\For{$i=1$ to $n$}{	
				Remove a job $\pi_k$ at random from $\Pi$ \;
				$\Pi':$ best permutation obtained by   inserting $\pi_k$ in any possible position in $\Pi$\;
				\If{$ C_{max}(\Pi') < C_{max}(\Pi)$}{		
			$\Pi \leftarrow \Pi'$\;}}}}
		\KwRet $\Pi$
		\caption{Iterative Improvement Procedure}
		\label{ITP}
	\end{algorithm}
\end{footnotesize}

\subsection{General variable neighborhood search} \label{GV}

Variable Neighborhood Search (VNS) is a local search based metaheuristic introduced by \cite{mladenovic1997variable}. It aims
to generate a solution that is a local optimum with respect to one or several neighborhood structures. VNS has been successfully applied to different scheduling problems (see  \cite{todosijevic2016nested,chung2019minimizing,elidrissi2022general,maecker2023unrelated}).  It consists of three main steps: $i)$ Shaking step (diversification), $ii)$ Local Search step (intensification), and $iii)$ Change Neighborhood step (Move or Not). 

We notice that VNS has been less used as a solution method for the PMSSS problem. Mainly, the following metaheuristics have been applied to the PMSSS problem:  \emph{simulated annealing} \citep{kim2012mip,hasani2014hybridization,hasani2014simulated,
hamzadayi2016event,hamzadayi2017modeling,bektur2019mathematical}; \emph{genetic algorithm} \citep{abdekhodaee2006scheduling,huang2010parallel,
hasani2014simulated,hamzadayi2017modeling}; \emph{tabu search} \citep{kim2012mip,hasani2014minimising,
bektur2019mathematical,alharkan2019tabu}; \emph{ant colony optimization} \citep{arnaout2017heuristics}; \emph{geometric
particle swarm optimization} \citep{alharkan2019tabu}; \emph{iterative local search} \citep{silva2019exact}; and \emph{worm optimization} \citep{arnaout2021worm}.

In this section, we propose a General VNS (GVNS) which uses a variable neighborhood descent (VND) as a local search \citep{hansen2017variable}. GVNS starts with an initial solution (generated by  the iterative improvement procedure or randomly). Then, the shaking procedure and VND are applied to try to improve the current solution. Finally, this procedure continues until all predefined neighborhoods have been explored and a stopping criterion is met (e.g., a time limit). As far as we know, this is the first study in the literature to propose a GVNS for a parallel machine scheduling problem involving a single server. 

\subsubsection{Neighborhood structures}
\label{NST}

Three commonly used neighborhood structures in the literature are adapted to the problem $P2,S1|s_j,t_j|C_{max}$. The first one is an \textsc{Interchange}-based neighborhood, the second one is  an \textsc{Insert}-based neighborhood and the last one is a \textsc{Reverse}-based neighborhood. These structures have been widely applied to solve different scheduling problems  (see \cite{hasani2014hybridization,hasani2014simulated,alharkan2019tabu}). 

The neighborhood structures are defined as follows:

\begin{itemize}
	\item \textsc{Interchange}($\Pi$) :  It consists of selecting a pair of jobs and exchanging their positions.
We consider a solution of the problem denoted as $\Pi_s= \{\pi_1^s,\ldots,\pi_k^s,\ldots,\pi_n^s\}$, and one of its neighbors $\Pi_t= \{\pi_1^t,\ldots,\pi_k^t,\ldots,\pi_n^t\}$. We fix two different positions ($a \neq b$), and we exchange the jobs scheduled at the two positions (i.e., $\pi_{a}^t = \pi_{b}^s$, $\pi_{b}^s = \pi_{a}^t$ and $\forall x, x \neq (a,b)$ $\pi_x^{t}= \pi_x^{s}$).
	\item \textsc{Reverse}($\Pi$) :  It consists of all solutions obtained from the solution $\Pi$ reversing a subsequence of $\Pi$. More precisely, given two jobs $\pi_a$ and $\pi_b$, we construct a new sequence by first deleting the connection between $\pi_a$ and its successor $\pi_{a+1}$ and the connection between $\pi_j$
	and its successor $\pi_{b+1}$. Next, we connect $\pi_a$ with $\pi_b$ and $\pi_{a+1}$ with $\pi_{b+1}$.
	\item \textsc{Insert}($\Pi$) : It consists of all solutions obtained from the solution $\Pi$ by removing a job and inserting it at another position in the sequence. We consider a solution $\Pi_s= \{\pi_1^s,\ldots,\pi_k^s,\ldots,\pi_n^s\}$, and one of its neighbors  $\Pi_t= \{\pi_1^t,\ldots,\pi_k^t,\ldots,\pi_n^t\}$. If $ a < b$, $\pi_{a}^t = \pi_{b}^s$, $\pi_{a+1}^t = \pi_{a}^s, \ldots, \pi_{b}^t = \pi_{b-1}^s$. Otherwise, if $b < a$, $\pi_{b}^t = \pi_{b+1}^s, \ldots, \pi_{a-1}^t = \pi_{a}^s$, $\pi_{a}^t = \pi_{b}^s$.

\end{itemize}

\subsubsection{Variable neighborhood descent}
\label{VNDh}

We present here the variable neighborhood descent procedure designed for the problem  $P2,S1|s_j,t_j|C_{max}$. It uses the neighborhood structures described in Section~\ref{NST}. VND has a solution which is a local optimum with respect to the \textsc{Interchange}, \textsc{Insert} and \textsc{Reverse} neighborhood structures. The order in which the neighborhoods are
explored and the way how to move from one neighborhood to another one modify the performance of VND. For this problem, after performing preliminary experiments that are not presented here (as they concern minor parameters with respect to the overall approaches), the following settings are proposed. First, we use a basic sequential VND as a strategy to switch from one neighborhood to another one. Second, the following neighborhood structures order is chosen in the VND procedure: $i)$ \textsc{Interchange}, $ii)$ \textsc{Insert} and $iii)$ \textsc{Reverse}. Finally,  the first-improvement strategy (stop generating neighbors as soon as a current solution can be improved in terms of quality) turns out to be  better than the best-improvement strategy (generate all the neighbors and choose the best one). The overall VND pseudo-code is presented in Algorithm~\ref{VND}.

\begin{footnotesize}
	\begin{algorithm}[!h]
	\scriptsize
		\SetAlgoLined \DontPrintSemicolon
		\KwData{A sequence $\Pi$}
		\KwResult{$\Pi$}
		\SetKwFunction{proc}{\textsc{Basic\_Sequential\_VND}}
		\SetKwProg{myproc}{procedure}{}{}
		\myproc{\proc{}}{
			\Repeat{there is no improvement}{
				$l \leftarrow 1$;\;
				\While { $l \leq 3$}
				{   
					\Switch{$l$}{
            \Case{$l=1$}{
                $\Pi' \leftarrow \textsc{interchange}(\Pi)$;\;
               \Break; 
            }
            \Case{$l=2$}{
                $\Pi' \leftarrow \textsc{insert}(\Pi)$;\;
                 \Break;
            }
            \Case{$l=3$}{
                $\Pi' \leftarrow \textsc{reverse}(\Pi)$;\;
                 \Break; 
            }
        }
        $ l \leftarrow l+1 $;\;
       \If{$ C_{max}(\Pi') < C_{max}(\Pi)$}{		
				$\Pi \leftarrow \Pi'$;\;
				$l \leftarrow 1$;\;
		}
		}}}
		\KwRet $\Pi$;
		\caption{Sequential VND}
		\label{VND}
	\end{algorithm}
\end{footnotesize}

\subsubsection{The proposed GVNS and shaking procedure}

To escape from the local optima and have a chance to obtain a global optimum, we propose the following shaking procedure depicted in Algorithm~\ref{shaking}. This procedure consists of generating $k$ random jumps from the current solution $\Pi'$
using the neighborhood structure \textsc{Reverse} (i.e., $k$ random iterations are performed in \textsc{Reverse}). After preliminary experiments, only one neighborhood structure is used (\textsc{Reverse}), and the value of the diversification factor $k$ is chosen as 15, since they offer the best combination between solution quality (i.e., the quality of the obtained solution) and speed (i.e., the time to generate a feasible solution). 

The overall pseudo-code of GVNS as it is designed
to solve the problem $P2,S1|s_j,t_j|C_{max}$ is depicted in Algorithm~\ref{GVNS}. After generating an initial solution (Step 1), a shaking procedure is then applied (Step 2). Once the shaking is performed, VND (Algorithm~\ref{VND}) starts exploring the three proposed neighborhood structures (Step 3). Step~2 and Step~3 until a stopping criterion (CPU)  is  met (the time limit denoted as $t_{max}$).
Since GVNS is a trajectory-based procedure, starting from an initial solution is needed.  In this paper, we compare two variants of GVNS, namely, one starting from the iterative improvement procedure which we denote as GVNS~I, and one starting from a random solution which we denote as GVNS~II. GVNS~I and GVNS~II are both compared in Section~\ref{sec6}. 

\begin{algorithm}[!h]
\scriptsize
	\SetAlgoLined\DontPrintSemicolon
	\KwData{$\Pi$, $k$ : diversification parameter}
	\KwResult{$\Pi$}
	\SetKwFunction{proc}{\textsc{Shaking}}
	\SetKwProg{myproc}{procedure}{}{}
	\myproc{\proc{}}{
		\For{$i=1$ to $k$}{		
			$\Pi'$ : a random solution with respect to the neighborhood structure \textsc{Reverse};\; 
			}
		$\Pi \leftarrow \Pi'$;\;	
		}
	\KwRet $\Pi$\;
	\caption{Shaking}
	\label{shaking}
\end{algorithm} 

\begin{footnotesize}
	\begin{algorithm}[htbp]
	\scriptsize
		\SetAlgoLined \DontPrintSemicolon
		\KwData{A sequence $\Pi$ of the problem $P2,S1|s_j,t_j|C_{max}$, $t_{max}$: time limit, $k_{max}$}
		\KwResult{$\Pi$, $C_{max}(\Pi)$}
		\SetKwFunction{proc}{GVNS}
		\SetKwProg{myproc}{procedure}{}{}
		\myproc{\proc{}}{
			$ \Pi \leftarrow$ \textsc{Initial Solution();} \;
		\Repeat{ $CPU > t_{max}$}
			{	
			$ k \leftarrow 1$;  \;	
			\Repeat{$k > k_{max}$}
			{		
				$\Pi' \leftarrow$ \textsc{Shaking}$(\Pi,k)$;\;
				$\Pi'' \leftarrow$ \textsc{Basic\_Sequential\_VND}($\Pi'$);\;
				\eIf{$C_{max}(\Pi'') < C_{max}(\Pi) $}{		
					$\Pi \leftarrow \Pi''$;\;
					$ k \leftarrow 1$; \;	
				}{
				$ k \leftarrow k + 1$; } \;

		}}}
		\KwRet $\Pi$
		\caption{General VNS for the problem $P2,S1|s_j,t_j|C_{max}$}
		\label{GVNS}
	\end{algorithm}
\end{footnotesize}

\subsection{Greedy randomized adaptive search procedures}
\label{GR}

The Greedy randomized adaptive search procedure (GRASP) is a local search metaheuristic introduced by \citep{feo1995greedy}. It has been suggested to solve different scheduling problems \citep{baez2019hybrid,yepes2020grasp}. Like GVNS, GRASP has two main phases: diversification phase which is based on a greedy randomized construction procedure and an intensification phase based on the use of a local search procedure. Both phases are repeated in every iteration until a stopping criterion is met (e.g., the number of iterations or/and a time limit). In this section, we propose a hybridization of the GRASP metaheuristic with the VND procedure for the problem $P2,S1|s_j,t_j|C_{max}$. Indeed, VND is used as a local search method (as it contributes to a significant improvement of the quality of solutions in the preliminary experiments). The overall pseudo-code of our designed GRASP with VND as a local search is presented in the following (Algorithm~\ref{grasp1}). To the best of our knowledge, our paper is the first one in the literature implementing a GRASP metaheuristic for a variant of the PMSSS problem.

\begin{footnotesize}
	\begin{algorithm}[!h]
	\scriptsize
		\SetAlgoLined \DontPrintSemicolon
		\KwData{A sequence $\Pi$ of the problem $P2,S1|s_j,t_j|C_{max}$, $t_{max}$: time limit}
		\KwResult{$\Pi$}
		\SetKwFunction{proc}{GRASP}
		\SetKwProg{myproc}{procedure}{}{}
		\myproc{\proc{}}{
			\Repeat{$ CPU > t_{max}$}
			{		
				$\Pi' \leftarrow$ \textsc{Greedy\_Randomized\_Construction}$(\Pi)$\;
				$\Pi'' \leftarrow$ \textsc{Basic\_Sequential\_VND}($\Pi'$)\;
				\If{$C_{max}(\Pi'') < C_{max}(\Pi) $}{		
					$\Pi \leftarrow \Pi''$\;
				}	
		}}
		\KwRet $\Pi$
		\caption{Greedy Randomized Adaptive Search Procedures}
		\label{grasp1}
	\end{algorithm}
\end{footnotesize}

\subsubsection{Greedy randomized construction}

The Greedy randomized construction (GRC) procedure of GRASP is presented in Algorithm~\ref{grasp2}. A solution $\Pi'=\{\pi'_1,\ldots,\pi'_k,\ldots,\pi'_n\}$ is generated iteratively. Indeed, at each iteration of the GRC procedure, a new job is added. First, all jobs from the initial list $\Pi$ are initially inserted into the Candidate List ($CL$). Then, the incremental cost associated with the incorporation of a job $\pi_s$ from $CL$ into the solution under construction is calculated. The incremental cost ($IC$) is calculated taking into account the loading and unloading constraints of the single server. Once the $IC$ of all jobs is calculated, we choose the largest and smallest ones, which are denoted by  $\Phi_{max}$ and $\Phi_{min}$, respectively. A Restrict Candidate List ($RCL$) is then created with the best candidate jobs that satisfy the following Inequality~(\ref{GR}).

\begin{equation}
\label{GR}
\Phi(\pi_s) \leq \Phi_{min} + \alpha(\Phi_{max}-\Phi_{min}) \quad \forall \pi_s \in CL 
\end{equation}

The parameter $\alpha$ controls the amounts of greediness and randomness  in the GRC procedure. $\alpha_g$ is generated uniformly  at random from the interval [0, 1] (a purely random construction corresponds to $\alpha = 1$, whereas the greedy construction corresponds to $\alpha = 0$). Hence, a job $\pi_s$ from $RCL$ is selected and scheduled on the first available machine taking into consideration the loading and unloading operations performed by the single server. Finally, $\pi_s$ is removed from $CL$ and added to the output solution $\Pi'$. 
GRASP stops when all jobs from $CL$ are scheduled on the two available machines. In order to improve the solution generated by the GRC procedure, the VND procedure (Algorithm~\ref{VND}) with the same neighborhood structures as described in Section~\ref{VNDh} is used. 
The stopping criterion of GRASP is the CPU time limit $t_{max}$.  
 
\begin{footnotesize}
	\begin{algorithm}[!h]
	\scriptsize
		\SetAlgoLined \DontPrintSemicolon
		\KwData{A job sequence $\Pi$ }
		\KwResult{$\Pi'$}
		\SetKwFunction{proc}{\textsc{Greedy\_Randomized\_Construction}}
		\SetKwProg{myproc}{procedure}{}{}
		\myproc{\proc{}}{
			\textbf{Initialization}: $ \Pi' \leftarrow \varnothing $, $ \alpha_g \in [0, 1] $,  $CL \leftarrow \Pi$;\;
			Evaluate the incremental cost $\Phi(\pi_k)$ of each job $\pi_k \in CL$; \;  
			\Repeat{$ CL = \varnothing $}
			{	
				Calculate $\Phi_{min} = \min_{\pi_k \in CL} \Phi(\pi_k)$;\;
				Calculate $\Phi_{max} = \max_{\pi_k \in CL} \Phi(\pi_k)$;\;
				RCL $\leftarrow	\{\pi_j \in CL | \Phi(\pi_j) \leq \Phi_{min} + \alpha(\Phi_{max}-\Phi_{min}) \} $; \;
				Select a job $\pi_s$ from the RCL list at random;\;
				Schedule the selected job $\pi_s$ on the first available machine at the earliest possible time;\;
				$CL \leftarrow CL \setminus \{\pi_s\} $; \;
				$\Pi' \leftarrow	\Pi' \cup \{\pi_s\}$; \;
				Reevaluate the incremental cost $\Phi(\pi_k)$ for each remaining job $\pi_k \in CL$;\;
		}}
		\KwRet $\Pi'$
		\caption{Greedy Randomized Construction}
		\label{grasp2}
	\end{algorithm}
\end{footnotesize}

\section{Computational experiments} \label{sec6}

This section evaluates the computational performance of the mathematical formulations and the metaheuristic approaches. First, the characteristics of the test instances are provided (Section~\ref{inst}). The performance of the mathematical formulations is presented and discussed in Section~\ref{EA}. Finally, the performance of the metaheuristic approaches is summarized and discussed in Section~\ref{MA}. The computational experiments were conducted on a personal computer Intel(R) Core(TM) with i7-4600M 2.90 GHz CPU and 16GB of RAM, running Windows 7. To solve the $CF$, $CF^{+}$, $TIF$, and $TIF^{+}$ formulations, we have used the Concert Technology library of CPLEX 12.6 version with default settings in C++. 
The time limit for solving the formulations was set to 3600~s. The meheuristics GVNS~I, GVNS~II, and GRASP were implemented in the C++ language. We recall that in the proposed GVNS~I, the initial solution is obtained using an iterative improvement procedure. In addition, in GVNS~II the initial solution is randomly generated. For the proposed GVNS~I, GVNS~II and GRASP, the time limit ($t_{max}$) is set  to 10 seconds for small-sized instances ($n \in \{8,10,12,25\}$), $t_{max}$ is set  to 100 seconds for medium-sized instances ($n \in \{50,100\}$), and $t_{max}$ is set  to 300 seconds for large-sized instances ($n \in \{250,500\}$). According to the best practice of the related literature, the metaheuristics were executed 10 times in all experiments, except for the small-sized instances for which one run is sufficient, and the best and average results are
provided.

\subsection{Benchmark instances} 
\label{inst}

To the best of our knowledge, there are no publicly available benchmark instances for the problem $P2,S1|s_j,t_j|C_{max}$ involving a single server for the loading and unloading operations. Therefore, we have generated a set of instances according to the recent literature, as proposed by \cite{kim2021scheduling} and \cite{lee2021heuristic}. 

The instances are characterized by the following features:

\begin{itemize}
	\item The number of jobs $n \in\{8,10,12,25,50,100,250,500\}$.
	\item The integer processing times $p_j$ are uniformly distributed in the interval [10, 100].
	\item The integer loading times $s_j = \alpha \times p_j$ ($\alpha \in \{\alpha_1,\alpha_2,\alpha_3\}$), where $\alpha$ is a coefficient randomly generated by the uniform distribution with  $\alpha_1 \in [0.01, 0.1]$,  $\alpha_2 \in [0.1, 0.2]$, and  $\alpha_3 \in [0.1, 0.5]$.
	\item The integer unloading times $t_j = \alpha \times p_j$.
\end{itemize}

Note that $\alpha_1$, $\alpha_2$, and $\alpha_3$, respectively, correspond  to small, moderate, and large loading/unloading times variance (see \cite{kim2021scheduling,lee2021heuristic}).  For each combination of $(n, \alpha_i)$ $\forall i \in \{1,\ldots,3\}$, ten instances were created, resulting in a total of 240 new instances. We recall that small-sized instances are those with $n \in \{8,10,12,25\}$, medium-sized instances are those with $n \in \{50,100\}$ and large-sized instances are  those with  $n \in \{250,500\}$.

\subsection{Exact approaches}
\label{EA}

In Table~\ref{tab-1}, we compare the performance of the $CV$, $CV^{+}$, $TIF$ and $TIF^{+}$ formulations for small and medium-sized instances. Note that the results for large-sized instances $n \in \{250,500\}$ are not reported, since all formulations are not able to produce a feasible solution within the time limit of 3600~s. The results are provided for each number $n$ of
jobs and for each loading/unloading times variance ($\alpha \in \{\alpha_1,\alpha_2,\alpha_3\}$). In addition, for each formulation, the following information is given: $i)$ the number of instances solved to optimality, $\#opt$; $ii)$ the average time required to obtain an optimal solution, $t(s)$; $iii)$ the number of instances with a feasible solution, $\#is$; and $iv)$ the average percentage gap to optimality, $gap_{LB}(\%)$.

\begin{table}[!h]
\caption{Comparison of $CF$, $CF^{+}$, $TIF$, and $TIF^{+}$ for $n \in \{8, 10,12,25,50,100\}$.}
\centering
\renewcommand{\tabcolsep}{1.3pt}
{\scriptsize
\renewcommand{\arraystretch}{1.0}
    \begin{tabular}{|c|c|ccc|ccc|ccc|ccc|}
    \hline
    \multicolumn{2}{|c|}{Instances} & \multicolumn{3}{c|}{$CF$} & \multicolumn{3}{c|}{$CF^{+}$} & \multicolumn{3}{c|}{$TIF$} & \multicolumn{3}{c|}{$TIF^{+}$} \\
    \hline
    $n$ & $\alpha$ & $\#opt $ & $t(s)$ & $\#is$ [$gap_{LB}(\%)$] & $\#opt $ & $t(s)$ & $\#is$ [$gap_{LB}(\%)$] & $\#opt $ & $t(s)$ & $\#is$ [$gap_{LB}(\%)$] & $\#opt $ & $t(s)$ & $\#is$ [$gap_{LB}(\%)$] \\
    \hline
    8 & $\alpha_1$ & 10 & 6.47 & 0 [0] & 10 & 0.53 & 0 [0] & 10 & 4.35 & 0 [0] & 10 & 3.26 & 0 [0] \\
      & $\alpha_1$ & 10 & 5.84 & 0 [0] & 10 & 0.74 & 0 [0] & 10 & 12.81 & 0 [0] & 10 & 7.09 & 0 [0] \\
      & $\alpha_1$ & 10 & 4.92 & 0 [0] & 10 & 0.87 & 0 [0] & 10 & 36.86 & 0 [0] & 10 & 11.21 & 0 [0] \\
      \hline
    10 & $\alpha_1$ & 10 & 880.51 & 0 [0] & 10 & 6.58 & 0 [0] & 10 & 15.09 & 0 [0] & 10 & 6.81 & 0 [0] \\
      & $\alpha_1$ & 10 & 166.43 & 0 [0] & 10 & 40.07 & 0 [0] & 10 & 28.64 & 0 [0] & 10 & 13.83 & 0 [0] \\
      & $\alpha_1$ & 10 & 85.68 & 0 [0] & 10 & 10.70 & 0 [0] & 10 & 97.43 & 0 [0] & 10 & 40.66 & 0 [0] \\
      \hline
    12 & $\alpha_1$ & 0 & 3600 & 10 [14.74] & 8 & 1648.67 & 2 [0.29] & 10 & 51.74 & 0 [0] & 10 & 12.4 & 0 [0] \\
      & $\alpha_1$ & 0 & 3600 &  10 [18.34] & 8 & 1666.33 & 2 [0.18] & 10 & 160.67 & 0 [0] & 10 & 44.74 & 0 [0] \\
      & $\alpha_1$ & 4 & 1469.49 &  6 [20.02] & 7 & 893.33 & 3 [0.69] & 10 & 555.21 & 0 [0] & 10 & 149.91 & 0 [0] \\
    \hline
    25 & $\alpha_1$ & 0 & 3600 &  10 [81.39] & 0 & 3600 &  10 [0.16] & 4 & 2462.14 & 6 [37.69] & 6 & 1717.90 & 4 [24.72] \\
      & $\alpha_1$ & 0 & 3600 &  9 [77.37] & 0 & 3600 &  10 [0.26] & 1 & 3537 & 9 [36.89] & 1 & 2456.37 & 9 [31.23] \\
      & $\alpha_1$ & 0 & 3600 &  1 [66.66] & 0 & 3600 &  10 [0.91] & 0 & 3600 & 10 [42.10] & 0 & 3600 & 10 [38.26] \\
    \hline
    50 & $\alpha_1$ & 0 & 3600 &  10 [91.79] & 0 & 3600 &  10 [0.26] & 0 & 3600 &  10 [66.46] & 0 & 3600 &  10 [52.03] \\
      & $\alpha_1$ & 0 & 3600 &  9 [91.53] & 0 & 3600 &  10 [0.60] & 0 & 3600 &  10 [71.87] & 0 & 3600 &  10 [58.81] \\
      & $\alpha_1$ & 0 & 3600 &  10 [90.36] & 0 & 3600 &  10 [3.40] & 0 & 3600 &  10 [73.41] & 0 & 3600 &  10 [55.26] \\
      \hline
    100 & $\alpha_1$ & * & * & * & 0 & 3600 &  10 [0.96] & 0 & 3600 &  10 [79.63] & 0 & 3600 &  10 [94.90] \\
      & $\alpha_1$ & * & * & * & 0 & 3600 &  10 [3.79] & * & * & * & 0 & 3600 &  10 [99.98] \\
      & $\alpha_1$ & * & * & * & 0 & 3600 &  6 [15.15] & * & * & * & 0 & 3600 &  10 [100.00] \\
    \hline
\end{tabular}
}
\label{tab-1}
\end{table}

The following observations can be made:

\begin{itemize}
\item For $n=8$ : Based on the formulations $CF$, $CF^{+}$, $TIF$ and $TIF^{+}$, CPLEX is able to find an optimal solution for any instance. It can be noticed that
for the improved formulations $CV^{+}$ and $TIF^{+}$, CPLEX is able to produce an optimal solution in significantly less computational time in comparison with the original formulation. The best overall performance is demonstrated by $CF^{+}$ with an average computational time required to find an optimal solution equal to 0.71 s.

\item For $n=10$ : For all formulations, CPLEX is able to find an optimal solution for any instance. The best overall performance is demonstrated by $CF^{+}$ for $\alpha_1$, $TIF^{+}$ for $\alpha_2$, and $CF^{+}$ for $\alpha_3$. Based on the formulation $CF^{+}$, CPLEX is able to produce an optimal solution in significantly  less computational time in comparison with $CF$.  The best overall performance is demonstrated by $CF^{+}$ with an average computational time required to find an optimal solution equal to 19.11 s.

\item For $n=12$: $TIF$ and $TIF^{+}$ are
the only formulations for which CPLEX is able to find an optimal solution for any instance. Based on the formulation $CF$, CPLEX is able to produce an optimal solution for only 4 instances (among 30 ones). In addition, based on the improved formulation $CF^{+}$, CPLEX is able to find an optimal solution for 24 instances (among 30 ones). Note that the improved formulation $CF^{+}$ reduced significantly the value of $gap_{LB}(\%)$. The best overall performance is demonstrated by $TIF^{+}$ with an average computational time required to find an optimal solution equal to 69.01 s. 
 
\item For $n=25$ : $TIF^{+}$ is the only
formulation for which CPLEX is able to find an
optimal solution for 6 instances for $\alpha_1$, and 1 instance for $\alpha_2$. In addition, based on the formulations $CF$ and $CF^{+}$, CPLEX is able to produce a feasible solution for any instance at best. It can be noticed that the improved formulation $CF^{+}$ reduced significantly the value of $gap_{LB}(\%)$ as compared with the original one. The best overall performance is demonstrated by $TIF^{+}$.

\item For $n=50$: For all formulations, CPLEX is able to find a feasible solution for any instance at best. The improved formulations $CF^{+}$ reduced significantly the value of $gap_{LB}(\%)$. Note that based on the formulation  $CF$, CPLEX is not able to find a feasible solution for 1 instance. The best overall performance is demonstrated by $CF^{+}$ with small values of $gap_{LB}(\%)$.

\item For $n=100$ : Based on the formulation $CF$, CPLEX is not able to produce a feasible solution for any instance. Based on formulation $CF^{+}$, CPLEX is able to find a feasible solution for 26 instances for $\alpha_3$ (among the 30 ones).  In addition, based on formulation $TIF$, CPLEX is not able to find a feasible solution for all instances for $\alpha_2$ and $\alpha_3$. It can be noticed that the improved formulations $CF^{+}$ reduced significantly the value of $gap_{LB}(\%)$.

\end{itemize}

To sum up, the computational comparison of $CF$, $CF^{+}$, $TIF$ and $TIF^{+}$ shows that their performance is related to the number of jobs and the loading/unloading times  variance. In addition, using the proposed strengthening constraints (\ref{val_1}) and (\ref{val_2}), the $CF^{+}$ formulation produced lower bounds better than all other formulations for all instances (instances with a feasible solution). Moreover, for $n \in \{8,10\}$, the best formulation in terms of the average computing time required to obtain an optimal solution is $CF^{+}$. For $n=12$, the best performance is demonstrated by $TIF^{+}$ (since it solved to optimality all instances). For $n=25$, the best performance is demonstrated by $TIF^{+}$ (since it solved 7 instances among the 30 ones to optimality). For $n \in \{50,100\}$, the best formulation in terms of the average percentage gap to optimality is demonstrated by $CF^{+}$. It turns out that the formulations $CF^{+}$ and $TIF^{+}$ are complementary. Subsequently, we compare only $CF^{+}$ and $TIF^{+}$ with the other approaches since they produce the best results. Note that $TIF^{+}$ was able to prove optimality within 3600~s for 7 instances among the 30 ones with $n=25$. Therefore, metaheuristics are needed being able to find an approximate solution in a very short computational time.

\subsection{Metaheuristic approaches}
\label{MA}

\subsubsection{Results for small-sized instances}

In Tables~\ref{tab-2},~\ref{tab-3},~\ref{tab-4}, we compare the performance of $CF^{+}$, $TIF^{+}$, GVNS~I, GVNS~II, and GRASP for instances with $n \in \{8,10,12\}$, where an optimal solution can be found within 3600~s by $CF^{+}$ and $TIF^{+}$. Each instance is characterized by the following information. The ID (for example I1 denotes the $1^{th}$ instance with $n = 8$ and $\alpha = \alpha_1$); the number $n$ of jobs; the loading/unloading times variance; the value of the theoretical lower bound $LB_{T}$ computed as in Section~\ref{lb}. Next, the optimal makespan (denoted by $C_{max}^{*}$) obtained with the $CF^{+}$ and $TIF^{+}$ formulations is given. Finally, the computational time (CPU) to find an optimal solution is given for $CF^{+}$, $TIF^{+}$, GVNS~I, GVNS~II, and GRASP. Note that after each 10 instances of Table~\ref{tab-2},~\ref{tab-3},~\ref{tab-4} (for example I1,$\ldots$,I10), the average results for each column are reported (the best results are indicated in bold).

\begin{table}[!h]
\caption{Comparison of $CF^{+}$, $TIF^{+}$, GVNS~I, GVNS~II, and GRASP in terms of CPU time for $n=8$.}
\centering
\renewcommand{\tabcolsep}{1.5pt}
{\scriptsize
\renewcommand{\arraystretch}{0.75}
\begin{tabular}{|ccc|c|c|c|c|c|c|c|}
\hline
\multicolumn{5}{|c|}{Instance} & $CF^{+}$ & $TIF^{+}$ & GVNS~I & GVNS~II & GRASP \\
\hline
\multicolumn{1}{|c|}{ID} & \multicolumn{1}{c|}{$n$} & $\alpha$ & $LB_{T}$ & $C_{max}^{*}$ & \multicolumn{5}{c|}{CPU} \\
 \hline  
    \multicolumn{1}{|c|}{I1} & \multicolumn{1}{c|}{} &   & 295 & 295 & 0.456 & 4.114 & 0.000 & 0.001 & 0.002 \\
    \multicolumn{1}{|c|}{I2} & \multicolumn{1}{c|}{} &   & 288.5 & 289 & 0.612 & 5.486 & 0.000 & 0.006 & 0.054 \\
    \multicolumn{1}{|c|}{I3} & \multicolumn{1}{c|}{} &   & 258.5 & 259 & 0.529 & 4.274 & 0.000 & 0.000 & 0.000 \\
    \multicolumn{1}{|c|}{I4} & \multicolumn{1}{c|}{} &   & 217 & 217 & 0.418 & 2.821 & 0.006 & 0.001 & 0.001 \\
    \multicolumn{1}{|c|}{I5} & \multicolumn{1}{c|}{8} & $\alpha_1$ & 236.5 & 237 & 0.346 & 2.945 & 0.000 & 0.000 & 0.000 \\
    \multicolumn{1}{|c|}{I6} & \multicolumn{1}{c|}{} &   & 237 & 238 & 0.749 & 3.033 & 0.000 & 0.000 & 0.000 \\
    \multicolumn{1}{|c|}{I7} & \multicolumn{1}{c|}{} &   & 216 & 218 & 0.306 & 2.953 & 0.001 & 0.001 & 0.000 \\
    \multicolumn{1}{|c|}{I8} & \multicolumn{1}{c|}{} &   & 229.5 & 230 & 0.629 & 3.061 & 0.001 & 0.000 & 0.000 \\
    \multicolumn{1}{|c|}{I9} & \multicolumn{1}{c|}{} &   & 193 & 193 & 0.731 & 1.935 & 0.000 & 0.001 & 0.001 \\
    \multicolumn{1}{|c|}{I10} & \multicolumn{1}{c|}{} &   & 194.5 & 196 & 0.560 & 1.942 & 0.000 & 0.001 & 0.001 \\
     \hline
    \multicolumn{3}{|c|}{\textbf{Avg.}} & \textbf{236.55} & \textbf{237.2} & 0.534 & 3.256 & \textbf{0.001} & \textbf{0.001} & 0.006 \\
     \hline
    \multicolumn{1}{|c|}{I11} & \multicolumn{1}{c|}{} &   & 382.5 & 383 & 1.048 & 9.911 & 0.013 & 0.052 & 0.091 \\
    \multicolumn{1}{|c|}{I12} & \multicolumn{1}{c|}{} &   & 277 & 277 & 0.540 & 5.428 & 0.002 & 0.004 & 0.007 \\
    \multicolumn{1}{|c|}{I13} & \multicolumn{1}{c|}{} &   & 274.5 & 276 & 0.561 & 6.058 & 0.000 & 0.004 & 0.007 \\
    \multicolumn{1}{|c|}{I14} & \multicolumn{1}{c|}{} &   & 326.5 & 328 & 0.726 & 5.673 & 0.001 & 0.001 & 0.000 \\
    \multicolumn{1}{|c|}{I15} & \multicolumn{1}{c|}{8} & $\alpha_2$ & 259.5 & 260 & 0.519 & 3.853 & 0.001 & 0.003 & 0.001 \\
    \multicolumn{1}{|c|}{I16} & \multicolumn{1}{c|}{} &   & 311.5 & 312 & 0.845 & 6.795 & 0.000 & 0.003 & 0.001 \\
    \multicolumn{1}{|c|}{I17} & \multicolumn{1}{c|}{} &   & 319 & 320 & 0.638 & 10.524 & 0.003 & 0.001 & 0.004 \\
    \multicolumn{1}{|c|}{I18} & \multicolumn{1}{c|}{} &   & 284.5 & 286 & 0.831 & 5.876 & 0.002 & 0.003 & 0.003 \\
    \multicolumn{1}{|c|}{I19} & \multicolumn{1}{c|}{} &   & 334.5 & 336 & 0.851 & 7.455 & 0.001 & 0.001 & 0.001 \\
    \multicolumn{1}{|c|}{I20} & \multicolumn{1}{c|}{} &   & 347 & 349 & 0.812 & 9.326 & 0.003 & 0.002 & 0.001 \\
    \hline
    \multicolumn{3}{|c|}{\textbf{Avg.}} & \textbf{311.65} & \textbf{312.7} & 0.737 & 7.090 & \textbf{0.003} & 0.007 & 0.011 \\
     \hline
    \multicolumn{1}{|c|}{I21} & \multicolumn{1}{c|}{} &   & 324 & 325 & 0.979 & 8.029 & 0.002 & 0.002 & 0.080 \\
    \multicolumn{1}{|c|}{I22} & \multicolumn{1}{c|}{} &   & 406.5 & 408 & 1.081 & 10.123 & 0.127 & 0.004 & 8.597 \\
    \multicolumn{1}{|c|}{I23} & \multicolumn{1}{c|}{} &   & 324 & 325 & 0.990 & 15.213 & 0.006 & 0.008 & 0.405 \\
    \multicolumn{1}{|c|}{I24} & \multicolumn{1}{c|}{} &   & 246.5 & 248 & 0.634 & 3.621 & 0.002 & 0.016 & 0.001 \\
    \multicolumn{1}{|c|}{I25} & \multicolumn{1}{c|}{8} & $\alpha_3$ & 350.5 & 352 & 0.839 & 15.506 & 0.005 & 0.004 & 0.002 \\
    \multicolumn{1}{|c|}{I26} & \multicolumn{1}{c|}{} &   & 331.5 & 335 & 0.769 & 13.742 & 0.024 & 0.380 & 1.893 \\
    \multicolumn{1}{|c|}{I27} & \multicolumn{1}{c|}{} &   & 264.5 & 266 & 0.810 & 3.301 & 0.002 & 0.007 & 0.299 \\
    \multicolumn{1}{|c|}{I28} & \multicolumn{1}{c|}{} &   & 293 & 300 & 0.861 & 25.168 & 0.000 & 0.005 & 0.201 \\
    \multicolumn{1}{|c|}{I29} & \multicolumn{1}{c|}{} &   & 385 & 387 & 0.830 & 9.247 & 0.008 & 0.010 & 0.800 \\
    \multicolumn{1}{|c|}{I30} & \multicolumn{1}{c|}{} &   & 331 & 331 & 0.898 & 8.179 & 0.005 & 0.003 & 0.499 \\
    \hline
    \multicolumn{3}{|c|}{\textbf{Avg.}} & \textbf{325.65} & \textbf{327.70} & 0.869 & 11.213 & \textbf{0.018} & 0.044 & 1.278 \\
    \hline
\end{tabular}
}
\label{tab-2}
\end{table}

The following observations can be made:

\begin{itemize}
\item In Table~\ref{tab-2}~:~$CF^{+}$, $TIF^{+}$, GVNS~I, GVNS~II, and GRASP are compared. It can be noticed that the three proposed metaheuristics can reach an optimal solution for any instance in significantly less computational time than the two formulations. For example, for $(n=8,\alpha=\alpha_3$), the average computational time for $CF^{+}$ is 0.869~s, while the average computational time for GVNS~I is 0.018~s.  It can be noted that the value of the theoretical lower bound ($LB_T$) is very tight (since the average gap between the optimal makespan and $LB_T$ is equal to 0.427\%).
The best overall performance is demonstrated by GVNS~I with a total average computational time of 0.007~s for $n=8$.

\item In Table~\ref{tab-3}~:~GVNS~I, GVNS~II, and GRASP can reach an optimal solution for any instance in significantly less computational time than the $CF^{+}$ and $TIF^{+}$ formulations. For example, for $(n=10,\alpha=\alpha_1)$, the average computational time for $TIF^{+}$ is 6.812~s, while the average computational time for GVNS~II is equal to 0.003~s. The average gap between $C_{max}^{*}$ and $LB_T$ is equal to 0.176\%. The best overall performance is shown by GVNS~I with a total average computational time  of 0.043~s for $n=10$.

\item In Table~\ref{tab-3}~:~$TIF^{+}$, GVNS~I, GVNS~II, and GRASP are compared (since $CF^{+}$ is not able to produce an optimal solution for all instances). It can be noticed that GVNS~I, GVNS~II, and GRASP can reach an optimal solution for any instance in significantly less computational time than the $TIF^{+}$ formulation. For example, for $(n=12,\alpha=\alpha_3$), the average computational time for $TIF^{+}$ is equal to 149.914~s, while the average computational time for GVNS~I is 2.785~s. Again, the value of $LB_T$ is very tight since the average gap between $C_{max}^{*}$ and $LB_T$ is equal to 0.115\%. The best overall performance is demonstrated by GVNS~I with a total average computational time of 0.930~s for $n=10$.
\end{itemize}

\begin{table}[!h]
\caption{Comparison of $CF^{+}$, $TIF^{+}$, GVNS~I, GVNS~II, and GRASP in terms of CPU time for $n=10$.}
\centering
\renewcommand{\tabcolsep}{1.5pt}
{\scriptsize
\renewcommand{\arraystretch}{0.75}
\begin{tabular}{|ccc|c|c|c|c|c|c|c|}
\hline
\multicolumn{5}{|c|}{Instance} & $CF^{+}$ & $TIF^{+}$ & GVNS~I & GVNS~II & GRASP\\
\hline
\multicolumn{1}{|c|}{ID} & \multicolumn{1}{c|}{$n$} & $\alpha$ & $LB_{T}$ & $C_{max}^{*}$ & \multicolumn{5}{c|}{CPU} \\
\hline
    \multicolumn{1}{|c|}{I31} & \multicolumn{1}{c|}{} &   & 277 & 277 & 0.676 & 4.221 & 0.006 & 0.008 & 0.007 \\
    \multicolumn{1}{|c|}{I32} & \multicolumn{1}{c|}{} &   & 241.5 & 242 & 0.716 & 5.046 & 0.001 & 0.001 & 0.001 \\
    \multicolumn{1}{|c|}{I33} & \multicolumn{1}{c|}{} &   & 310 & 310 & 9.827 & 7.938 & 0.002 & 0.004 & 0.007 \\
    \multicolumn{1}{|c|}{I34} & \multicolumn{1}{c|}{} &   & 298.5 & 299 & 1.865 & 4.864 & 0.001 & 0.000 & 0.001 \\
    \multicolumn{1}{|c|}{I35} & \multicolumn{1}{c|}{10} & $\alpha_1$ & 312.5 & 313 & 1.154 & 6.523 & 0.002 & 0.001 & 0.001 \\
    \multicolumn{1}{|c|}{I36} & \multicolumn{1}{c|}{} &   & 380 & 380 & 17.25 & 8.272 & 0.006 & 0.005 & 0.001 \\
    \multicolumn{1}{|c|}{I37} & \multicolumn{1}{c|}{} &   & 310.5 & 311 & 1.069 & 8.639 & 0.001 & 0.001 & 0.002 \\
    \multicolumn{1}{|c|}{I38} & \multicolumn{1}{c|}{} &   & 293 & 293 & 1.668 & 10.129 & 0.005 & 0.004 & 0.013 \\
    \multicolumn{1}{|c|}{I39} & \multicolumn{1}{c|}{} &   & 321 & 321 & 0.677 & 9.687 & 0.039 & 0.007 & 0.012 \\
    \multicolumn{1}{|c|}{I40} & \multicolumn{1}{c|}{} &   & 193 & 193 & 30.944 & 2.803 & 0.001 & 0.001 & 0.001 \\
    \hline
    \multicolumn{3}{|c|}{\textbf{Avg.}} & \textbf{293.70} & \textbf{293.90} & 6.585 & 6.812 & 0.006 & \textbf{0.003} & 0.004 \\
    \hline
    \multicolumn{1}{|c|}{I41} & \multicolumn{1}{c|}{} &   & 414.5 & 416 & 44.247 & 38.625 & 0.004 & 0.008 & 0.007 \\
    \multicolumn{1}{|c|}{I42} & \multicolumn{1}{c|}{} &   & 445.5 & 448 & 14.011 & 59.407 & 0.004 & 0.001 & 0.002 \\
    \multicolumn{1}{|c|}{I43} & \multicolumn{1}{c|}{} &   & 310 & 311 & 9.018 & 8.268 & 0.005 & 0.002 & 0.002 \\
    \multicolumn{1}{|c|}{I44} & \multicolumn{1}{c|}{} &   & 227 & 227 & 0.529 & 6.365 & 0.007 & 0.002 & 0.010 \\
    \multicolumn{1}{|c|}{I45} & \multicolumn{1}{c|}{10} & $\alpha_2$ & 347 & 347 & 6.84 & 14.487 & 0.016 & 0.044 & 0.004 \\
    \multicolumn{1}{|c|}{I46} & \multicolumn{1}{c|}{} &   & 222 & 222 & 1.981 & 3.899 & 0.098 & 0.002 & 0.001 \\
    \multicolumn{1}{|c|}{I47} & \multicolumn{1}{c|}{} &   & 316 & 317 & 298.729 & 7.428 & 0.011 & 0.036 & 0.010 \\
    \multicolumn{1}{|c|}{I48} & \multicolumn{1}{c|}{} &   & 280.5 & 281 & 1.432 & 9.469 & 0.010 & 0.007 & 0.007 \\
    \multicolumn{1}{|c|}{I49} & \multicolumn{1}{c|}{} &   & 273 & 273 & 12.541 & 4.958 & 0.004 & 0.004 & 0.013 \\
    \multicolumn{1}{|c|}{I50} & \multicolumn{1}{c|}{} &   & 356.5 & 357 & 11.419 & 15.357 & 0.006 & 0.001 & 0.005 \\
    \hline
    \multicolumn{3}{|c|}{\textbf{Avg.}} & \textbf{319.20} & \textbf{319.90} & 40.075 & 16.826 & 0.016 & 0.011 & \textbf{0.006} \\
    \hline
    \multicolumn{1}{|c|}{I51} & \multicolumn{1}{c|}{} &   & 477.5 & 479 & 5.491 & 19.16 & 0.333 & 5.346 & 41.661 \\
    \multicolumn{1}{|c|}{I52} & \multicolumn{1}{c|}{} &   & 316 & 316 & 21.075 & 16.941 & 0.037 & 0.515 & 2.255 \\
    \multicolumn{1}{|c|}{I53} & \multicolumn{1}{c|}{} &   & 518 & 519 & 8.023 & 30.661 & 0.079 & 0.063 & 0.354 \\
    \multicolumn{1}{|c|}{I54} & \multicolumn{1}{c|}{} &   & 456.5 & 458 & 17.57 & 43.689 & 0.022 & 0.039 & 0.051 \\
    \multicolumn{1}{|c|}{I55} & \multicolumn{1}{c|}{10} & $\alpha_3$ & 411 & 413 & 5.638 & 33.458 & 0.301 & 1.459 & 0.845 \\
    \multicolumn{1}{|c|}{I56} & \multicolumn{1}{c|}{} &   & 347.5 & 349 & 7.123 & 21.732 & 0.040 & 0.014 & 0.034 \\
    \multicolumn{1}{|c|}{I57} & \multicolumn{1}{c|}{} &   & 356.5 & 358 & 7.302 & 25.912 & 0.114 & 0.496 & 0.480 \\
    \multicolumn{1}{|c|}{I58} & \multicolumn{1}{c|}{} &   & 485 & 487 & 7.683 & 67.399 & 0.023 & 0.034 & 0.106 \\
    \multicolumn{1}{|c|}{I59} & \multicolumn{1}{c|}{} &   & 523 & 523 & 13.466 & 92.011 & 0.022 & 0.120 & 0.007 \\
    \multicolumn{1}{|c|}{I60} & \multicolumn{1}{c|}{} &   & 443.5 & 444 & 13.581 & 55.662 & 0.076 & 0.552 & 0.127 \\
    \hline
    \multicolumn{3}{|c|}{\textbf{Avg.}} & \textbf{433.45} & \textbf{434.60} & 10.695 & 40.663 & \textbf{0.105} & 0.864 & 4.592 \\
\hline
\end{tabular}
}
\label{tab-3}
\end{table}

\begin{table}[!h]
\caption{Comparison of $TIF^{+}$, GVNS~I, GVNS~II, and GRASP in terms of CPU time for $n=12$.}
\centering
\renewcommand{\tabcolsep}{1.5pt}
{\scriptsize
\renewcommand{\arraystretch}{0.75}
\begin{tabular}{|ccc|c|c|c|c|c|c|}
\hline
\multicolumn{5}{|c|}{Instance} & $TIF^{+}$ & GVNS I & GVNS II & GRASP \\
 \hline
\multicolumn{1}{|c|}{ID} & \multicolumn{1}{c|}{$n$} & $\alpha$ & $LB_{T}$ & $C_{max}^{*}$ & \multicolumn{4}{c|}{CPU} \\
    \hline
    \multicolumn{1}{|c|}{I61} & \multicolumn{1}{c|}{} &   & 403 & 403 & 21.772 & 0.000 & 0.000 & 0.005 \\
    \multicolumn{1}{|c|}{I62} & \multicolumn{1}{c|}{} &   & 338 & 338 & 13.471 & 0.005 & 0.006 & 0.003 \\
    \multicolumn{1}{|c|}{I63} & \multicolumn{1}{c|}{} &   & 338.5 & 339 & 6.650 & 0.000 & 0.002 & 0.002 \\
    \multicolumn{1}{|c|}{I64} & \multicolumn{1}{c|}{} &   & 330 & 330 & 9.110 & 0.002 & 0.005 & 0.002 \\
    \multicolumn{1}{|c|}{I65} & \multicolumn{1}{c|}{12} & $\alpha_1$ & 393 & 393 & 17.890 & 0.000 & 0.000 & 0.005 \\
    \multicolumn{1}{|c|}{I66} & \multicolumn{1}{c|}{} &   & 219.5 & 220 & 4.217 & 0.000 & 0.000 & 0.000 \\
    \multicolumn{1}{|c|}{I67} & \multicolumn{1}{c|}{} &   & 352 & 352 & 7.460 & 0.000 & 0.002 & 0.002 \\
    \multicolumn{1}{|c|}{I68} & \multicolumn{1}{c|}{} &   & 383 & 383 & 12.570 & 0.002 & 0.013 & 0.005 \\
    \multicolumn{1}{|c|}{I69} & \multicolumn{1}{c|}{} &   & 368.5 & 369 & 10.415 & 0.000 & 0.000 & 0.002 \\
    \multicolumn{1}{|c|}{I70} & \multicolumn{1}{c|}{} &   & 372 & 372 & 20.398 & 0.002 & 0.008 & 0.005 \\
     \hline
    \multicolumn{3}{|c|}{\textbf{Avg.}} & \textbf{349.75} & \textbf{349.9} & 12.395 & \textbf{0.001} & 0.003 & 0.003 \\
     \hline
    \multicolumn{1}{|c|}{I71} & \multicolumn{1}{c|}{} &   & 474 & 475 & 44.769 & 0.000 & 0.008 & 0.002 \\
    \multicolumn{1}{|c|}{I72} & \multicolumn{1}{c|}{} &   & 444.5 & 446 & 36.420 & 0.003 & 0.000 & 0.006 \\
    \multicolumn{1}{|c|}{I73} & \multicolumn{1}{c|}{} &   & 459 & 459 & 47.066 & 0.000 & 0.042 & 0.014 \\
    \multicolumn{1}{|c|}{I74} & \multicolumn{1}{c|}{} &   & 381 & 382 & 18.444 & 0.003 & 0.009 & 0.000 \\
    \multicolumn{1}{|c|}{I75} & \multicolumn{1}{c|}{12} & $\alpha_2$ & 352.5 & 353 & 21.959 & 0.003 & 0.000 & 0.005 \\
    \multicolumn{1}{|c|}{I76} & \multicolumn{1}{c|}{} &   & 500.5 & 501 & 64.852 & 0.016 & 0.009 & 0.053 \\
    \multicolumn{1}{|c|}{I77} & \multicolumn{1}{c|}{} &   & 363.5 & 364 & 14.919 & 0.016 & 0.005 & 0.006 \\
    \multicolumn{1}{|c|}{I78} & \multicolumn{1}{c|}{} &   & 546.5 & 547 & 69.651 & 0.002 & 0.000 & 0.005 \\
    \multicolumn{1}{|c|}{I79} & \multicolumn{1}{c|}{} &   & 589.5 & 590 & 105.075 & 0.006 & 0.000 & 0.002 \\
    \multicolumn{1}{|c|}{I80} & \multicolumn{1}{c|}{} &   & 401 & 402 & 24.291 & 0.000 & 0.000 & 0.002 \\
     \hline
    \multicolumn{3}{|c|}{\textbf{Avg.}} & \textbf{451.2} & \textbf{451.9} & 44.745 & \textbf{0.005} & 0.007 & 0.009 \\
     \hline
    \multicolumn{1}{|c|}{I81} & \multicolumn{1}{c|}{} &   & 518 & 519 & 150.531 & 0.480 & 0.196 & 0.273 \\
    \multicolumn{1}{|c|}{I82} & \multicolumn{1}{c|}{} &   & 510.5 & 512 & 196.291 & 2.799 & 3.911 & 1.894 \\
    \multicolumn{1}{|c|}{I83} & \multicolumn{1}{c|}{} &   & 589.5 & 590 & 171.298 & 2.878 & 4.770 & 3.999 \\
    \multicolumn{1}{|c|}{I84} & \multicolumn{1}{c|}{} &   & 684.5 & 686 & 213.726 & 1.797 & 3.697 & 4.922 \\
    \multicolumn{1}{|c|}{I85} & \multicolumn{1}{c|}{12} & $\alpha_3$ & 485.5 & 486 & 99.752 & 1.039 & 0.013 & 1.232 \\
    \multicolumn{1}{|c|}{I86} & \multicolumn{1}{c|}{} &   & 570 & 570 & 116.832 & 3.005 & 1.557 & 2.717 \\
    \multicolumn{1}{|c|}{I87} & \multicolumn{1}{c|}{} &   & 632 & 632 & 207.810 & 0.187 & 0.020 & 0.201 \\
    \multicolumn{1}{|c|}{I88} & \multicolumn{1}{c|}{} &   & 532.5 & 533 & 148.399 & 0.208 & 0.286 & 0.614 \\
    \multicolumn{1}{|c|}{I89} & \multicolumn{1}{c|}{} &   & 431 & 432 & 46.514 & 0.925 & 1.004 & 1.407 \\
    \multicolumn{1}{|c|}{I90} & \multicolumn{1}{c|}{} &   & 492.5 & 493 & 147.986 & 14.533 & 52.101 & 30.002 \\
     \hline
    \multicolumn{3}{|c|}{\textbf{Avg.}} & \textbf{544.60} & \textbf{545.30} & 149.914 & \textbf{2.785} & 6.756 & 4.726 \\
    \hline
\end{tabular}
}
\label{tab-4}
\end{table}

\begin{table}[!h]
\caption{Comparison of $CF^{+}$, $TIF^{+}$, GVNS~I, GVNS~II and GRASP for $n=25$.}
\centering
\renewcommand{\tabcolsep}{0.75pt}
{\scriptsize
\renewcommand{\arraystretch}{0.70}
\begin{tabular}{|c|c|c|c|cccc|cccc|ccc|ccc|ccc|}
\hline
 \multicolumn{4}{|c|}{Instance} & \multicolumn{4}{c|}{$CF^{+}$} & \multicolumn{4}{c|}{$TIF^{+}$} & \multicolumn{3}{c|}{GVNS I} & \multicolumn{3}{c|}{GVNS II} & \multicolumn{3}{c|}{GRASP} \\
 \hline
    ID & $n$ & $\alpha$ & $LB_{T}$ & $UB_{CF^{+}}$ & $LB_{CF^{+}}$ & $Gap_{CF^{+}}(\%)$ & CPU & $UB_{TIF^{+}}$ & $LB_{TIF^{+}}$ & $Gap_{TIF^{+}}(\%)$ & CPU & $C_{max}^{best}$& $C_{max}^{avg}$ & CPU & $C_{max}^{best}$ & $C_{max}^{avg}$ & CPU & $C_{max}^{best}$ & $C_{max}^{avg}$ & CPU \\
\hline
    I91 &   &   & 800.5 & 801 & 799.5 & 0.19 & 3600 & 801 & 549 & 31.47 & 3600 & \textbf{801} & 801 & 0.03 & \textbf{801} & 801 & 0.02 & \textbf{801} & 801 & 0.02 \\
    I92 &   &   & 763 & 763 & 762 & 0.13 & 3600 & 763 & 763 & 0 & 1782.26 & \textbf{763} & 763 & 0.16 & \textbf{763} & 763 & 0.23 & \textbf{763} & 763 & 0.13 \\
    I93 &   &   & 860.5 & 861 & 859.5 & 0.17 & 3600 & 861 & 861 & 0 & 1454.95 & \textbf{861} & 861 & 0.02 & \textbf{861} & 861 & 0.01 & \textbf{861} & 861 & 0.01 \\
    I94 &   &   & 881 & 882 & 880 & 0.23 & 3600 & 881 & 677 & 23.16 & 3600 & \textbf{881} & 881 & 0.08 & \textbf{881} & 881 & 0.10 & \textbf{881} & 881 & 0.09 \\
    I95 & 25 & $\alpha_1$ & 816 & 816 & 815 & 0.12 & 3600 & 816 & 816 & 0 & 970.43 & \textbf{816} & 816 & 0.05 & \textbf{816} & 816 & 0.07 & \textbf{816} & 816 & 0.05 \\
    I96 &   &   & 705 & 705 & 704 & 0.14 & 3600 & 705 & 545 & 22.70 & 3600 & \textbf{705} & 705 & 0.01 & \textbf{705} & 705 & 0.01 & \textbf{705} & 705 & 0.02 \\
    I97 &   &   & 711 & 711 & 710 & 0.14 & 3600 & 711 & 711 & 0 & 1189.41 & \textbf{711} & 711 & 0.05 & \textbf{711} & 711 & 0.03 & \textbf{711} & 711 & 0.04 \\
    I98 &   &   & 899.5 & 900 & 898.5 & 0.17 & 3600 & 901 & 706.7 & 21.56 & 3600 & \textbf{900} & 900 & 0.01 & \textbf{900} & 900 & 0.01 & \textbf{900} & 900 & 0.01 \\
    I99 &   &   & 790 & 790 & 789 & 0.13 & 3600 & 790 & 790 & 0 & 3025.83 & \textbf{790} & 790 & 0.02 & \textbf{790} & 790 & 0.02 & \textbf{790} & 790 & 0.01 \\
    I100 &   &   & 742 & 742 & 741 & 0.13 & 3600 & 742 & 742 & 0 & 1884.51 & \textbf{742} & 742 & 0.02 & \textbf{742} & 742 & 0.02 & \textbf{742} & 742 & 0.02\\
    \hline
    I101 &   &   & 1079.5 & 1080 & 1077 & 0.28 & 3600 & 1085 & 719.6 & 33.68 & 3600 & \textbf{1080} & 1080 & 0.26 & \textbf{1080} & 1080 & 0.20 & \textbf{1080} & 1080 & 0.33 \\
    I102 &   &   & 997 & 998 & 996 & 0.20 & 3600 & 1008 & 664.9 & 34.04 & 3600 & \textbf{997} & 997 & 1.65 & \textbf{997} & 997 & 1.76 & \textbf{997} & 997 & 2.70 \\
    I103 &   &   & 767.5 & 768 & 766.5 & 0.20 & 3600 & 768 & 768 & 0 & 2456.37 & \textbf{768} & 768 & 0.20 & \textbf{768} & 768 & 0.18 & \textbf{768} & 768 & 0.14 \\
    I104 &   &   & 831 & 832 & 829.5 & 0.30 & 3600 & 831 & 586.3 & 29.45 & 3600 & \textbf{831} & 831 & 0.76 & \textbf{831} & 831 & 0.48 & \textbf{831} & 831 & 0.81 \\
    I105 & 25 & $\alpha_2$ & 854 & 855 & 853 & 0.23 & 3600 & 859 & 573.1 & 33.28 & 3600 & \textbf{854} & 854.2 & 3.23 & \textbf{854} & 854.3 & 2.53 & \textbf{854} & 854.3 & 1.99 \\
    I106 &   &   & 974.5 & 975 & 972.5 & 0.26 & 3600 & 975 & 650.8 & 33.25 & 3600 & \textbf{975} & 975 & 0.52 & \textbf{975} & 975 & 0.39 & \textbf{975} & 975 & 0.58 \\
    I107 &   &   & 888 & 890 & 887 & 0.34 & 3600 & 891 & 595 & 33.22 & 3600 & \textbf{888} & 888 & 1.96 & \textbf{888} & 888 & 2.28 & \textbf{888} & 888 & 3.28 \\
    I108 &   &   & 737.5 & 739 & 736.5 & 0.34 & 3600 & 738 & 558 & 24.39 & 3600 & \textbf{738} & 738 & 0.20 & \textbf{738} & 738 & 0.44 & \textbf{738} & 738 & 0.40 \\
    I109 &   &   & 978 & 978 & 976.5 & 0.15 & 3600 & 988 & 632.6 & 35.97 & 3600 & \textbf{978} & 978 & 0.29 & \textbf{978} & 978 & 0.24 & \textbf{978} & 978 & 0.44 \\
    I110 &   &   & 756 & 757 & 754.5 & 0.33 & 3600 & 757 & 576.9 & 23.79 & 3600 & \textbf{756} & 756 & 1.78 & \textbf{756} & 756 & 1.91 & \textbf{756} & 756 & 1.84 \\
    \hline
    I111 &   &   & 1023.5 & 1028 & 1021 & 0.68 & 3600 & 1052 & 648.7 & 38.34 & 3600 & 1027 & 1031.6 & 5.24 & 1027 & 1030.3 & 6.05 & \textbf{1026} & 1032.7 & 3.85 \\
    I112 &   &   & 1215 & 1227 & 1213 & 1.14 & 3600 & 1250 & 666.1 & 46.71 & 3600 & 1217 & 1220.2 & 4.31 & \textbf{1215} & 1221.3 & 4.77 & 1220 & 1224.4 & 5.19 \\
    I113 &   &   & 1108 & 1123 & 1105.5 & 1.56 & 3600 & 1156 & 718.4 & 37.85 & 3600 & 1124 & 1128.1 & 5.58 & \textbf{1116} & 1128.8 & 4.36 & 1127 & 1131.7 & 5.12 \\
    I114 &   &   & 1182.5 & 1187 & 1179.5 & 0.63 & 3600 & 1255 & 681.8 & 45.68 & 3600 & \textbf{1184} & 1189.6 & 5.79 & 1187 & 1192.4 & 5.50 & 1190 & 1192.6 & 4.68 \\
    I115 & 25 & $\alpha_3$ & 1216.5 & 1229 & 1214.5 & 1.18 & 3600 & 1241 & 794.5 & 35.98 & 3600 & 1226 & 1236.1 & 3.31 & \textbf{1223} & 1232.8 & 7.29 & 1230 & 1238.7 & 4.83 \\
    I116 &   &   & 1113.5 & 1123 & 1110.5 & 1.11 & 3600 & 1141 & 745.4 & 34.67 & 3600 & 1116 & 1121.9 & 5.34 & 1116 & 1121.8 & 6.24 & \textbf{1115} & 1121.4 & 5.52 \\
    I117 &   &   & 1045 & 1054 & 1043 & 1.04 & 3600 & 1074 & 659.3 & 38.62 & 3600 & 1057 & 1059.3 & 4.77 & 1055 & 1058 & 2.92 & \textbf{1050} & 1055.7 & 5.55 \\
    I118 &   &   & 943 & 944 & 941 & 0.32 & 3600 & 953 & 634.8 & 33.39 & 3600 & 947 & 951.2 & 2.96 & \textbf{946} & 950.8 & 5.04 & \textbf{946} & 949.9 & 6.76 \\
    I119 &   &   & 1230.5 & 1238 & 1227 & 0.89 & 3600 & 1237 & 780.9 & 36.87 & 3600 & \textbf{1233} & 1239.5 & 3.20 & \textbf{1233} & 1239.6 & 4.41 & 1236 & 1241 & 5.26 \\
    I120 &   &   & 1135 & 1138 & 1132 & 0.53 & 3600 & 1159 & 758.9 & 34.52 & 3600 & 1141 & 1145.9 & 3.31 & 1138 & 1146.3 & 5.10 & \textbf{1137} & 1144.8 & 4.97 \\
\hline
\end{tabular}%
}
\label{tab-5}
\end{table}

In Table~\ref{tab-5}, we compare the performance of $CF^{+}$, $TIF^{+}$, GVNS~I, GVNS~II, and GRASP for $n=25$. Each instance is characterized by the following information. The ID; the number $n$ of jobs; the loading/unloading times variance; and the value of the theoretical lower bound $LB_{T}$. For the formulation $CF^{+}$ (resp. $TIF^{+}$), the following information is given. The upper bound $UB_{{CF}^{+}}$ (respectively $UB_{{TIF}^{+}}$), the lower bound $LB_{{CF}^{+}}$ (respectively $LB_{{TIF}^{+}}$), the percentage gap to optimality $Gap_{{CF}^{+}}(\%)$ (resp. $Gap_{{TIF}^{+}}(\%)$), and the CPU time required to prove optimality (below 3600~s). The following results are presented for GVNS~I, GVNS~II and GRASP: the best (resp. the average) makespan value over 10 runs denoted as $C_{max}^{best}$  (resp. $C_{max}^{avg}$). Finally, the average computational times are also provided that are computed over the 10 runs, and the computational time of a run corresponds to the time at which the best solution
is found (the best results are indicated in bold face). 

The following observations can be made. Based on $CF^{+}$, CPLEX is not able to produce an optimal solution for all instances.  Furthermore, based on the $TIF^{+}$ formulation, CPLEX is able to find an optimal solution for 6 instances  for ($n=25,\alpha=\alpha_1$) (I2, I3, I5, I7, I9 and I10). For these instances, GVNS~I, GVNS~II, and GRASP are able to find the same optimal solution in a significantly less computational time in comparison with $TIF^{+}$. For ($n=25,\alpha=\alpha_2$),  based on $TIF^{+}$ formulation, CPLEX is able to generate an optimal solution for only 1 instance (I103). GVNS~I, GVNS~II, and GRASP are able to find the optimal solution for instance I103 in a significantly less computing time in comparison with $TIF^{+}$.  For $\alpha_3$, $CF^{+}$ produced better upper bounds than $TIF^{+}$. On average, GVNS~I, GVNS~II, and GRASP are able to produce approximate solutions of better quality in comparison with the  upper bounds generated by the formulations $CF^{+}$ and $TIF^{+}$. To sum up, for $n=25$ GVNS~I, GVNS~II, and GRASP have a similar performance. In addition, the difference between $C_{max}^{best}$ and $C_{max}^{avg}$ is very small (often below one unit).

\subsubsection{Results for medium-sized instances}
In Table~\ref{tab-6} in the Appendix, we compare the performance of $CF^{+}$, $TIF^{+}$, GVNS~I, GVNS~II, and GRASP for $n=50$, and in Table~\ref{tab-7}, we compare the performance of GVNS~I, GVNS~II, and GRASP only with $CF^{+}$ (since the $TIF^{+}$ formulation is not able to find a feasible solution for the majority of instances). Tables~\ref{tab-6},~\ref{tab-7} have the same structure as Table~\ref{tab-5}, and the best results are indicated in bold face.

The following observations can be made. 

\begin{itemize}
\item For $n=50$ : Overall, among the 30 instances, GVNS~I found 25 best solutions (83.33\%), whereas GVNS~II and GRASP found 23 (76.66\%) and 20 (66.66\%) ones, respectively.  It can be noted that the difference between $C_{max}^{best}$ and $C_{max}^{avg}$ is on average very small for GVNS~I, GVNS~II and GRASP.  For the instances with $\alpha_3$, $CF^{+}$ produced better upper bounds than $TIF^{+}$. On average, GVNS~I, GVNS~II and GRASP are able to produce approximate solutions of better quality in comparison with the  upper bounds generated by the formulations $CF^{+}$ and $TIF^{+}$. In addition, for all instances of $\alpha_1$ and $\alpha_2$, the gap between $C_{max}^{best}$ and $LB_{T}$ is very small for GVNS~I, GVNS~II and GRASP. The best performance in terms of solution quality and computational time is demonstrated by GVNS~I, with 24 best solutions and an overall average computational time of 23.86~s.

\item For $n=100$ : Overall, among the 30 instances, GVNS~I found 24 best solutions (80\%),
whereas GVNS~II and GRASP found only 16 (53.33\%) and 11 (36.66\%) ones, respectively. Based on $CF^{+}$, CPLEX is able to produce a feasible solution at best for 24 instances (4 instances without a feasible solution). For the instances with $\alpha_1$,  the gap between $C_{max}^{best}$ and $LB_{T}$ is very small for GVNS~I, GVNS~II and GRASP. On average, GVNS~I, GVNS~II and GRASP are able to produce approximate solutions of better quality in comparison with the  upper bounds generated by the formulation $CF^{+}$. For the instances with $\alpha_3$ and for GVNS~I, GVNS~II and GRASP, the difference between $C_{max}^{best} $ and $C_{max}^{avg}$ grows significantly. The best overall performance is demonstrated by GVNS~I, with 24 best solutions and an overall average computational time of 27.31~s. 
\end{itemize}

\subsubsection{Results for large-sized instances}

Tables~\ref{tab-8},~\ref{tab-9} describe the performance of GVNS~I, GVNS~II, and GRASP for large-sized instances. They have the same structure as Table~\ref{tab-7}, and the best results are indicated in bold face. Note that no mathematical formulation is able to obtain a feasible solution within 3600~s.

The following observations can be made:

\begin{itemize}
\item For $n=250$ : Overall, among the 30 instances, GVNS~I found 27 best solutions (90\%), whereas GVNS~II and GRASP found only 3 (10\%) and 1 (3.33\%), respectively. 
For all instances and for GVNS~I, GVNS~II and GRASP, the difference between $C_{max}^{best} $ and $C_{max}^{avg}$ grows significantly. The average computational time for GVNS~I (resp. GVNS~II and GRASP) is equal to 113~s (resp. 185.64 and 154.10). The best overall performance is demonstrated by GVNS~I with 27 best solutions and an overall average computational time of 113~s. 

\item For $n=500$ : Overall, among the 30 instances, GVNS~I found 29 best solutions (96.66\%), whereas GVNS~II and GRASP found only 1 (3.33\%) and 0 (0\%), respectively. For GVNS~I, GVNS~II and GRASP, the difference between $C_{max}^{best} $ and $C_{max}^{avg}$ grows significantly especially for the instances with $\alpha_3$. The average computational time for GVNS~I (resp. GVNS~II and GRASP) is equal to 99.72~s (resp. 225.52 and 218.88). The best overall performance is demonstrated by GVNS~I with 29 best solutions and an average computational time of 99.72~s.

\end{itemize}
To sum up, for small-sized instances, GVNS~I, GVNS~II and GRASP are able to produce an optimal solution for all instances (among the 120 instances) in significantly less computing time in comparison with $CF^{+}$ and $TIF^{+}$ formulations. For medium and large-sized instances, among the 120 instances, GVNS~I found 105 best solutions (87.50\%), whereas GVNS~II and GRASP found only 43 (35.83\%) and 32 (26.66\%) ones, respectively. This success can be explained by the quality of the initial solution since the iterative improvement procedure contributes significantly to the minimization of the makespan. 

Note that the difference between $C_{max}^{best} $ and $C_{max}^{avg}$ for all methods grows with $n$ and $\alpha$. These results can lead to the indication that the instances with $\alpha_3$ are more difficult to solve in comparison with $\alpha_1$ and $\alpha_2$, which is expected since the loading/unloading times variance is large for $\alpha_3$.

\subsubsection{Discussion}

Table~\ref{tab--} presents the performance of GVNS~I, GVNS~II, and GRASP  in terms of the
percentage deviation from the theoretical lower bound  ($Gap_{LB_T}(\%)$) according to the number of jobs and the  loading/unloading variance coefficient. In order to compute each percentage deviation, two values are compared: the value of $LB_{T}$, and the value of the best approximate makespan obtained over 10 runs by the considered method (see Equation~\ref{GAPLBT}). 

\begin{equation}
\label{GAPLBT}
Gap_{LB_T}(\%) = 100 \times \frac{C_{max}^{best} - LB_T}{LB_T}
\end{equation}

\begin{table}[!h]
\centering
\scriptsize
\caption{Comparison of GVNS~I, GVNS~II, and GRASP in terms of percentage deviation from the theoretical lower bound for $n \in \{8, 10,12,25,50,100,250,500\}$.}
\begin{tabular}{|c|c|c|c|c|}
\hline
\multirow{2}[4]{*}{$n$} & \multirow{2}[4]{*}{$\alpha$} & \multicolumn{3}{c|}{$Gap_{LB_{T}}(\%)$} \\
\cline{3-5}      &   & GVNSI & GVNSI & GRASP \\
    \hline
    8 & $\alpha_1$ & 0.29 & 0.29 & 0.29 \\
      & $\alpha_2$ & 0.34 & 0.34 & 0.34 \\
      & $\alpha_3$ & 0.66 & 0.66 & 0.66 \\
    \hline
    10 & $\alpha_1$ & 0.07 & 0.07 & 0.07 \\
      & $\alpha_2$ & 0.19 & 0.19 & 0.19 \\
      & $\alpha_3$ & 0.27 & 0.27 & 0.27 \\
    \hline
    12 & $\alpha_1$ & 0.05 & 0.05 & 0.05 \\
      & $\alpha_2$ & 0.16 & 0.16 & 0.16 \\
      & $\alpha_3$ & 0.13 & 0.13 & 0.13 \\
    \hline
    25 & $\alpha_1$ & 0.02 & 0.02 & 0.02 \\
      & $\alpha_2$ & 0.02 & 0.02 & 0.02 \\
      & $\alpha_3$ & 0.54 & 0.39 & 0.57 \\
    \hline
    50 & $\alpha_1$ & 0.02 & 0.02 & 0.02 \\
      & $\alpha_2$ & 0.02 & 0.01 & 0.02 \\
      & $\alpha_3$ & 1.50 & 1.57 & 1.80 \\
    \hline
    100 & $\alpha_1$ & 0.01 & 0.01 & 0.01 \\
      & $\alpha_2$ & 0.14 & 0.15 & 0.17 \\
      & $\alpha_3$ & 2.91 & 3.80 & 3.68 \\
    \hline
    250 & $\alpha_1$ & 0.03 & 0.06 & 0.09 \\
      & $\alpha_2$ & 0.35 & 0.82 & 0.85 \\
      & $\alpha_3$ & 3.70 & 5.73 & 6.22 \\
    \hline
    500 & $\alpha_1$ & 0.02 & 0.23 & 0.21 \\
      & $\alpha_2$ & 0.31 & 1.37 & 1.34 \\
      & $\alpha_3$ & 3.66 & 7.07 & 7.31 \\
    \hline
\end{tabular}
\label{tab--}
\end{table}%

The average percentage deviation from the theoretical lower bound of GVNS~I  (respectively GVNS~II and GRASP) for the instances with $\alpha_1$ and $n \in \{8, 10, 12, 25, 50, 100, 250, 500\}$ is equal to 0.06 (respectively 0.09 and 0.10). The average percentage deviation from the theoretical lower bound of GVNS~I  (resp. GVNS~II and GRASP) for the instances with $\alpha_2$ and $n \in \{8, 10, 12, 25, 50, 100, 250, 500\}$ is equal to 0.19 (respectively 0.38\% and 0.39\%). Finally, the average percentage deviation from the theoretical lower bound of GVNS~I  (respectively GVNS~II and GRASP) for $\alpha_3$ and $n \in \{8, 10, 12, 25, 50, 100, 250, 500\}$ is equal to 1.67\% (resp. 2.45\% and 2.58\%). 
One can see that the average percentage deviation from the theoretical lower bound increases with the increase of the loading/unloading times variance. Indeed, as shown in the preceding section, the instances with a large loading/unloading times variance are more difficult to solve than the other instances (we recall that the difference between $C_{max}^{best}$ and $C_{max}^{avg}$ is very small for $\alpha_1$ and $\alpha_2$). The total average percentage deviation from the theoretical lower bound of GVNS~I  (respectively GVNS~II and GRASP) for all instances is equal to 0.642\% (respectively 0.98\% and 1.02\%). To sum up, we can observe the superiority of GVNS~I over GVNS~II and GRASP.

Moreover, Table~\ref{tab--0} summarizes the performance of GVNS~I, GVNS~II, and GRASP in terms of the percentage deviation from the best-known solution
(the best one over all the runs of all the metaheuristics, and the one obtained by the considered metaheuristic) according to $n$ and $\alpha$. For each metaheuristic,
the following features are given: the minimum value of the percentage deviation over all instance', Min; the average value of the percentage deviation over
all instance', Avg; and the maximum value of the percentage deviation over all instance', Max. The last line of the table shows total average results. The results show again that GVNS~I based on the iterative improvement procedure as  initial-solution finding mechanism, on average, yielded a superior performance in terms of minimum, average and maximum gaps when compared to  GVNS~II, and GRASP.

\begin{table}[!h]
\centering
\scriptsize
\caption{Percentage deviation from the best known solution for GVNS~I, GVNS~II, and GRASP.}
\begin{tabular}{|cc|ccc|ccc|ccc|}
\hline
\multicolumn{1}{|c|}{\multirow{2}[6]{*}{n}} & \multirow{2}[6]{*}{$\alpha$} & \multicolumn{9}{c|}{$Gap_{Dev}(\%)$} \\
\cline{3-11}    \multicolumn{1}{|c|}{} &   & \multicolumn{3}{c|}{GVNS~I} & \multicolumn{3}{c|}{GVNS~II} & \multicolumn{3}{c|}{GRASP} \\
\cline{3-11}    \multicolumn{1}{|c|}{} &   & Min. & Avg.  & Max. & Min. & Avg.  & Max. & Min. & Avg.  & Max. \\
    \hline
    \multicolumn{1}{|c|}{8} & $\alpha_1$ & 0.00 & 0.00 & 0.00 & 0.00 & 0.00 & 0.00 & 0.00 & 0.00 & 0.00 \\
    \multicolumn{1}{|c|}{} & $\alpha_2$ & 0.00 & 0.00 & 0.00 & 0.00 & 0.00 & 0.00 & 0.00 & 0.00 & 0.00 \\
    \multicolumn{1}{|c|}{} & $\alpha_3$ & 0.00 & 0.00 & 0.00 & 0.00 & 0.00 & 0.00 & 0.00 & 0.00 & 0.00 \\
    \hline
    \multicolumn{1}{|c|}{10} & $\alpha_1$ & 0.00 & 0.00 & 0.00 & 0.00 & 0.00 & 0.00 & 0.00 & 0.00 & 0.00 \\
    \multicolumn{1}{|c|}{} & $\alpha_2$ & 0.00 & 0.00 & 0.00 & 0.00 & 0.00 & 0.00 & 0.00 & 0.00 & 0.00 \\
    \multicolumn{1}{|c|}{} & $\alpha_3$ & 0.00 & 0.00 & 0.00 & 0.00 & 0.00 & 0.00 & 0.00 & 0.00 & 0.00 \\
    \hline
    \multicolumn{1}{|c|}{12} & $\alpha_1$ & 0.00 & 0.00 & 0.00 & 0.00 & 0.00 & 0.00 & 0.00 & 0.00 & 0.00 \\
    \multicolumn{1}{|c|}{} & $\alpha_2$ & 0.00 & 0.00 & 0.00 & 0.00 & 0.00 & 0.00 & 0.00 & 0.00 & 0.00 \\
    \multicolumn{1}{|c|}{} & $\alpha_3$ & 0.00 & 0.00 & 0.00 & 0.00 & 0.00 & 0.00 & 0.00 & 0.00 & 0.00 \\
    \hline
    \multicolumn{1}{|c|}{25} & $\alpha_1$ & 0.00 & 0.00 & 0.00 & 0.00 & 0.00 & 0.00 & 0.00 & 0.00 & 0.00 \\
    \multicolumn{1}{|c|}{} & $\alpha_2$ & 0.00 & 0.00 & 0.00 & 0.00 & 0.00 & 0.00 & 0.00 & 0.00 & 0.00 \\
    \multicolumn{1}{|c|}{} & $\alpha_3$ & 0.00 & 0.24 & 0.72 & 0.00 & 0.10 & 0.48 & 0.00 & 0.27 & 0.99 \\
    \hline
    \multicolumn{1}{|c|}{50} & $\alpha_1$ & 0.00 & 0.00 & 0.00 & 0.00 & 0.00 & 0.00 & 0.00 & 0.00 & 0.00 \\
    \multicolumn{1}{|c|}{} & $\alpha_2$ & 0.00 & 0.01 & 0.06 & 0.00 & 0.00 & 0.00 & 0.00 & 0.01 & 0.06 \\
    \multicolumn{1}{|c|}{} & $\alpha_3$ & 0.00 & 0.13 & 0.61 & 0.00 & 0.20 & 0.66 & 0.00 & 0.43 & 1.26 \\
    \hline
    \multicolumn{1}{|c|}{100} & $\alpha_1$ & 0.00 & 0.00 & 0.00 & 0.00 & 0.00 & 0.00 & 0.00 & 0.00 & 0.00 \\
    \multicolumn{1}{|c|}{} & $\alpha_2$ & 0.00 & 0.04 & 0.15 & 0.00 & 0.05 & 0.17 & 0.00 & 0.08 & 0.18 \\
    \multicolumn{1}{|c|}{} & $\alpha_3$ & 0.00 & 0.00 & 0.00 & 0.00 & 0.87 & 1.47 & 0.16 & 0.75 & 1.29 \\
    \hline
    \multicolumn{1}{|c|}{250} & $\alpha_1$ & 0.00 & 0.00 & 0.01 & 0.00 & 0.03 & 0.06 & 0.04 & 0.06 & 0.13 \\
    \multicolumn{1}{|c|}{} & $\alpha_2$ & 0.00 & 0.02 & 0.15 & 0.22 & 0.48 & 0.74 & 0.00 & 0.51 & 0.84 \\
    \multicolumn{1}{|c|}{} & $\alpha_3$ & 0.00 & 0.00 & 0.04 & 0.00 & 1.97 & 3.45 & 0.91 & 2.44 & 3.94 \\
    \hline
    \multicolumn{1}{|c|}{500} & $\alpha_1$ & 0.00 & 0.00 & 0.00 & 0.14 & 0.21 & 0.29 & 0.11 & 0.19 & 0.25\\
    \multicolumn{1}{|c|}{} & $\alpha_2$ & 0.00 & 0.00 & 0.00 & 0.77 & 1.06 & 1.40 & 0.63 & 1.02 & 1.27 \\
    \multicolumn{1}{|c|}{} & $\alpha_3$ & 0.00 & 0.10 & 0.96 & 0.00 & 3.41 & 4.32 & 1.83 & 3.64 & 4.80 \\
    \hline
    \multicolumn{2}{|c|}{\textbf{Avg.}} & \textbf{0.00} & \textbf{0.02} & \textbf{0.11} & \textbf{0.05} & \textbf{0.35} & \textbf{0.54} & \textbf{0.15} & \textbf{0.39} & \textbf{0.63} \\
\hline
\end{tabular}
\label{tab--0}
\end{table}

\section{Managerial insights} \label{sec7}

In this section, some managerial insights are presented regarding the investigated problem $P2,S1|s_j,t_j|C_{max}$. We propose to compare our results with the ones of \cite{benmansour2021scheduling}. Indeed, \cite{benmansour2021scheduling} suggested a MILP formulation and a GVNS metaheuristic for the problem $P2,S2|s_j,t_j|C_{max}$ involving two dedicated servers : one for the loading operations and one for the unloading operations. Indeed, in the problem $P2,S2|s_j,t_j|C_{max}$, each job has to be loaded by a dedicated (loading) server and unloaded by a dedicated (unloading) server, respectively, immediately before and after being
processed on one of the two machines, while in the problem $P2,S1|s_j,t_j|C_{max}$, only one resource (server) is in charge of both the loading and unloading operations. The objective of this section is to show the impact of removing the unloading server on the makespan (i.e., the single server will be in charge of both the loading and unloading operations).

We propose first to improve the mathematical formulation  proposed in \cite{benmansour2021scheduling} denoted by $MIP_{2S}$ by adding the two valid inequalities proposed in Section~\ref{valin}. The new obtained formulation is denoted by $MIP_{2S}^{+}$. Therefore, the formulations $MIP_{2S}$ and $MIP_{2S}^{+}$ are compared with the formulations $CF^{+}$ and $TIF^{+}$ (since they presented a better performance than $CF$ and $TIF$). In total, four mathematical formulations are obtained and compared: two formulations regarding the problem $P2,S1|s_j,t_j|C_{max}$ with one single server ($CF^{+}$ and $TIF^{+}$), and two formulations regarding the problem $P2,S2|s_j,t_j|C_{max}$ with two dedicated servers ($MIP_{2S}$ and $MIP_{2S}^{+}$).  The computational experiments were conducted using the same computer as described in Section~\ref{sec6}. In addition, the time limit for solving the formulations $CF^{+}$, $TIF^{+}$, $MIP_{2S}$, and $MIP_{2S}^{+}$ was set to 3600~s. Note that the four formulations are compared using the same instances as presented in Section~\ref{inst} with up to 25 jobs (since we can obtain a proof of optimality for at least one formulation for each problem). To solve the $CF^{+}$, $TIF^{+}$, $MIP_{2S}$, and $MIP_{2S}^{+}$ formulations, we have used the Concert Technology library of the CPLEX 12.6 version with default settings in C++.

In Table~\ref{tabM1}, we compare the performance of $CV^{+}$, $TIF^{+}$, $MIP_{2S}$ and $MIP_{2S}^{*}$ for $n=8$. First, each instance is characterized by the following information: the ID; the number $n$ of jobs; the loading/unloading variance coefficient ($\alpha_1$, $\alpha_2$, and $\alpha_3$). Next, for each mathematical formulation, the following features are given: $i)$ the optimal makespan solution ($C_{max}^{*}$) and  $ii)$ the time required to find an optimal solution (CPU). Finally, the gap between the optimal makespan of the problem with one server and the optimal makespan of the problem with two servers denoted as $Gap_{MI}(\%)$ (calculated in  Equation~\ref{GAPML}) are presented:  

\begin{equation}
\label{GAPML}
Gap_{MI}(\%) = 100 \times \frac{C_{max}^{*}(1 server) - C_{max}^{*}(2 servers)}{C_{max}^{*}(2 servers)}
\end{equation}

The following observations can be made:

\begin{itemize}
\item Based on the formulations $CV^{+}$, $TIF^{+}$, $MIP_{2S}$ and $MIP_{2S}^{*}$, CPLEX is able to find an optimal solution for any instance. The average computing time for $CV^{+}$ is 0.71 s, and the average computational time for $MIP_{2S}^{*}$ is 0.27 s.
\item All formulations are able to find the same optimal solution for 5 instances (I1, I4, I5, I7 and I24). 
\item The average value of $Gap_{MI}(\%)$ is equal to 0.31\% for $\alpha_1$, it is equal to 0.76\% for $\alpha_2$, and it is equal to 1.45\% for $\alpha_3$. Therefore, the value $Gap_{MI}(\%)$ increases for a large loading/unloading times variance.
\item The overall average value of $Gap_{MI}(\%)$ over the 30 instances is equal to 0.84\%. It can be noticed that the gap between the two problems is very small and for 5 instances,  the same optimal solution is obtained with only one single server. 
\end{itemize}

\begin{table}[htbp]
\caption{Comparison of $CF^{+}$, $TIF^{+}$, $MIP_{2S}^{+}$  with $MIP_{2S}$ by \cite{benmansour2021scheduling} for $n=8$.}
\centering
\vspace*{3mm}
\renewcommand{\tabcolsep}{1.5pt}
{\scriptsize
\renewcommand{\arraystretch}{0.8}
\begin{tabular}{|ccc|cc|cc|cc|cc|c|}
\hline
     & \multicolumn{1}{c}{Instance} &   & \multicolumn{4}{c|}{2 servers} & \multicolumn{4}{c|}{1 server} &  \\
\cline{4-11}      & \multicolumn{1}{r}{} &   & \multicolumn{2}{c|}{$MIP_{2S}$} & \multicolumn{2}{c|}{$MIP_{2S}^{+}$} & \multicolumn{2}{c|}{$CF^{+}$} & \multicolumn{2}{c|}{$TIF^{+}$} & $Gap_{MI}(\%)$ \\
\cline{1-11}    \multicolumn{1}{|c|}{ID} & \multicolumn{1}{c|}{$n$} & \multicolumn{1}{c|}{$\alpha$} &$ C_{max}^{*}$ & CPU & $C_{max}^{*}$ & CPU & $C_{max}^{*}$ & CPU & $C_{max}^{*}$ & CPU &  \\
\hline
    \multicolumn{1}{|c|}{I1} & \multicolumn{1}{r|}{} &   & 295 & 2.98 & 295 & 0.33 & 295 & 0.46 & 295 & 4.11 & 0 \\
    \multicolumn{1}{|c|}{I2} & \multicolumn{1}{c|}{} &   & 288 & 2.19 & 288 & 0.29 & 289 & 0.61 & 289 & 5.49 & 0.35 \\
    \multicolumn{1}{|c|}{I3} & \multicolumn{1}{c|}{} &   & 258 & 2.66 & 258 & 0.31 & 259 & 0.53 & 259 & 4.27 & 0.39 \\
    \multicolumn{1}{|c|}{I4} & \multicolumn{1}{c|}{} &   & 217 & 2.22 & 217 & 0.36 & 217 & 0.42 & 217 & 2.82 & 0 \\
    \multicolumn{1}{|c|}{I5} & \multicolumn{1}{c|}{8} & \multicolumn{1}{c|}{$\alpha_{1}$} & 237 & 2.84 & 237 & 0.28 & 237 & 0.35 & 237 & 2.95 & 0 \\
    \multicolumn{1}{|c|}{I6} & \multicolumn{1}{c|}{} &   & 236 & 3.20 & 236 & 0.30 & 238 & 0.75 & 238 & 3.03 & 0.85 \\
    \multicolumn{1}{|c|}{I7} & \multicolumn{1}{c|}{} &   & 218 & 4.10 & 218 & 0.26 & 218 & 0.31 & 218 & 2.95 & 0 \\
    \multicolumn{1}{|c|}{I8} & \multicolumn{1}{c|}{} &   & 229 & 7.32 & 229 & 0.26 & 230 & 0.63 & 230 & 3.06 & 0.44 \\
    \multicolumn{1}{|c|}{I9} & \multicolumn{1}{c|}{} &   & 192 & 1.48 & 192 & 0.34 & 193 & 0.73 & 193 & 1.94 & 0.52 \\
    \multicolumn{1}{|c|}{I10} & \multicolumn{1}{c|}{} &   & 195 & 1.82 & 195 & 0.45 & 196 & 0.56 & 196 & 1.94 & 0.51 \\
    \hline
    \multicolumn{1}{|c|}{I11} & \multicolumn{1}{r|}{} &   & 376 & 11.79 & 376 & 0.23 & 383 & 1.05 & 383 & 9.91 & 1.86 \\
    \multicolumn{1}{|c|}{I12} & \multicolumn{1}{c|}{} &   & 276 & 1.05 & 276 & 0.23 & 277 & 0.54 & 277 & 5.43 & 0.36 \\
    \multicolumn{1}{|c|}{I13} & \multicolumn{1}{c|}{} &   & 275 & 2.82 & 275 & 0.32 & 276 & 0.56 & 276 & 6.06 & 0.36 \\
    \multicolumn{1}{|c|}{I14} & \multicolumn{1}{c|}{} &   & 325 & 1.49 & 325 & 0.27 & 328 & 0.73 & 328 & 5.67 & 0.92 \\
    \multicolumn{1}{|c|}{I15} & \multicolumn{1}{c|}{8} & \multicolumn{1}{c|}{$\alpha_{2}$} & 259 & 1.14 & 259 & 0.24 & 260 & 0.52 & 260 & 3.85 & 0.39 \\
    \multicolumn{1}{|c|}{I16} & \multicolumn{1}{c|}{} &   & 310 & 8.60 & 310 & 0.26 & 312 & 0.85 & 312 & 6.80 & 0.65 \\
    \multicolumn{1}{|c|}{I17} & \multicolumn{1}{c|}{} &   & 319 & 2.19 & 319 & 0.23 & 320 & 0.64 & 320 & 10.52 & 0.31 \\
    \multicolumn{1}{|c|}{I18} & \multicolumn{1}{c|}{} &   & 285 & 1.06 & 285 & 0.25 & 286 & 0.83 & 286 & 5.88 & 0.35 \\
    \multicolumn{1}{|c|}{I19} & \multicolumn{1}{c|}{} &   & 333 & 3.03 & 333 & 0.21 & 336 & 0.85 & 336 & 7.46 & 0.90 \\
    \multicolumn{1}{|c|}{I20} & \multicolumn{1}{c|}{} &   & 344 & 1.96 & 344 & 0.25 & 349 & 0.81 & 349 & 9.33 & 1.45 \\
    \hline
    \multicolumn{1}{|c|}{I21} & \multicolumn{1}{r|}{} &   & 321 & 2.27 & 321 & 0.26 & 325 & 0.98 & 325 & 8.03 & 1.25 \\
    \multicolumn{1}{|c|}{I22} & \multicolumn{1}{c|}{} &   & 397 & 4.82 & 397 & 0.30 & 408 & 1.08 & 408 & 10.12 & 2.77 \\
    \multicolumn{1}{|c|}{I23} & \multicolumn{1}{c|}{} &   & 321 & 5.72 & 321 & 0.24 & 325 & 0.99 & 325 & 15.21 & 1.25 \\
    \multicolumn{1}{|c|}{I24} & \multicolumn{1}{c|}{} &   & 248 & 2.07 & 248 & 0.22 & 248 & 0.63 & 248 & 3.62 & 0 \\
    \multicolumn{1}{|c|}{I25} & \multicolumn{1}{c|}{8} & \multicolumn{1}{c|}{$\alpha_{3}$} & 345 & 1.78 & 345 & 0.24 & 352 & 0.84 & 352 & 15.51 & 2.03 \\
    \multicolumn{1}{|c|}{I26} & \multicolumn{1}{c|}{} &   & 329 & 1.53 & 329 & 0.27 & 335 & 0.77 & 335 & 13.74 & 1.82 \\
    \multicolumn{1}{|c|}{I27} & \multicolumn{1}{c|}{} &   & 262 & 1.69 & 262 & 0.29 & 266 & 0.81 & 266 & 3.30 & 1.53 \\
    \multicolumn{1}{|c|}{I28} & \multicolumn{1}{c|}{} &   & 297 & 1.03 & 297 & 0.27 & 300 & 0.86 & 300 & 25.17 & 1.01 \\
    \multicolumn{1}{|c|}{I29} & \multicolumn{1}{c|}{} &   & 381 & 1.98 & 381 & 0.31 & 387 & 0.83 & 387 & 9.25 & 1.57 \\
    \multicolumn{1}{|c|}{I30} & \multicolumn{1}{c|}{} &   & 327 & 2.43 & 327 & 0.19 & 331 & 0.90 & 331 & 8.18 & 1.22 \\
    \hline
      & Avg. &   & 289.83 & 3.04 & 289.83 & 0.27 & 292.53 & 0.71 & 292.53 & 7.19 & 0.84 \\
\hline
\end{tabular}
}
\label{tabM1}
\end{table}

In Tables~\ref{tabM2},~\ref{tabM3},~\ref{tabM4}, we compare the performance of  $MIP_{2S}$, $MIP_{2S}^{+}$, $CF^{+}$, $TIF^{+}$, for $n \in \{10,12,25\}$. First, each instance is characterized by the following information: the ID; the number $n$ of jobs; the loading/unloading times variance coefficient ($\alpha_1$, $\alpha_2$, and $\alpha_3$). Then for each formulation, the following information is given: the upper bound ($UB_{MIP_{2S}}$, $UB_{MIP_{2S}^{+}}$, $UB_{{CF}^{+}}$, $UB_{{TIF}^{+}}$), the lower bound ($LB_{MIP_{2S}}$, $LB_{MIP_{2S}^{+}}$, $LB_{{CF}^{+}}$, $LB_{{TIF}^{+}}$), the percentage gap to optimality ($Gap_{MIP_{2S}}(\%)$, $Gap_{MIP_{2S}^{+}}(\%)$, $Gap_{{CF}^{+}}(\%)$, $Gap_{{TIF}^{+}}(\%)$), and the time required to prove optimality (CPU). Finally, the gap between the optimal makespan of the problem with one server and the optimal makespan of the problem with two servers $Gap_{MI}(\%)$ is given. (In Table~\ref{tabM4}, $Gap_{MI}(\%)$ is not reported since CPLEX is not able to find an optimal solution for any formulation). 

The following observations can be made:

\begin{itemize}
\item For $n=10$: Based on the formulations $MIP_{2S}^{+}$, $CF^{+}$, $TIF^{+}$, CPLEX is able to find an optimal solution for any instance.  Based on the formulation $MIP_{2S}$, CPLEX is able to produce an optimal solution only for 17 instances among the 30 ones. It can be noted that for the improved formulation $MIP_{2S}^{+}$, CPLEX is able to produce optimal solutions in less computational time in comparison with the original one. The average CPU time for $MIP_{2S}^{+}$ is equal to 0.33~s, whereas the average CPU time for $CF^{+}$ is equal to 19.12~s.  The overall average value of $Gap_{MI}(\%)$ over the 30 instances is equal to 0.62\%. 
\item For $n=12$ : Based on the formulations $MIP_{2S}^{+}$ and $TIF^{+}$, CPLEX is able to find an optimal solution for any instance.  Based on the formulation $CF^{+}$, CPLEX is able to produce an optimal solution only for 23 instances among the 30 ones. In addition, CPLEX is not able to produce an optimal solution solution for any instance using the formulation $MIP_{2S}$. It can be noted that for the improved formulation $MIP_{2S}^{+}$, CPLEX is able to produce optimal solutions in less computational time in comparison with the original one. The average CPU time for $MIP_{2S}^{+}$ is equal to 0.37~s, whereas the average CPU time for $TIF^{+}$ is equal to 69.02~s.  The overall average value of $Gap_{MI}(\%)$ over the 30 instances is equal to 0.61\%.  
\item For $n=25$ : In the case  of 1 server, $TIF^{+}$ is the only formulation for which CPLEX is able to produce an optimal solution for 7 instances among the 30 ones. In the case of 2 servers, $MIP_{2S}^{+}$ is the only formulation for which CPLEX is able to produce an optimal solution for 20 instances among the 30 ones. For $\alpha_3$, using the formulations $MIP_{2S}$ and $MIP_{2S}^{+}$, CPLEX is not able to produce a feasible solution for all instances (except 1 instance I118).  The overall average value of $Gap_{MI}(\%)$ for the instances that can be solved to optimality using $MIP_{2S}^{+}$ and $TIF^{+}$ (I92, I93, I95, I97, I99, I100, and I103) is equal to 0.13\%.
\end{itemize}

As it was shown in the previous Table~\ref{tabM1} and in the Tables~\ref{tabM2},~\ref{tabM3},~\ref{tabM4} in the Appendix, the gap between the optimal makespan of the problem with one server and the optimal makespan of the problem with two servers ($Gap_{MI}(\%)$) is very small for almost all small-sized instances with $n \in \{8,10,12,25\}$. Therefore, the use of an extra resource (unloading server) has a small impact on the reduction of the makespan. It can be noticed that in some cases, the same optimal makespan is obtained using only one single server. Indeed, since the loading, processing and unloading operations are non separable, the waiting time of a machine due to the unavailability of the unloading server will always be significant. To sum up, from a managerial point of view, we indicate to use only one single server for both the loading and unloading operations, which can lead to a significant reduction of costs (e.g., electricity and maintenance) and hence a better preservation of environment.

\section{Conclusions} \label{sec8}

In this paper, the scheduling problem with two identical parallel machines and a single server in charge of loading and unloading operations of jobs was addressed. Each job has to be loaded by a single server immediately before its processing. Once the processing operations finished, each job has to be immediately unloaded by the same server. The objective function considered was 
the minimization of the makespan. We presented two mixed-integer linear programming (MILP) formulations to model the problem, namely: a completion-time variables  formulation $CF$ and a time-indexed formulation $TIF$, and we proposed sets of inequalities that can be used to improve the $CF$ formulation. We also proposed some polynomial-time solvable cases and a tight lower bound. In addition, we showed that the minimization of the makespan is equivalent to the minimization of the total idle time of the machines.

Since the mathematical formulations were not able to cope with the majority of instances, an efficient General Variable Neighborhood Search metaheuristic (GVNS) with two  mechanisms for finding an initial solution (one with an iterative improvement procedure (GVNS~I) and one with a random initial solution (GVNS~II)), and a Greedy Randomized Adaptive Search Procedures (GRASP) metaheuristic were designed. To validate the performance of the proposed GVNS~I, GVNS~II and GRASP approaches, exhaustive computational experiments on 240 instances were performed. For small-sized instances with up to 25 jobs, GVNS~I,  GVNS~II and GRASP algorithms outperformed the MILP formulations in terms of
the computational time to find an optimal solution. However, for medium and large-sized instances, the GVNS~I yielded better results than the other approaches. In particular, the average percentage deviation from the theoretical lower bound was equal to 0.642\%. Finally, we presented some managerial insights, and our MILP formulations were compared with the ones of \cite{benmansour2021scheduling} regarding the problem $P2,S2|s_j,p_j|C_{max}$ involving a dedicated loading server and a dedicated unloading server. It turned out that adding an unloading server (extra resource) contributed less to the reduction of the makespan (since the average percentage deviation of the optimal makespan for the two problems is equal to 0.69\% for small-sized instances with up to 12 jobs), and in some cases the same optimal solution can be obtained using only one single server. 
In future research, it would be interesting to adapt the
presented approaches to the more general problem $P,S1|s_j,t_j|C_{max}$ involving an arbitrary number of machines. Further works could also measure the effect of the unloading server on the makespan for the problem $P,S2|s_j,t_j|C_{max}$.




\bibliography{mybibfile}

\newpage
\appendix

\section{Detailed results}

\begin{table}[!h]
\caption{Comparison of $CF^{+}$, $TIF^{+}$, GVNS~I, GVNS~II and GRASP for $n=50$.}
\centering
\vspace*{3mm}
\renewcommand{\tabcolsep}{0.85pt}
{\scriptsize
\renewcommand{\arraystretch}{0.70}
\begin{tabular}{|c|c|c|c|cccc|cccc|ccc|ccc|ccc|}
\hline
\multicolumn{4}{|c|}{Instance} & \multicolumn{4}{c|}{$CF^{+}$} & \multicolumn{4}{c|}{$TIF^{+}$} & \multicolumn{3}{c|}{GVNS I} & \multicolumn{3}{c|}{GVNS II} & \multicolumn{3}{c|}{GRASP} \\
\hline
    ID & $n$ & $\alpha$ & $LB_{T}$ & $UB_{CF^{+}}$ & $LB_{CF^{+}}$ & $Gap_{CF^{+}}(\%)$ & CPU & $UB_{TIF^{+}}$ & $LB_{TIF^{+}}$ & $Gap_{TIF^{+}}(\%)$ & CPU & $C_{max}^{best}$ & $C_{max}^{avg}$ & CPU & $C_{max}^{best}$ & $C_{max}^{avg}$ & CPU & $C_{max}^{best}$ & $C_{max}^{avg}$ & CPU \\
\hline
    I121 &   &   & 1742.5 & 1745 & 1741.5 & 0.20 & 3600 & 1764 & 911.2 & 48.35 & 3600 & \textbf{1743} & 1743 & 0.19 & \textbf{1743} & 1743 & 0.38 & \textbf{1743} & 1743 & 0.19 \\
    I122 &   &   & 1558 & 1563 & 1557 & 0.38 & 3600 & 1582 & 874.8 & 44.70 & 3600 & \textbf{1558} & 1558 & 0.15 & \textbf{1558} & 1558 & 0.66 & \textbf{1558} & 1558 & 0.39 \\
    I123 &   &   & 1659.5 & 1662 & 1658.5 & 0.21 & 3600 & 1683 & 869 & 48.37 & 3600 & \textbf{1660} & 1660 & 0.12 & \textbf{1660} & 1660 & 0.35 & \textbf{1660} & 1660 & 0.36 \\
    I124 &   &   & 1453 & 1454 & 1452 & 0.14 & 3600 & 1482 & 815.2 & 45.00 & 3600 & \textbf{1453} & 1453 & 0.30 & \textbf{1453} & 1453 & 0.42 & \textbf{1453} & 1453 & 0.29 \\
    I125 & 50 & $\alpha_1$ & 1627 & 1634 & 1626 & 0.49 & 3600 & 1649 & 851.6 & 48.36 & 3600 & \textbf{1627} & 1627 & 0.72 & \textbf{1627} & 1627 & 0.64 & \textbf{1627} & 1627 & 1.25 \\
    I126 &   &   & 1664.5 & 1671 & 1663.5 & 0.45 & 3600 & 1226628 & 870.7 & 99.93 & 3600 & \textbf{1665} & 1665 & 0.22 & \textbf{1665} & 1665 & 0.27 & \textbf{1665} & 1665 & 0.20 \\
    I127 &   &   & 1527 & 1528 & 1526 & 0.13 & 3600 & 1553 & 870.3 & 43.96 & 3600 & \textbf{1527} & 1527 & 0.18 & \textbf{1527} & 1527 & 0.58 & \textbf{1527} & 1527 & 0.50 \\
    I128 &   &   & 1503.5 & 1504 & 1502.5 & 0.10 & 3600 & 1537 & 857.6 & 44.20 & 3600 & \textbf{1504} & 1504 & 0.07 & \textbf{1504} & 1504 & 0.13 & \textbf{1504} & 1504 & 0.18 \\
    I129 &   &   & 1698 & 1699 & 1697 & 0.12 & 3600 & 1735 & 889 & 48.76 & 3600 & \textbf{1698} & 1698 & 1.96 & \textbf{1698} & 1698 & 3.23 & \textbf{1698} & 1698 & 2.87 \\
    I130 &   &   & 1724.5 & 1730 & 1723.5 & 0.38 & 3600 & 1759 & 902.5 & 48.69 & 3600 & \textbf{1725} & 1725 & 0.08 & \textbf{1725} & 1725 & 0.26 & \textbf{1725} & 1725 & 0.35 \\
    \hline
    I131 &   &   & 1820.5 & 1825 & 1819 & 0.33 & 3600 & 1843 & 954.2 & 48.23 & 3600 & \textbf{1821} & 1821 & 22.47 & \textbf{1821} & 1821 & 23.73 & \textbf{1821} & 1821.3 & 33.17 \\
    I132 &   &   & 1980.5 & 1998 & 1979.5 & 0.93 & 3600 & 1841388 & 1038.1 & 99.94 & 3600 & \textbf{1981} & 1981.9 & 26.68 & \textbf{1981} & 1981.8 & 33.72 & \textbf{1981} & 1982 & 36.83 \\
    I133 &   &   & 1636 & 1637 & 1635 & 0.12 & 3600 & 1680 & 861 & 48.75 & 3600 & \textbf{1636} & 1636.6 & 25.43 & \textbf{1636} & 1636.9 & 23.33 & \textbf{1636} & 1636.8 & 45.90 \\
    I134 &   &   & 1724.5 & 1732 & 1723.5 & 0.49 & 3600 & 1749 & 906 & 48.20 & 3600 & \textbf{1725} & 1725.2 & 41.57 & \textbf{1725} & 1725.6 & 31.09 & \textbf{1725} & 1725.8 & 42.43 \\
    I135 & 50 & $\alpha_2$ & 1722 & 1739 & 1721 & 1.04 & 3600 & 1761 & 904 & 48.67 & 3600 & \textbf{1722} & 1722.8 & 44.52 & \textbf{1722} & 1723 & 38.90 & 1723 & 1723.1 & 32.88 \\
    I136 &   &   & 1619 & 1621 & 1618 & 0.19 & 3600 & 1644 & 850.4 & 48.27 & 3600 & 1620 & 1620 & 20.20 & \textbf{1619} & 1619.7 & 33.83 & 1620 & 1620.1 & 24.41 \\
    I137 &   &   & 1749 & 1761 & 1748 & 0.74 & 3600 & 1393159 & 918.5 & 99.93 & 3600 & \textbf{1749} & 1749.5 & 44.27 & \textbf{1749} & 1749.4 & 20.54 & \textbf{1749} & 1749.4 & 34.38 \\
    I138 &   &   & 1754.5 & 1768 & 1752.5 & 0.88 & 3600 & 1788 & 920 & 48.55 & 3600 & \textbf{1755} & 1755.1 & 31.69 & \textbf{1755} & 1755 & 13.08 & \textbf{1755} & 1755.2 & 34.82 \\
    I139 &   &   & 1462 & 1474 & 1461 & 0.88 & 3600 & 1513 & 771 & 49.04 & 3600 & \textbf{1462} & 1462.6 & 22.75 & \textbf{1462} & 1462.5 & 26.40 & \textbf{1462} & 1462.5 & 41.41 \\
    I140 &   &   & 1994 & 2001 & 1993 & 0.40 & 3600 & 2027 & 1044.4 & 48.48 & 3600 & \textbf{1994} & 1995.1 & 24.91 & \textbf{1994} & 1994.7 & 33.26 & \textbf{1994} & 1994.5 & 41.06 \\
    \hline
    I141 &   &   & 2031 & 2090 & 2030 & 2.87 & 3600 & 2094 & 1068.1 & 48.99 & 3600 & \textbf{2058} & 2068.1 & 25.90 & 2063 & 2070 & 48.70 & 2064 & 2069.6 & 61.96 \\
    I142 &   &   & 1929.5 & 1977 & 1928.5 & 2.45 & 3600 & 1777900 & 1015.2 & 99.94 & 3600 & 1955 & 1967.9 & 42.96 & \textbf{1953} & 1965 & 32.29 & 1958 & 1966 & 52.15 \\
    I143 &   &   & 2289.5 & 2389 & 2288 & 4.23 & 3600 & 2341 & 1200.3 & 48.73 & 3600 & \textbf{2324} & 2352.8 & 53.18 & 2328 & 2347.6 & 50.41 & 2344 & 2352.4 & 53.89 \\
    I144 &   &   & 2127.5 & 2214 & 2126 & 3.97 & 3600 & 2452 & 1118.3 & 54.39 & 3600 & 2159 & 2178.4 & 46.19 & \textbf{2146} & 2173.3 & 43.45 & 2173 & 2181.6 & 39.35 \\
    I145 & 50 & $\alpha_3$ & 2255 & 2322 & 2252.5 & 2.99 & 3600 & 2382 & 1182.3 & 50.37 & 3600 & \textbf{2281} & 2311.7 & 28.00 & 2296 & 2311.4 & 29.97 & 2297 & 2308.3 & 58.64 \\
    I146 &   &   & 1984 & 2034 & 1981.5 & 2.58 & 3600 & 2129 & 1041.4 & 51.08 & 3600 & \textbf{1997} & 2015 & 48.82 & 2003 & 2015.3 & 53.33 & 2008 & 2014.4 & 39.57 \\
    I147 &   &   & 2205 & 2302 & 2202.5 & 4.32 & 3600 & 2443 & 1156.9 & 52.64 & 3600 & \textbf{2242} & 2250.6 & 47.34 & 2245 & 2251.9 & 38.19 & 2244 & 2250.8 & 45.78 \\
    I148 &   &   & 2313.5 & 2375 & 2312.5 & 2.63 & 3600 & 2366 & 1214.4 & 48.67 & 3600 & 2349 & 2365.1 & 39.90 & \textbf{2347} & 2365.9 & 49.14 & 2353 & 2366.1 & 41.94 \\
    I149 &   &   & 2104 & 2176 & 2103 & 3.35 & 3600 & 2145 & 1105.7 & 48.45 & 3600 & \textbf{2147} & 2160.1 & 37.86 & 2149 & 2162.6 & 37.67 & \textbf{2147} & 2160.4 & 70.19 \\
    I150 &   &   & 2196 & 2301 & 2194.5 & 4.63 & 3600 & 2272 & 1152.3 & 49.28 & 3600 & 2247 & 2259.1 & 37.19 & 2243 & 2254.8 & 49.39 & \textbf{2235} & 2257.4 & 62.15 \\
\hline
\end{tabular}
}
\label{tab-6}
\end{table}

\begin{table}[!h]
\caption{Comparison of $CF^{+}$, GVNS~I, GVNS~II and GRASP for $n=100$.}
\centering
\vspace*{3mm}
\renewcommand{\tabcolsep}{1.5pt}
{\scriptsize
\renewcommand{\arraystretch}{0.85}
    \begin{tabular}{|c|c|c|c|cccc|ccc|ccc|ccc|}
    \hline
    \multicolumn{4}{|c|}{Instance} & \multicolumn{4}{c|}{$CF^{+}$} & \multicolumn{3}{c|}{GVNS I} & \multicolumn{3}{c|}{GVNS II} & \multicolumn{3}{c|}{GRASP} \\
\hline
    ID & $n$ & $\alpha$ & $LB_{T}$ & $UB_{CF^{+}}$ & $LB_{CF^{+}}$ & $Gap_{CF^{+}}(\%)$ & CPU & $C_{max}^{best}$ & $C_{max}^{avg}$ & CPU & $C_{max}^{best}$ & $C_{max}^{avg}$ & CPU & $C_{max}^{best}$ & $C_{max}^{avg}$ & CPU \\
    \hline
    I151 &   &   & 3296 & 3309 & 3295 & 0.42 & 3600 & \textbf{3296} & 3296 & 12.27 & \textbf{3296} & 3296 & 14.88 & \textbf{3296} & 3296 & 14.31 \\
    I152 &   &   & 3293.5 & 3325 & 3292.5 & 0.98 & 3600 & \textbf{3294} & 3294 & 3.69 & \textbf{3294} & 3294 & 6.05 & \textbf{3294} & 3294 & 8.22 \\
    I153 &   &   & 3349 & 3399 & 3348 & 1.50 & 3600 & \textbf{3349} & 3349 & 22.51 & \textbf{3349} & 3349 & 21.44 & \textbf{3349} & 3349 & 21.71 \\
    I154 &   &   & 3128 & 3133 & 3127 & 0.19 & 3600 & \textbf{3128} & 3128 & 11.03 & \textbf{3128} & 3128 & 13.00 & \textbf{3128} & 3128 & 13.96 \\
    I155 & 100 & $\alpha_1$ & 2840.5 & 2853 & 2839.5 & 0.47 & 3600 & \textbf{2841} & 2841 & 4.53 & \textbf{2841} & 2841 & 3.89 & \textbf{2841} & 2841 & 5.91 \\
    I156 &   &   & 3105.5 & 3140 & 3104.5 & 1.13 & 3600 & \textbf{3106} & 3106 & 7.56 & \textbf{3106} & 3106 & 5.39 & \textbf{3106} & 3106 & 5.82 \\
    I157 &   &   & 3044 & 3052 & 3043 & 0.29 & 3600 & \textbf{3044} & 3044 & 9.01 & \textbf{3044} & 3044 & 14.66 & \textbf{3044} & 3044 & 13.94 \\
    I158 &   &   & 3167 & 3235 & 3166 & 2.13 & 3600 & \textbf{3167} & 3167.1 & 18.28 & \textbf{3167} & 3167 & 21.30 & \textbf{3167} & 3167 & 11.03 \\
    I159 &   &   & 3126.5 & 3157 & 3125.5 & 1.00 & 3600 & \textbf{3127} & 3127 & 7.01 & \textbf{3127} & 3127 & 5.40 & \textbf{3127} & 3127 & 5.51 \\
    I160 &   &   & 3088 & 3132 & 3087 & 1.44 & 3600 & \textbf{3088} & 3088.1 & 19.97 & \textbf{3088} & 3088 & 20.61 & \textbf{3088} & 3088.1 & 30.69 \\
    \hline
    I161 &   &   & 3580.5 & 3728 & 3579.5 & 3.98 & 3600 & 3586 & 3590.8 & 54.01 & 3587 & 3591.8 & 40.44 & \textbf{3585} & 3591.4 & 48.55 \\
    I162 &   &   & 3639 & 3860 & 3638 & 5.75 & 3600 & \textbf{3644} & 3650.9 & 26.92 & 3646 & 3651 & 46.31 & 3645 & 3650.5 & 51.24 \\
    I163 &   &   & 3212.5 & 3285 & 3211.5 & 2.24 & 3600 & 3220 & 3222.9 & 34.46 & \textbf{3217} & 3222.2 & 59.64 & 3218 & 3223.8 & 33.83 \\
    I164 &   &   & 3372 & 3451 & 3371 & 2.32 & 3600 & 3377 & 3382 & 32.69 & \textbf{3375} & 3382.9 & 54.89 & 3381 & 3385.5 & 40.25 \\
    I165 & 100 & $\alpha_2$ & 3576.5 & 3656 & 3575.5 & 2.20 & 3600 & 3582 & 3588.1 & 37.54 & \textbf{3581} & 3586.9 & 48.74 & 3583 & 3587.5 & 64.07 \\
    I166 &   &   & 3355.5 & 3788 & 3354.5 & 11.44 & 3600 & \textbf{3358} & 3363.9 & 34.64 & 3362 & 3365.3 & 47.19 & 3362 & 3367.2 & 55.01 \\
    I167 &   &   & 3565 & 3666 & 3564 & 2.78 & 3600 & 3572 & 3576.1 & 38.93 & \textbf{3569} & 3575.3 & 36.02 & 3570 & 3577 & 33.46 \\
    I168 &   &   & 3344 & 3448 & 3343 & 3.05 & 3600 & 3351 & 3353.2 & 45.02 & \textbf{3346} & 3352.3 & 60.23 & 3349 & 3352.5 & 58.84 \\
    I169 &   &   & 3627 & 3749 & 3626 & 3.28 & 3600 & \textbf{3629} & 3636.4 & 47.02 & 3635 & 3637.3 & 31.12 & 3635 & 3640.5 & 42.71 \\
    I170 &   &   & 3270 & 3296 & 3269 & 0.82 & 3600 & \textbf{3272} & 3278.9 & 34.82 & 3275 & 3280.4 & 51.30 & 3274 & 3281 & 45.42 \\
    \hline
    I171 &   &   & 4461.5 & * & * & * & * & \textbf{4564} & 4650.7 & 23.31 & 4631 & 4670.7 & 61.91 & 4622 & 4678.2 & 54.00 \\
    I172 &   &   & 4337 & 5254 & 4335.5 & 17.48 & 3600 & \textbf{4460} & 4509.9 & 34.86 & 4508 & 4522.3 & 56.35 & 4498 & 4518.3 & 58.21 \\
    I173 &   &   & 3755.5 & * & * & * & * & \textbf{3855} & 3898.1 & 21.28 & 3871 & 3902 & 44.49 & 3861 & 3911.7 & 47.72 \\
    I174 &   &   & 4390 & 4531 & 4389 & 3.13 & 3600  & \textbf{4540} & 4590.5 & 48.07 & \textbf{4540} & 4583.4 & 46.22 & 4561 & 4590.1 & 59.40 \\
    I175 & 100 & $\alpha_3$ & 4144 & * & * & * & * & \textbf{4260} & 4314 & 20.84 & 4318 & 4340.9 & 40.58 & 4298 & 4342.1 & 50.20 \\
    I176 &   &   & 4608.5 & * & * & * & * & \textbf{4752} & 4801.3 & 29.87 & 4776 & 4812.6 & 49.47 & 4769 & 4815.6 & 39.64 \\
    I177 &   &   & 4411.5 & 5093 & 4410.5 & 13.40 & 3600 & \textbf{4570} & 4616.8 & 32.81 & 4625 & 4645.5 & 44.56 & 4601 & 4650.3 & 54.33 \\
    I178 &   &   & 4384.5 & 4771 & 4383.5 & 8.12 & 3600 & \textbf{4509} & 4565 & 34.10 & 4526 & 4562.2 & 38.27 & 4553 & 4576.6 & 52.03 \\
    I179 &   &   & 4152.5 & 5674 & 4151.5 & 26.83 & 3600 & \textbf{4270} & 4318.6 & 23.29 & 4316 & 4335.9 & 42.11 & 4294 & 4343.2 & 40.96 \\
    I180 &   &   & 3986 & 5105 & 3985 & 21.94 & 3600 & \textbf{4093} & 4159.2 & 49.02 & 4143 & 4164.1 & 42.89 & 4146 & 4177.5 & 60.32 \\
    \hline
\end{tabular}
}
\label{tab-7}
\end{table}

\begin{table}[!h]
\caption{Comparison of GVNS~I, GVNS~II and GRASP for $n=250$.}
\centering
\vspace*{3mm}
\renewcommand{\tabcolsep}{1.5pt}
{\scriptsize
\renewcommand{\arraystretch}{0.80}
\begin{tabular}{|c|c|c|c|ccc|ccc|ccc|}
\hline
\multicolumn{4}{|c|}{Instance} & \multicolumn{3}{c|}{GVNS~I} & \multicolumn{3}{c|}{GVNS~II} & \multicolumn{3}{c|}{GRASP} \\
\hline
ID & $n$ & $\alpha$ & $LB_{T}$ & $C_{max}^{best}$ & $C_{max}^{avg}$ & CPU & $C_{max}^{best}$ & $C_{max}^{avg}$ & CPU & $C_{max}^{best}$ & $C_{max}^{avg}$ & CPU \\
    \hline
    I181 &   &   & 7985 & \textbf{7988} & 7993.9 & 157.51 & \textbf{7988} & 7994.4 & 181.84 & 7991 & 7995.5 & 155.60 \\
    I182 &   &   & 8118 & \textbf{8120} & 8124.7 & 180.11 & 8124 & 8129.2 & 198.61 & 8124 & 8128.4 & 157.28 \\
    I183 &   &   & 8361.5 & \textbf{8364} & 8369.1 & 135.30 & 8366 & 8371.9 & 178.68 & 8370 & 8374.3 & 116.51 \\
    I184 &   &   & 7941.5 & \textbf{7943} & 7949.5 & 151.44 & 7947 & 7953.3 & 133.65 & 7946 & 7953.6 & 121.18 \\
    I185 & 250 & $\alpha_1$ & 7812 & \textbf{7812} & 7820.9 & 139.86 & 7813 & 7821.2 & 188.11 & 7822 & 7825.6 & 133.99 \\
    I186 &   &   & 7925 & 7930 & 7933.7 & 132.45 & \textbf{7929} & 7935.3 & 192.06 & 7933 & 7937 & 179.39 \\
    I187 &   &   & 8768 & \textbf{8769} & 8772.8 & 45.51 & 8773 & 8776.5 & 192.07 & 8774 & 8781.6 & 129.45 \\
    I188 &   &   & 8215.5 & \textbf{8218} & 8224.1 & 139.55 & 8220 & 8227.1 & 254.57 & 8222 & 8227.7 & 149.58 \\
    I189 &   &   & 8350.5 & \textbf{8352} & 8356.8 & 84.99 & 8353 & 8361.1 & 121.01 & 8359 & 8363.6 & 132.89 \\
    I190 &   &   & 7698.5 & \textbf{7704} & 7707.1 & 121.58 & 7709 & 7712 & 123.33 & 7708 & 7712.9 & 175.84 \\
    \hline
    I191 &   &   & 9315.5 & \textbf{9332} & 9365.4 & 50.05 & 9401 & 9415.8 & 172.26 & 9402 & 9421.5 & 204.80 \\
    I192 &   &   & 9120 & \textbf{9148} & 9209.8 & 191.30 & 9198 & 9221.2 & 198.00 & 9208 & 9228.6 & 167.77 \\
    I193 &   &   & 9611.5 & \textbf{9652} & 9699.1 & 124.61 & 9685 & 9707 & 209.73 & 9690 & 9723.9 & 186.50 \\
    I194 &   &   & 8972.5 & \textbf{9001} & 9055.8 & 135.90 & 9056 & 9077.8 & 193.12 & 9065 & 9081.3 & 154.83 \\
    I195 & 250 & $\alpha_2$ & 9138.5 & \textbf{9174} & 9234.3 & 168.61 & 9195 & 9227.7 & 220.72 & 9195 & 9246.8 & 186.97 \\
    I196 &   &   & 9239.5 & \textbf{9257} & 9305.2 & 48.81 & 9316 & 9338.9 & 167.96 & 9335 & 9351.7 & 103.83 \\
    I197 &   &   & 9230 & \textbf{9278} & 9320.8 & 140.50 & 9298 & 9330.8 & 146.42 & 9324 & 9341.9 & 148.61 \\
    I198 &   &   & 9247.5 & \textbf{9258} & 9307.8 & 57.58 & 9320 & 9345.2 & 172.22 & 9305 & 9347.1 & 193.85 \\
    I199 &   &   & 9178 & \textbf{9212} & 9265 & 118.21 & 9262 & 9278.6 & 200.98 & 9261 & 9283.3 & 207.56 \\
    I200 &   &   & 9162 & 9223 & 9244.9 & 173.41 & 9238 & 9260.4 & 208.11 & \textbf{9209} & 9263.4 & 163.48 \\
    \hline
    I201 &   &   & 11626.5 & \textbf{12019} & 12281.6 & 96.05 & 12266 & 12337.9 & 149.72 & 12323 & 12374.2 & 187.44 \\
    I202 &   &   & 11171.5 & \textbf{11491} & 11756.2 & 45.78 & 11888 & 11930.1 & 178.71 & 11869 & 11944.1 & 132.79 \\
    I203 &   &   & 11095 & \textbf{11417} & 11667 & 59.15 & 11702 & 11848.7 & 234.54 & 11751 & 11853.1 & 144.67 \\
    I204 &   &   & 11264.5 & 11874 & 11958.6 & 162.26 & \textbf{11869} & 11971.3 & 253.40 & 11977 & 12033.8 & 80.53 \\
    I205 & 250 & $\alpha_3$ & 11385 & \textbf{11965} & 12095.2 & 161.03 & 11996 & 12133 & 220.39 & 12105 & 12183.2 & 168.57 \\
    I206 &   &   & 11298.5 & \textbf{11640} & 11814.5 & 1.58 & 11971 & 12045.4 & 179.96 & 11972 & 12063.9 & 178.81 \\
    I207 &   &   & 11368 & \textbf{11727} & 11966.7 & 82.51 & 12030 & 12111.2 & 167.23 & 12072 & 12149.9 & 85.02 \\
    I208 &   &   & 11262.5 & \textbf{11562} & 11826.1 & 46.11 & 11919 & 12001 & 132.54 & 12018 & 12062.1 & 159.30 \\
    I209 &   &   & 11443.5 & \textbf{11962} & 12160 & 101.88 & 12182 & 12235.7 & 224.77 & 12196 & 12260.2 & 173.39 \\
    I210 &   &   & 11460 & \textbf{11922} & 12144.9 & 136.39 & 12047 & 12176.9 & 174.42 & 12147 & 12245.4 & 142.54 \\
    \hline
\end{tabular}
}
\label{tab-8}
\end{table}

\begin{table}[!h]
\caption{Comparison of GVNS~I, GVNS~II and GRASP for $n=500$.}
\centering
\vspace*{3mm}
\renewcommand{\tabcolsep}{1.5pt}
{\scriptsize
\renewcommand{\arraystretch}{0.80}
\begin{tabular}{|c|c|c|c|ccc|ccc|ccc|}
\hline
\multicolumn{4}{|c|}{Instance} & \multicolumn{3}{c|}{GVNS~I} & \multicolumn{3}{c|}{GVNS~II} & \multicolumn{3}{c|}{GRASP} \\
    \hline
    ID & $n$ & $\alpha$ & $LB_{T}$ & $C_{max}^{best}$ & $C_{max}^{avg}$ & CPU & $C_{max}^{best}$ & $C_{max}^{avg}$ & CPU & $C_{max}^{best}$ & $C_{max}^{avg}$ & CPU \\
    \hline
    I211 &   &   & 16183 & \textbf{16193} & 16221.3 & 95.35 & 16215 & 16239.5 & 219.08 & 16212 & 16231.7 & 229.66 \\
    I212 &   &   & 16168.5 & \textbf{16169} & 16190.7 & 44.93 & 16206 & 16223.2 & 269.14 & 16208 & 16224.4 & 217.81 \\
    I213 &   &   & 16052 & \textbf{16053} & 16061.5 & 0.00 & 16094 & 16110.3 & 140.95 & 16093 & 16104.6 & 163.79 \\
    I214 &   &   & 16039 & \textbf{16040} & 16059.9 & 36.41 & 16068 & 16091.9 & 215.93 & 16072 & 16084.6 & 176.23 \\
    I215 & 500 & $\alpha_1$ & 16386.5 & \textbf{16387} & 16408.9 & 49.72 & 16424 & 16450 & 240.40 & 16421 & 16436.7 & 210.35 \\
    I216 &   &   & 16549 & \textbf{16551} & 16562 & 0.00 & 16586 & 16605.4 & 214.22 & 16570 & 16599.7 & 212.29 \\
    I217 &   &   & 16123.5 & \textbf{16128} & 16154.1 & 106.97 & 16154 & 16170.5 & 280.12 & 16154 & 16171.5 & 215.64 \\
    I218 &   &   & 16501.5 & \textbf{16503} & 16514.1 & 58.74 & 16541 & 16557.5 & 226.26 & 16542 & 16550.6 & 211.71 \\
    I219 &   &   & 16476 & \textbf{16478} & 16505.6 & 167.26 & 16526 & 16541.1 & 149.62 & 16513 & 16524.7 & 220.66 \\
    I220 &   &   & 16395 & \textbf{16399} & 16436.5 & 93.41 & 16430 & 16459.6 & 279.06 & 16428 & 16444.8 & 212.82 \\
    \hline
    I221 &   &   & 18336.5 & \textbf{18391} & 18521.6 & 26.58 & 18614 & 18659.6 & 207.13 & 18591 & 18649.2 & 207.35 \\
    I222 &   &   & 18452 & \textbf{18479} & 18719.9 & 168.25 & 18737 & 18794 & 205.39 & 18714 & 18805.3 & 174.90 \\
    I223 &   &   & 18351 & \textbf{18421} & 18554.5 & 108.90 & 18605 & 18658.8 & 188.27 & 18596 & 18662.1 & 268.61 \\
    I224 &   &   & 18287.5 & \textbf{18352} & 18527.3 & 238.02 & 18541 & 18627.4 & 243.43 & 18468 & 18589.7 & 224.05 \\
    I225 & 500 & $\alpha_2$ & 18507 & \textbf{18565} & 18735.6 & 105.35 & 18737 & 18820.3 & 196.79 & 18780 & 18825 & 176.34 \\
    I226 &   &   & 17979.5 & \textbf{18037} & 18202.2 & 145.61 & 18221 & 18280.20 & 227.92 & 18224 & 18286.3 & 297.78 \\
    I227 &   &   & 18974 & \textbf{19065} & 19235.3 & 236.62 & 19211 & 19301.10 & 232.63 & 19227 & 19285.3 & 193.06 \\
    I228 &   &   & 18554.5 & \textbf{18613} & 18733.9 & 121.78 & 18798 & 18857.00 & 189.89 & 18801 & 18848.9 & 171.66 \\
    I229 &   &   & 18235 & \textbf{18273} & 18477.2 & 122.11 & 18467 & 18521.40 & 216.95 & 18462 & 18550.4 & 206.81 \\
    I230 &   &   & 18619 & \textbf{18676} & 18739.6 & 0.00 & 18891 & 18932.80 & 192.21 & 18899 & 18925.6 & 242.40 \\
    \hline
    I231 &   &   & 23731.5 & 25317 & 25496.4 & 129.52 & \textbf{25077} & 25615.40 & 253.68 & 25535 & 25721.4 & 204.73 \\
    I232 &   &   & 21490.5 & \textbf{22093} & 22885.5 & 95.73 & 23048 & 23265.20 & 245.40 & 22862 & 23160.9 & 218.52 \\
    I233 &   &   & 22338 & \textbf{23166} & 24063.5 & 172.02 & 24007 & 24169.80 & 302.08 & 23958 & 24143.5 & 278.67 \\
    I234 &   &   & 23310.5 & \textbf{24229} & 24902.1 & 160.24 & 24880 & 25142.60 & 275.73 & 24970 & 25184.4 & 260.70 \\
    I235 & 500 & $\alpha_3$ & 23104 & \textbf{23905} & 24303.2 & 13.52 & 24907 & 25162.40 & 278.16 & 25034 & 25133.4 & 219.44 \\
    I236 &   &   & 22535.5 & \textbf{23253} & 24055.7 & 166.31 & 24199 & 24376.10 & 193.90 & 24167 & 24312 & 196.42 \\
    I237 &   &   & 21615.5 & \textbf{22336} & 22953.3 & 141.24 & 23132 & 23298.40 & 174.15 & 23264 & 23450.3 & 232.60 \\
    I238 &   &   & 22629.5 & \textbf{23454} & 23969.7 & 47.90 & 24301 & 24547.60 & 240.40 & 24309 & 24447.9 & 271.19 \\
    I239 &   &   & 21475.5 & \textbf{22011} & 22522.5 & 74.06 & 22923 & 23110.90 & 209.81 & 23067 & 23236.9 & 210.89 \\
    I240 &   &   & 22855.5 & \textbf{23611} & 24153.2 & 64.92 & 24519 & 24670.40 & 256.98 & 24396 & 24781.9 & 239.34 \\
    \hline
\end{tabular}
}
\label{tab-9}
\end{table}


\begin{sidewaystable}
\caption{Comparison of $CF^{+}$, $TIF^{+}$, $MIP_{2S}^{+}$  with $MIP_{2S}$ by \cite{benmansour2021scheduling} for $n=10$.}
\centering
\vspace*{3mm}
\renewcommand{\tabcolsep}{0.85pt}
{\scriptsize
\renewcommand{\arraystretch}{0.85}
\begin{tabular}{|ccc|cccc|cccc|cccc|cccc|c|}
\hline
      & \multicolumn{1}{c}{Instance} &   & \multicolumn{8}{c|}{2 servers} & \multicolumn{8}{c|}{1 server} &  \\
\cline{4-19}      & \multicolumn{1}{c}{} &   & \multicolumn{4}{c|}{$MIP_{2S}$} & \multicolumn{4}{c|}{$MIP_{2S}^{+}$} & \multicolumn{4}{c|}{$CF^{+}$} & \multicolumn{4}{c|}{$TIF^{+}$} & $Gap_{MI}(\%)$ \\
\cline{1-19}    \multicolumn{1}{|c|}{ID} & \multicolumn{1}{c|}{$n$} & $\alpha$ & $UB_{MIP_{2S}}$ & $LB_{MIP_{2S}}$ & $Gap_{MIP_{2S}}(\%)$ & CPU & $UB_{MIP_{2S}^{+}}$ & $LB_{MIP_{2S}^{+}}$ & $Gap_{MIP_{2S}^{+}}(\%)$ & CPU & $UB_{CF^{+}}$& $LB_{CF^{+}}$ & $Gap_{CF^{+}}(\%)$ & CPU & $UB_{TIF^{+}}$ & $LB_{TIF^{+}}$ & $Gap_{TIF^{+}}(\%)$ & CPU &  \\
\hline
    \multicolumn{1}{|c|}{I31} & \multicolumn{1}{c|}{} &  & 276 & 276 & 0 & 221.90 & 276 & 276 & 0 & 0.39 & 277 & 277 & 0 & 0.68 & 277 & 277 & 0 & 4.22 & 0.36 \\
    \multicolumn{1}{|c|}{I32} & \multicolumn{1}{c|}{} &   & 241 & 232 & 3.73 & 3600 & 241 & 241 & 0 & 0.56 & 242 & 242 & 0 & 0.72 & 242 & 242 & 0 & 5.05 & 0.41 \\
    \multicolumn{1}{|c|}{I33} & \multicolumn{1}{c|}{} &   & 309 & 301 & 2.59 & 3600 & 309 & 309 & 0 & 0.26 & 310 & 310 & 0 & 9.83 & 310 & 310 & 0 & 7.94 & 0.32 \\
    \multicolumn{1}{|c|}{I34} & \multicolumn{1}{c|}{} &   & 298 & 294 & 1.34 & 3600 & 298 & 298 & 0 & 0.37 & 299 & 299 & 0 & 1.87 & 299 & 299 & 0 & 4.86 & 0.34 \\
    \multicolumn{1}{|c|}{I35} & \multicolumn{1}{c|}{10} &  $\alpha_{1}$ & 312 & 312 & 0 & 2349.40 & 312 & 312 & 0 & 0.27 & 313 & 313 & 0 & 1.15 & 313 & 313 & 0 & 6.52 & 0.32 \\
    \multicolumn{1}{|c|}{I36} & \multicolumn{1}{c|}{} &   & 379 & 368 & 2.90 & 3600 & 379 & 379 & 0 & 0.26 & 380 & 380 & 0 & 17.25 & 380 & 380 & 0 & 8.27 & 0.26 \\
    \multicolumn{1}{|c|}{I37} & \multicolumn{1}{c|}{} &   & 310 & 310 & 0 & 1554.82 & 310 & 310 & 0 & 0.31 & 311 & 311 & 0 & 1.07 & 311 & 311 & 0 & 8.64 & 0.32 \\
    \multicolumn{1}{|c|}{I38} & \multicolumn{1}{c|}{} &   & 292 & 292 & 0 & 2414.57 & 292 & 292 & 0 & 0.24 & 293 & 293 & 0 & 1.67 & 293 & 293 & 0 & 10.13 & 0.34 \\
    \multicolumn{1}{|c|}{I39} & \multicolumn{1}{c|}{} &   & 320 & 320 & 0 & 1029.12 & 320 & 320 & 0 & 0.28 & 321 & 321 & 0 & 0.68 & 321 & 321 & 0 & 9.69 & 0.31 \\
    \multicolumn{1}{|c|}{I40} & \multicolumn{1}{c|}{} &   & 192 & 179 & 6.77 & 3600 & 192 & 192 & 0 & 0.23 & 193 & 193 & 0 & 30.94 & 193 & 193 & 0 & 2.80 & 0.52 \\
        \hline
    \multicolumn{1}{|c|}{I41} & \multicolumn{1}{c|}{} &  & 414 & 414 & 0 & 641.43 & 414 & 414 & 0 & 0.22 & 416 & 416 & 0 & 44.25 & 416 & 416 & 0 & 38.63 & 0.48 \\
    \multicolumn{1}{|c|}{I42} & \multicolumn{1}{c|}{} &   & 444 & 434 & 2.25 & 3600 & 444 & 444 & 0 & 0.31 & 448 & 448 & 0 & 14.01 & 448 & 448 & 0 & 59.41 & 0.90 \\
    \multicolumn{1}{|c|}{I43} & \multicolumn{1}{c|}{} &   & 309 & 309 & 0 & 1004.36 & 309 & 309 & 0 & 0.35 & 311 & 311 & 0 & 9.02 & 311 & 311 & 0 & 8.27 & 0.65 \\
    \multicolumn{1}{|c|}{I44} & \multicolumn{1}{c|}{} &   & 226 & 226 & 0 & 81.07 & 226 & 226 & 0 & 0.31 & 227 & 227 & 0 & 0.53 & 227 & 227 & 0 & 6.37 & 0.44 \\
    \multicolumn{1}{|c|}{I45} & \multicolumn{1}{c|}{10} & $\alpha_{2}$ & 346 & 346 & 0 & 165.86 & 346 & 346 & 0 & 0.57 & 347 & 347 & 0 & 6.84 & 347 & 347 & 0 & 14.49 & 0.29 \\
    \multicolumn{1}{|c|}{I46} & \multicolumn{1}{c|}{} &   & 221 & 221 & 0.01 & 567.00 & 221 & 221 & 0 & 0.33 & 222 & 222 & 0 & 1.98 & 222 & 222 & 0 & 3.90 & 0.45 \\
    \multicolumn{1}{|c|}{I47} & \multicolumn{1}{c|}{} &   & 314 & 311 & 0.96 & 3600 & 314 & 314 & 0 & 0.38 & 317 & 317 & 0 & 298.73 & 317 & 317 & 0 & 7.43 & 0.96 \\
    \multicolumn{1}{|c|}{I48} & \multicolumn{1}{c|}{} &   & 280 & 280 & 0.01 & 712.92 & 280 & 280 & 0 & 0.33 & 281 & 281 & 0 & 1.43 & 281 & 281 & 0 & 9.47 & 0.36 \\
    \multicolumn{1}{|c|}{I49} & \multicolumn{1}{c|}{} &   & 272 & 271 & 0.37 & 3600 & 272 & 272 & 0 & 0.42 & 273 & 273 & 0 & 12.54 & 273 & 273 & 0 & 4.96 & 0.37 \\
    \multicolumn{1}{|c|}{I50} & \multicolumn{1}{c|}{} &   & 355 & 355 & 0.01 & 1935.59 & 355 & 355 & 0 & 0.41 & 357 & 357 & 0 & 11.42 & 357 & 357 & 0 & 15.36 & 0.56 \\
    \hline
    \multicolumn{1}{|c|}{I51} & \multicolumn{1}{c|}{} &  & 477 & 477 & 0 & 840.29 & 477 & 477 & 0 & 0.43 & 479 & 479 & 0 & 5.49 & 479 & 479 & 0 & 19.16 & 0.42 \\
    \multicolumn{1}{|c|}{I52} & \multicolumn{1}{c|}{} &   & 313 & 300 & 4.15 & 3600 & 313 & 313 & 0 & 0.26 & 316 & 316 & 0 & 21.08 & 316 & 316 & 0 & 16.94 & 0.96 \\
    \multicolumn{1}{|c|}{I53} & \multicolumn{1}{c|}{} &   & 515 & 514 & 0.19 & 3600 & 515 & 515 & 0 & 0.30 & 519 & 519 & 0 & 8.02 & 519 & 519 & 0 & 30.66 & 0.78 \\
    \multicolumn{1}{|c|}{I54} & \multicolumn{1}{c|}{} &   & 451 & 434 & 3.77 & 3600 & 451 & 451 & 0 & 0.25 & 458 & 458 & 0 & 17.57 & 458 & 458 & 0 & 43.69 & 1.55 \\
    \multicolumn{1}{|c|}{I55} & \multicolumn{1}{c|}{10} &  $\alpha_{3}$ & 410 & 410 & 0 & 439.66 & 410 & 410 & 0 & 0.28 & 413 & 413 & 0 & 5.64 & 413 & 413 & 0 & 33.46 & 0.73 \\
    \multicolumn{1}{|c|}{I56} & \multicolumn{1}{c|}{} &   & 345 & 338 & 2.03 & 3600 & 345 & 345 & 0 & 0.31 & 349 & 349 & 0 & 7.12 & 349 & 349 & 0 & 21.73 & 1.16 \\
    \multicolumn{1}{|c|}{I57} & \multicolumn{1}{c|}{} &   & 356 & 356 & 0.00 & 588.64 & 356 & 356 & 0 & 0.38 & 358 & 358 & 0 & 7.30 & 358 & 358 & 0 & 25.91 & 0.56 \\
    \multicolumn{1}{|c|}{I58} & \multicolumn{1}{c|}{} &   & 481 & 481 & 0 & 91.85 & 481 & 481 & 0 & 0.27 & 487 & 487 & 0 & 7.68 & 487 & 487 & 0 & 67.40 & 1.25 \\
    \multicolumn{1}{|c|}{I59} & \multicolumn{1}{c|}{} &   & 518 & 518 & 0.01 & 3264.02 & 518 & 518 & 0 & 0.41 & 523 & 523 & 0 & 13.47 & 523 & 523 & 0 & 92.01 & 0.97 \\
    \multicolumn{1}{|c|}{I60} & \multicolumn{1}{c|}{} &   & 438 & 395 & 9.82 & 3601 & 438 & 438 & 0 & 0.30 & 444 & 444 & 0 & 13.58 & 444 & 444 & 0 & 55.66 & 1.37 \\
    \hline
      & Avg. &   & 347.13 & 342.47 & 1.36 & 2156.77 & 347.13 & 347.13 & 0 & 0.33 & 349.47 & 349.47 & 0 & 19.12 & 349.47 & 349.47 & 0 & 21.43 & 0.62 \\
\hline
\end{tabular}
}
\label{tabM2}
\end{sidewaystable}

\begin{sidewaystable}
\caption{Comparison of $CF^{+}$, $TIF^{+}$, $MIP_{2S}^{+}$  with $MIP_{2S}$ by \cite{benmansour2021scheduling} for $n=12$.}
\centering
\vspace*{3mm}
\renewcommand{\tabcolsep}{0.85pt}
{\scriptsize
\renewcommand{\arraystretch}{0.85}
\begin{tabular}{|ccc|cccc|cccc|cccc|cccc|c|}
\hline
      & \multicolumn{1}{c}{Instance} &   & \multicolumn{8}{c|}{2 servers} & \multicolumn{8}{c|}{1 server} &  \\
\cline{4-19}      & \multicolumn{1}{c}{} &   & \multicolumn{4}{c|}{$MIP_{2S}$} & \multicolumn{4}{c|}{$MIP_{2S}^{+}$} & \multicolumn{4}{c|}{$CF^{+}$} & \multicolumn{4}{c|}{$TIF^{+}$} & $Gap_{MI}(\%)$ \\
\cline{1-19}    \multicolumn{1}{|c|}{ID} & \multicolumn{1}{c|}{$n$} & $\alpha$ & $UB_{MIP_{2S}}$ & $LB_{MIP_{2S}}$ & $Gap_{MIP_{2S}}(\%)$ & CPU & $UB_{MIP_{2S}^{+}}$ & $LB_{MIP_{2S}^{+}}$ & $Gap_{MIP_{2S}^{+}}(\%)$ & CPU & $UB_{CF^{+}}$& $LB_{CF^{+}}$ & $Gap_{CF^{+}}(\%)$ & CPU & $UB_{TIF^{+}}$ & $LB_{TIF^{+}}$ & $Gap_{TIF^{+}}(\%)$ & CPU &  \\
\hline
    \multicolumn{1}{|c|}{I61} & \multicolumn{1}{r|}{} &   & 402 & 318 & 20.90 & 3600 & 402 & 402 & 0 & 0.21 & 403 & 403 & 0 & 709.72 & 403 & 403 & 0 & 21.77 & 0.25 \\
    \multicolumn{1}{|c|}{I62} & \multicolumn{1}{c|}{} &   & 337 & 253.1 & 24.89 & 3600 & 337 & 337 & 0 & 0.28 & 338 & 338 & 0 & 2780.13 & 338 & 338 & 0 & 13.47 & 0.30 \\
    \multicolumn{1}{|c|}{I63} & \multicolumn{1}{c|}{} &   & 338 & 240.2 & 28.93 & 3600 & 338 & 338 & 0 & 0.40 & 339 & 339 & 0 & 1327.14 & 339 & 339 & 0 & 6.65 & 0.30 \\
    \multicolumn{1}{|c|}{I64} & \multicolumn{1}{c|}{} &   & 329 & 273 & 17.02 & 3600 & 329 & 329 & 0 & 0.40 & 330 & 329 & 0.30 & 3600 & 330 & 330 & 0 & 9.11 & 0.30 \\
    \multicolumn{1}{|c|}{I65} & \multicolumn{1}{c|}{12} & \multicolumn{1}{c|}{$\alpha_1$} & 391 & 278.7 & 28.71 & 3600 & 391 & 391 & 0 & 0.49 & 393 & 393 & 0 & 1318.60 & 393 & 393 & 0 & 17.89 & 0.51 \\
    \multicolumn{1}{|c|}{I66} & \multicolumn{1}{c|}{} &   & 219 & 135 & 38.33 & 3600 & 219 & 219 & 0 & 0.44 & 220 & 220 & 0 & 3101.09 & 220 & 220 & 0 & 4.22 & 0.46 \\
    \multicolumn{1}{|c|}{I67} & \multicolumn{1}{c|}{} &   & 351 & 325 & 7.41 & 3600 & 351 & 351 & 0 & 0.27 & 352 & 351 & 0.28 & 3600 & 352 & 352 & 0 & 7.46 & 0.28 \\
    \multicolumn{1}{|c|}{I68} & \multicolumn{1}{c|}{} &   & 382 & 274.4 & 28.16 & 3600 & 382 & 382 & 0 & 0.27 & 383 & 383 & 0 & 2033.21 & 383 & 383 & 0 & 12.57 & 0.26 \\
    \multicolumn{1}{|c|}{I69} & \multicolumn{1}{c|}{} &   & 367 & 244.8 & 33.31 & 3600 & 367 & 367 & 0 & 0.46 & 369 & 369 & 0 & 811.76 & 369 & 369 & 0 & 10.42 & 0.54 \\
    \multicolumn{1}{|c|}{I70} & \multicolumn{1}{c|}{} &   & 371 & 262 & 29.38 & 3600 & 371 & 371 & 0 & 0.35 & 372 & 372 & 0 & 1107.71 & 372 & 372 & 0 & 20.40 & 0.27 \\
    \hline
    \multicolumn{1}{|c|}{I71} & \multicolumn{1}{c|}{} &   & 472 & 324.9 & 31.17 & 3600 & 472 & 472 & 0 & 0.32 & 475 & 475 & 0 & 1933.03 & 475 & 475 & 0 & 44.77 & 0.64 \\
    \multicolumn{1}{|c|}{I72} & \multicolumn{1}{c|}{} &   & 444 & 304 & 31.53 & 3600 & 444 & 444 & 0 & 0.37 & 446 & 446 & 0 & 1062.03 & 446 & 446 & 0 & 36.42 & 0.45 \\
    \multicolumn{1}{|c|}{I73} & \multicolumn{1}{c|}{} &   & 457 & 299.5 & 34.46 & 3600 & 457 & 457 & 0 & 0.44 & 459 & 459 & 0 & 1393.99 & 459 & 459 & 0 & 47.07 & 0.44 \\
    \multicolumn{1}{|c|}{I74} & \multicolumn{1}{c|}{} &   & 379 & 271.2 & 28.44 & 3600 & 379 & 379 & 0 & 0.54 & 382 & 382 & 0 & 1687.02 & 382 & 382 & 0 & 18.44 & 0.79 \\
    \multicolumn{1}{|c|}{I75} & \multicolumn{1}{c|}{12} & \multicolumn{1}{c|}{$\alpha2$} & 351 & 284.4 & 18.97 & 3600 & 351 & 351 & 0 & 0.33 & 353 & 353 & 0 & 1490.81 & 353 & 353 & 0 & 21.96 & 0.57 \\
    \multicolumn{1}{|c|}{I76} & \multicolumn{1}{c|}{} &   & 496 & 319.5 & 35.58 & 3600 & 496 & 496 & 0 & 0.39 & 501 & 500 & 0.20 & 3600 & 501 & 501 & 0 & 64.85 & 1.01 \\
    \multicolumn{1}{|c|}{I77} & \multicolumn{1}{c|}{} &   & 362 & 300 & 17.13 & 3600 & 362 & 362 & 0 & 0.31 & 364 & 364 & 0 & 1915.16 & 364 & 364 & 0 & 14.92 & 0.55 \\
    \multicolumn{1}{|c|}{I78} & \multicolumn{1}{c|}{} &   & 543 & 327 & 39.77 & 3600 & 543 & 543 & 0 & 0.40 & 547 & 547 & 0 & 2810.76 & 547 & 547 & 0 & 69.65 & 0.74 \\
    \multicolumn{1}{|c|}{I79} & \multicolumn{1}{c|}{} &   & 584 & 369 & 36.82 & 3600 & 584 & 584 & 0 & 0.40 & 590 & 589 & 0.17 & 3600 & 590 & 590 & 0 & 105.08 & 1.03 \\
    \multicolumn{1}{|c|}{I80} & \multicolumn{1}{c|}{} &   & 400 & 332 & 17.00 & 3600 & 400 & 400 & 0 & 0.26 & 402 & 402 & 0 & 1037.85 & 402 & 402 & 0 & 24.29 & 0.50 \\
    \hline
    \multicolumn{1}{|c|}{I81} & \multicolumn{1}{c|}{} &   & 512 & 306.2 & 40.19 & 3600 & 512 & 512 & 0 & 0.42 & 519 & 519 & 0 & 1848.72 & 519 & 519 & 0 & 150.53 & 1.37 \\
    \multicolumn{1}{|c|}{I82} & \multicolumn{1}{c|}{} &   & 508 & 398.8 & 21.50 & 3600 & 508 & 508 & 0 & 0.56 & 512 & 512 & 0 & 459.09 & 512 & 512 & 0 & 196.29 & 0.79 \\
    \multicolumn{1}{|c|}{I83} & \multicolumn{1}{c|}{} &   & 587 & 414 & 29.47 & 3600 & 587 & 587 & 0 & 0.30 & 590 & 589.9 & 0 & 1538.58 & 590 & 590 & 0 & 171.30 & 0.51 \\
    \multicolumn{1}{|c|}{I84} & \multicolumn{1}{c|}{} &   & 680 & 395 & 41.91 & 3600 & 680 & 680 & 0 & 0.42 & 686 & 686 & 0 & 1020.65 & 686 & 686 & 0 & 213.73 & 0.88 \\
    \multicolumn{1}{|c|}{I85} & \multicolumn{1}{c|}{12} & \multicolumn{1}{c|}{$\alpha_3$} & 483 & 399.5 & 17.29 & 3600 & 483 & 483 & 0 & 0.28 & 486 & 486 & 0 & 219.78 & 486 & 486 & 0 & 99.75 & 0.62 \\
    \multicolumn{1}{|c|}{I86} & \multicolumn{1}{c|}{} &   & 566 & 388.3 & 31.40 & 3600 & 566 & 566 & 0 & 0.35 & 570 & 569 & 0.18 & 3600 & 570 & 570 & 0 & 116.83 & 0.71 \\
    \multicolumn{1}{|c|}{I87} & \multicolumn{1}{c|}{} &   & 625 & 397.5 & 36.39 & 3600 & 625 & 625 & 0 & 0.40 & 632 & 626 & 0.95 & 3600 & 632 & 632 & 0 & 207.81 & 1.12 \\
    \multicolumn{1}{|c|}{I88} & \multicolumn{1}{c|}{} &   & 527 & 372.8 & 29.26 & 3600 & 527 & 527 & 0 & 0.28 & 533 & 528 & 0.94 & 3600 & 533 & 533 & 0 & 148.40 & 1.14 \\
    \multicolumn{1}{|c|}{I89} & \multicolumn{1}{c|}{} &   & 429 & 270.2 & 37.03 & 3600 & 429 & 429 & 0 & 0.40 & 432 & 432 & 0 & 949.44 & 432 & 432 & 0 & 46.51 & 0.70 \\
    \multicolumn{1}{|c|}{I90} & \multicolumn{1}{c|}{} &   & 491 & 395 & 19.55 & 3600 & 491 & 491 & 0 & 0.30 & 493 & 493 & 0 & 217.03 & 493 & 493 & 0 & 147.99 & 0.41 \\
    \hline
      & Avg. &   & 446.1 & 315.77 & 28.40 & 3600 & 446.1 & 446.1 & 0.00 & 0.37 & 449.03 & 448.50 & 0.10 & 1932.47 & 449.03 & 449.03 & 0 & 69.02 & 0.61 \\
\hline
\end{tabular}
}
\label{tabM3}
\end{sidewaystable}

\begin{sidewaystable}
\caption{Comparison of $CF^{+}$, $TIF^{+}$, $MIP_{2S}^{+}$  with $MIP_{2S}$ by \cite{benmansour2021scheduling} for $n=25$.}
\centering
\vspace*{3mm}
\renewcommand{\tabcolsep}{0.85pt}
{\scriptsize
\renewcommand{\arraystretch}{0.85}
\begin{tabular}{|c|c|c|cccc|cccc|cccc|cccc|}
\hline
\multicolumn{1}{|r}{} & \multicolumn{1}{c}{Instance} &   & \multicolumn{8}{c|}{2 servers} & \multicolumn{8}{c|}{1 server} \\
\cline{4-19}    \multicolumn{1}{|r}{} & \multicolumn{1}{r}{} &   & \multicolumn{4}{c|}{$MIP_{2S}$} & \multicolumn{4}{c|}{$MIP_{2S}^{+}$} & \multicolumn{4}{c|}{$CF^{+}$} & \multicolumn{4}{c|}{$TIF^{+}$} \\
\hline
    ID & $n$ & $\alpha$ & $UB_{MIP_{2S}}$ & $LB_{MIP_{2S}}$ & $Gap_{MIP_{2S}}(\%)$ & CPU & $UB_{MIP_{2S}^{+}}$ & $LB_{MIP_{2S}^{+}}$ & $Gap_{MIP_{2S}^{+}}(\%)$ & CPU & $UB_{CF^{+}}$ & $LB_{CF^{+}}$ & $Gap_{CF^{+}}(\%)$ & CPU & $UB_{TIF^{+}}$ & $LB_{TIF^{+}}$ & $Gap_{TIF^{+}}(\%)$ & CPU \\
    \hline
    I91 &   &   & 800 & 246 & 69.25 & 3600 & 800 & 800 & 0 & 277.00 & 801 & 799.5 & 0.19 & 3600 & 801 & 549 & 31.47 & 3600 \\
    I92 &   &   & 762 & 309 & 59.45 & 3600 & 762 & 762 & 0 & 0.68 & 763 & 762 & 0.13 & 3600 & 763 & 763 & 0 & 1782.26 \\
    I93 &   &   & 860 & 303.5 & 64.71 & 3600 & 860 & 860 & 0 & 253.33 & 861 & 859.5 & 0.17 & 3600 & 861 & 861 & 0 & 1454.95 \\
    I94 &   &   & 880 & 300 & 65.91 & 3600 & 880 & 880 & 0 & 1.01 & 882 & 880 & 0.23 & 3600 & 881 & 677 & 23.16 & 3600 \\
    I95 & 25 & $\alpha_1$ & 815 & 280.5 & 65.58 & 3600 & 815 & 815 & 0 & 0.86 & 816 & 815 & 0.12 & 3600 & 816 & 816 & 0 & 970.43 \\
    I96 &   &   & 704 & 240.6 & 65.83 & 3600 & 704 & 704 & 0 & 0.49 & 705 & 704 & 0.14 & 3600 & 705 & 545 & 22.70 & 3600 \\
    I97 &   &   & 710 & 275 & 61.27 & 3600 & 710 & 710 & 0 & 1.71 & 711 & 710 & 0.14 & 3600 & 711 & 711 & 0 & 1189.41 \\
    I98 &   &   & 899 & 275 & 69.41 & 3600 & 899 & 899 & 0 & 1193.36 & 900 & 898.5 & 0.17 & 3600 & 901 & 706.7 & 21.56 & 3600 \\
    I99 &   &   & 789 & 294 & 62.74 & 3600 & 789 & 789 & 0 & 10.56 & 790 & 789 & 0.13 & 3600 & 790 & 790 & 0 & 3025.83 \\
    I100 &   &   & 741 & 287.7 & 61.18 & 3600 & 741 & 741 & 0 & 0.52 & 742 & 741 & 0.13 & 3600 & 742 & 742 & 0 & 1884.51 \\
    \hline
    I101 &   &   & * & * & * & * & * & * & * & * & 1080 & 1077 & 0.28 & 3600 & 1085 & 719.6 & 33.68 & 3600 \\
    I102 &   &   & 996 & 349 & 64.96 & 3600 & 996 & 996 & 0 & 0.92 & 998 & 996 & 0.20 & 3600 & 1008 & 664.9 & 34.04 & 3600 \\
    I103 &   &   & 767 & 315 & 58.93 & 3600 & 767 & 767 & 0 & 1244.90 & 768 & 766.5 & 0.20 & 3600 & 768 & 768 & 0 & 2456.37 \\
    I104 &   &   & 830 & 317.9 & 61.70 & 3600 & 830 & 830 & 0 & 686.88 & 832 & 829.5 & 0.30 & 3600 & 831 & 586.3 & 29.45 & 3600 \\
    I105 & 25 & $\alpha_2$ & 853 & 336 & 60.61 & 3600 & 853 & 853 & 0 & 0.69 & 855 & 853 & 0.23 & 3600 & 859 & 573.1 & 33.28 & 3600 \\
    I106 &   &   & 973 & 325 & 66.60 & 3600 & 973 & 973 & 0 & 425.89 & 975 & 972.5 & 0.26 & 3600 & 975 & 650.8 & 33.25 & 3600 \\
    I107 &   &   & 887 & 373 & 57.95 & 3600 & 887 & 887 & 0 & 1.01 & 890 & 887 & 0.34 & 3600 & 891 & 595 & 33.22 & 3600 \\
    I108 &   &   & 737 & 302 & 59.02 & 3600 & 737 & 737 & 0 & 550.15 & 739 & 736.5 & 0.34 & 3600 & 738 & 558 & 24.39 & 3600 \\
    I109 &   &   & 977 & 363 & 62.85 & 3600 & 977 & 977 & 0 & 1038.33 & 978 & 976.5 & 0.15 & 3600 & 988 & 632.6 & 35.97 & 3600 \\
    I110 &   &   & 755 & 281 & 62.78 & 3600 & 755 & 755 & 0 & 3088.83 & 757 & 754.5 & 0.33 & 3600 & 757 & 576.9 & 23.79 & 3600 \\
    \hline
    I111 &   &   & * & * & * & * & * & * & * & * & 1028 & 1021 & 0.68 & 3600 & 1052 & 648.7 & 38.34 & 3600 \\
    I112 &   &   & * & * & * & * & * & * & * & * & 1227 & 1213 & 1.14 & 3600 & 1250 & 666.1 & 46.71 & 3600 \\
    I113 &   &   & * & * & * & * & * & * & * & * & 1123 & 1105.5 & 1.56 & 3600 & 1156 & 718.4 & 37.85 & 3600 \\
    I114 &   &   & * & * & * & * & * & * & * & * & 1187 & 1179.5 & 0.63 & 3600 & 1255 & 681.8 & 45.68 & 3600 \\
    I115 & 25 & $\alpha_3$ & * & * & * & * & * & * & * & * & 1229 & 1214.5 & 1.18 & 3600 & 1241 & 794.5 & 35.98 & 3600 \\
    I116 &   &   & * & * & * & * & * & * & * & * & 1123 & 1110.5 & 1.11 & 3600 & 1141 & 745.4 & 34.67 & 3600 \\
    I117 &   &   & * & * & * & * & * & * & * & * & 1054 & 1043 & 1.04 & 3600 & 1074 & 659.3 & 38.62 & 3600 \\
    I118 &   &   & 941 & 376 & 60.042508 & 3600 & 941 & 941 & 0 & 0.65 & 944 & 941 & 0.32 & 3600 & 953 & 634.8 & 33.39 & 3600 \\
    I119 &   &   & * & * & * & * & * & * & * & * & 1238 & 1227 & 0.89 & 3600 & 1237 & 780.9 & 36.87 & 3600 \\
    I120 &   &   & * & * & * & * & * & * & * & * & 1138 & 1132 & 0.53 & 3600 & 1159 & 758.9 & 34.52 & 3600 \\
    \hline
\end{tabular}
}
\label{tabM4}
\end{sidewaystable}


\end{document}